\documentclass[a4paper,11pt]{article}

\usepackage{amsmath,amsthm}
\usepackage{amssymb}
\usepackage{breqn}
\usepackage{enumerate}
\usepackage{graphicx}
\usepackage{bm}
\usepackage{bbm}
\usepackage[affil-it]{authblk}
\usepackage{tabu}
\usepackage{bold-extra}
\usepackage[hang,flushmargin]{footmisc}
\usepackage[hypertex]{hyperref}
\usepackage{mathtools}
\usepackage{relsize}
\usepackage{scalerel}
\usepackage{xcolor}
\usepackage{titlesec}
\usepackage{apptools}
\usepackage{appendix}

\newtheorem{theorem}{Theorem}

\newtheorem{lemma}[theorem]{Lemma}
\newtheorem{proposition}[theorem]{Proposition}
\newtheorem{question}[theorem]{Question}

\newtheorem{claim}[theorem]{Claim}

\theoremstyle{definition}
\newtheorem{defin}[theorem]{Definition}

\titleformat{\section}[hang]{\scshape\large\bfseries\filcenter}{\S\thesection}{4pt}{}
\titleformat{\subsection}[hang]{\scshape\bfseries}{\thesubsection.}{4pt}{}

\allowdisplaybreaks

\newcommand\id{\mathbbm{1}}	
\newcommand{\tss}[1]{\textsuperscript{#1}}
\newcommand{\on}[1]{
	\operatorname{#1}
}

\newcommand{\tightoverset}[2]{
  \mathop{#2}\limits^{\vbox to -.5ex{\kern-1.15ex\hbox{$#1$}\vss}}}

\newcommand\blfootnote[1]{%
  \begingroup
  \renewcommand\thefootnote{}\footnote{#1}%
  \addtocounter{footnote}{-1}%
  \endgroup
}

\renewenvironment{thebibliography}[1]
{
  \begin{oldthebibliography}{#1}
    \setlength{\itemsep}{0em  plus 0.3ex}
    \setlength{\parskip}{0em}
}
{
  \end{oldthebibliography}
}

\newcommand\ssk[1]{
	\substack{#1}
}

\newcommand\ex{\mathop{\mathbb{E}}}

\newcommand{\exx}{
  \mathop{
    \mathchoice{\vcenter{\hbox{\larger[4]$\mathbb{E}$}}}
               {\kern0pt\mathbb{E}}
               {\kern0pt\mathbb{E}}
               {\kern0pt\mathbb{E}}
  }\displaylimits
}

\makeatletter
\newcommand*\bcdot{\mathpalette\bigcdot@{0.5}}
\newcommand*\bigcdot@[2]{\mathbin{\vcenter{\hbox{\scalebox{#2}{$\m@th#1\bullet$}}}}}
\makeatother

\makeatletter
\def\blfootnote{\gdef\@thefnmark{}\@footnotetext}
\makeatother

\newcommand{\blc}{\bm{\mathsf{C}}}

\begin{document}

\begin{center}\Large\noindent{\bfseries{\scshape Quasipolynomial inverse theorem for the $\mathsf{U}^4(\mathbb{F}_p^n)$ norm}}\\[24pt]\normalsize\noindent{\scshape Luka Mili\'cevi\'c\dag}
\end{center}
\blfootnote{\noindent\dag\ Mathematical Institute of the Serbian Academy of Sciences and Arts\\\phantom{\dag\ }Email: luka.milicevic@turing.mi.sanu.ac.rs\\
\phantom{\dag\ }2020 Mathematics Subject Classification. Primary 11P70; Secondary 05C25, 05E16}

\hfill \textit{For Teodora,}\hspace{1.5cm}\null\\
\null\hfill \textit{for her love and encouragement}\\[6pt]

\footnotesize
\begin{changemargin}{1in}{1in}
\centerline{\sc{\textbf{Abstract}}}
\phantom{a}\hspace{12pt}~The inverse theory for Gowers uniformity norms is one of the central topics in additive combinatorics and one of the most important aspects of the theory is the question of bounds. In this paper, we prove a quasipolynomial inverse theorem for the $\mathsf{U}^4$ norm in finite vector spaces. The proof follows a different strategy compared to the existing quantitative inverse theorems. In particular, the argument relies on a novel argument, which we call the abstract Balog-Szemer\'edi-Gowers theorem, and combines several other ingredients such as algebraic regularity method, bilinear Bogolyubov argument and algebraic dependent random choice.
\end{changemargin}
\normalsize

\section{Introduction}

\hspace{17pt}This paper concerns the theory of the (Gowers) uniformity norms in finite vector spaces $\mathbb{F}_p^n$, where $p$ is a fixed prime and $n$ is large. Let us begin by recalling the definition of the uniformity norms.

\begin{defin}
    Let $f : G \to \mathbb{C}$ be a function on a finite abelian group $G$. The \textit{discrete multiplicative derivative} with \textit{shift} $a$ is the operator that maps $f$ to the function $\partial_a f$, given by the formula $\partial_a f(x) = f(x + a)\overline{f(x)}$. With this notation, the uniformity norm $\|f\|_{\mathsf{U}^k}$ is defined as 
    \[\Big(|G|^{-k - 1}\sum_{x, a_1, \dots, a_k \in G} \partial_{a_1} \dots \partial_{a_k} f(x)\Big)^{2^{-k}}.\]
\end{defin}

Although it is not immediately clear from the definition, it is a well-known fact that $\|\cdot\|_{\mathsf{U}^k}$ is a norm for $k \geq 2$.\\

The uniformity norms were introduced by Gowers in his groundbreaking work~\cite{GowU4, GowerskAP}, where he proved a quantitative version of Szemer\'edi's theorem on arithmetic progressions. These norms measure the amount of algebraic structure that a function on an abelian group possesses. As observed by Gowers, in the case of Szemer\'edi's theorem, if one considers a set $A \subseteq G$ of density $\frac{|A|}{|G|} = \delta$, then small enough value of the norm $\|\id_A - \delta\|_{\mathsf{U}^{k}}$ (here, $\id_A - \delta$ is the \textit{balanced function} of the set $A$) guarantees that $A$ contains roughly the same number of arithmetic progressions of length $k + 1$ as a randomly chosen subset of $G$ of density $\delta$, namely $\delta^{k + 1}|G|^2$. Thus, to prove Szemer\'edi's theorem, one needs to obtain a sufficient understanding of functions $f : G \to \mathbb{D} = \{z \in \mathbb{C} : |z| \leq 1\}$ whose uniformity norm is large. This question became known as the \textit{inverse problem for uniformity norms}, and has been a central question of additive combinatorics and led to development of higher order Fourier analysis, which is now a field of study in its own right.\\

To complete the proof of Szemer\'edi's theorem, Gowers provided a partial answer to the inverse problem. Namely, in the case of the cyclic group $\mathbb{Z}/N\mathbb{Z}$, he showed that, for any function $f : \mathbb{Z}/N\mathbb{Z} \to \mathbb{D}$ with $\|f\|_{\mathsf{U}^k} \geq c$, one can partition $\mathbb{Z}/N\mathbb{Z}$ into arithmetic progressions $P_1, \dots, P_m$ of lengths at least $N^{\delta_{k,c}}$, where $\delta_{k,c} \in (0,1)$ is a constant depending on $k$ and $c$ only, and find polynomials $p_1, \dots, p_m$ of degree $k - 1$, such that $f$ locally correlates with phases of these polynomials, namely
\[\sum_{i \in [m]}\Big|\sum_{x \in P_i} f(x) \on{e}(p_i(x))\Big| \geq \Omega_c(N).\]
This description is incomplete in the sense that not all such piece-wise phase polynomials have large uniformity norm.\\

The prototypical answer that one looks for is that $f$ correlates \textit{globally} with a function whose algebraic structure is rich, and whose description allows reasonably efficient exploitation of this structure, namely $|\ex_{x \in G} f(x) \overline{g(x)}| \geq c'$, where $g$ typically has a polynomial-like behaviour. Here $c'$ is a parameter that depends on $c = \|f\|_{\mathsf{U}^k}$, that we shall refer to as the \textit{correlation bound} in the rest of the introduction, and $\ex_{x \in G}$ is the standard shorthand for the average $\frac{1}{|G|}\sum_{x \in G}$. Thus, an important part of the answer is to describe a family of functions from which we may pick $g$ and we refer to such function as \textit{obstructions to uniformity}. The answers to the inverse question typically depend on the ambient group.\\
\indent When it comes to $\mathsf{U}^3(G)$ norm, the inverse theorem was proved by Green and Tao~\cite{GreenTaoU3}, in the case of groups $G$ of odd order, and by Samorodnitsky~\cite{SamorU3}. More recently, Jamneshan and Tao~\cite{JamTao} provided a unified theory by giving a proof that works in all finite abelian groups.\\
\indent For $k \geq 4$, the inverse theory for uniformity norms $\mathsf{U}^k$ is considerably harder. For instance, unlike the $\mathsf{U}^3$ case, we still have no inverse theorems that work in general abelian groups, even in a qualitative sense, and instead the known results treat separately finite vector spaces and cyclic groups, (or, slightly more generally, low rank groups and bounded torsion groups). For general $k$, the inverse theorem for $\mathsf{U}^k(\mathbb{F}_p^n)$ norm, in the so-called case of the high characteristic where $k \leq p$, where polynomial phases can be taken to be the obstruction functions, is a remarkable result of Bergelson, Tao and Ziegler~\cite{BTZ, TaoZieglerCorr}. That result was later extended by Tao and Ziegler in a further major work~\cite{TaoZiegler}, to also include the low characteristic case, when $k > p$. Unlike the high characteristic case, in the latter situation one needs to consider a more general family of obstructions, namely the \textit{non-classical polynomials}. Such functions are natural objects to consider as they arise precisely as the solution of the extremal problem of finding $g : \mathbb{F}_p^n \to \mathbb{D}$ such that $\|g\|_{\mathsf{U}^k} = 1$. When it comes to cyclic groups, the inverse theorem for the $\mathsf{U}^k(\mathbb{Z}/N\mathbb{Z})$ norm is an outstanding result of Green, Tao and Ziegler~\cite{GTZU4, GTZ}, in which case one can take nilsequences as the family of obstructions, which we will not define here. Their work was an important part of the celebrated programme of Green and Tao~\cite{GTprimes1, GTprimes2}, devoted to obtaining asymptotic estimates for the counts of linear configurations in primes. The dichotomy between finite vector spaces and cyclic groups is present in all further works that we shall mention.\\

\indent The above-mentioned results gave only qualitative bounds on the correlation with the obstruction functions, meaning that for the given value of norm $c > 0$, we only know that there exists a bound $c' > 0$. In the case of finite vector spaces, the initial work Bergelson, Tao and Ziegler was ergodic-theoretic and thus cannot be made quantitative. Moreover, there are further additional difficulties, for example Tao and Ziegler make use of multidimensional Szemer\'edi theorem. On the other hand, Green, Tao and Ziegler rely on certain regularization procedures and remark that their proof would give Ackermannian bounds, even if the use of ultrafilters was avoided\footnote{Leng, Sah and Sawhney~\cite{LengSahSawhney} actually state that they believe the proof essentially already gives quantitative bounds.}. Another approach to the inverse question is via the theory of nilspaces, originating in papers by Szegedy~\cite{Szeg} and Camarena and Szegedy~\cite{CamSzeg}. This theory is a rich subject, with further developments by Candela~\cite{CandelaNotes1, CandelaNotes2}, Candela and Szegedy~\cite{CandelaSzegedy1, CandelaSzegedy2} and Gutman, Manners and Varj\'u~\cite{GMV1, GMV2, GMV3}. Notably, using this theory, Candela, Gonz\'alez-S\'anchez and Szegedy~\cite{nilspacesCharp} gave a new proof the Tao-Ziegler inverse theorem. However, the nilspace theory approach is decisively infinitary as well and therefore only qualitative.\\

\indent We already mentioned some unusual aspects of the proofs of the inverse theorem in $\mathbb{F}_p^n$, namely the initial reliance on ergodic theory, hence not having a completely combinatorial proof and the use of multidimensional Szemer\'edi theorem. Apart from this, and the fact that the low characteristic presents significant challenges in their own right (see discussion in~\cite{LukaU56}), one may argue that the case of finite vector spaces has some intrinsic difficulties compared to the case of cyclic groups~\cite{WolfSurvey} in the inverse question for $\mathsf{U}^k$ norms for $k \geq 4$. This is in contrast with the widely held perspective that the finite vector spaces are an environment in which theorems in additive combinatorics should have cleaner proofs, due to abundance of exact algebraic structures, such as subspaces, which allows one to avoid many technical difficulties that arise when working with integers. However, it turns out that, in the case of the inverse question for uniformity norms, the property of being of low rank has arguably some advantages in comparison with having a finite vector space as the ambient group. We shall return to this discussion in the outline of our proof.\\

\indent \textbf{Quantitative aspects.} Given many applications to arithmetic counting problems that uniformity norms have, it is of great importance to improve the bounds in the inverse theorems. For example, Green~\cite{Green100} mentions the question of bounds as one the biggest problems in additive combinatorics.\\
\indent In that direction, it should be noted that the inverse theorems for $\mathsf{U}^3(G)$ norms of Green and Tao, and of Samorodnitsky, were already quantitative, essentially given by a single exponential (more precisely, $c' \geq \exp(-c^{-O(1)})$). This bound was improved to quasipolynomial (bounds of the shape $c' \geq \exp(-\log^{O(1)}(2c^{-1}))$) in a remarkable work of Sanders~\cite{Sanders}, with an important contribution of Croot and Sisask~\cite{CrootSisaskPaper}. Finally, a recent striking result~\cite{Marton1, Marton2} of Gowers, Green, Manners and Tao implies the polynomial bounds ($c' \geq c^{O(1)}$) in the setting of finite vector spaces.\\
\indent We now turn to the discussion of higher norms. When it comes to finite vector spaces, first quantitative bounds were proved by Gowers and the author~\cite{U4paper} for the case of $\mathsf{U}^4$ for $p \geq 5$. These bounds, which were of the form $c' \geq \exp^{(3)}(\log^{O(1)}(2c^{-1}))$, so roughly doubly exponential, were improved by Kim, Li and Tidor~\cite{KimLiTidor}, and independently Lovett, who observed that a step proof from~\cite{U4paper}, where an exponential dependence is introduced through inefficient use of the Inclusion-Exclusion principle, can be improved. Gowers and the author~\cite{FMulti} also proved a quantitative version of the inverse theorem for $\mathsf{U}^k$ norms in the high characteristic, where the bound involved bounded number of exponentials, depending on $k$ only. There is also progress in the low characteristic, where, building upon works~\cite{U4paper, FMulti}, Tidor~\cite{Tidor} proved a quantitative inverse theorem for $\mathsf{U}^4$ norm, and the author~\cite{LukaU56} proved a corresponding result for norms $\mathsf{U}^5$ and $\mathsf{U}^6$.\\
\indent Finally, in the case of cyclic groups, according to Leng, Sah and Sawhney~\cite{LengSahSawhney}, it might be the case that with more care, the work of Green, Tao and Ziegler already gives quantitative bounds. Furthermore, we have breakthroughs of Manners, who proved doubly exponential bounds in the inverse theorem~\cite{MannersUk} for all $k$, of Leng~\cite{LengNil2}, who proved quasipolynomial bounds in the case of $\mathsf{U}^4$ norm, and a recent one of Leng, Sah and Sawhney~\cite{LengSahSawhney}, who proved quasipolynomial bounds for all orders. We shall discuss their strategies slightly later, but let us mention for now that Leng, Sah and Sawhney build upon the work of Green, Tao and Ziegler.\\

In the present paper, we prove the quasipolynomial inverse theorem for the $\mathsf{U}^4$ norm in finite vector spaces.

\begin{theorem}[Inverse theorem for $\mathsf{U}^4$ norm in $\mathbb{F}_p^n$]\label{u4invmain}
    Suppose that $f : \mathbb{F}_p^n \to \mathbb{D}$ is a function such that $\|f\|_{\mathsf{U}^4} \geq c$. Then there exists a non-classical cubic polynomial $q : \mathbb{F}_p^n \to \mathbb{T}$ such that
    \[\Big|\exx_{x \in \mathbb{F}_p^n} f(x) \on{e}(q(x))\Big| \geq \exp(-\log^{O(1)}(c^{-1})).\]
\end{theorem}

\noindent\textbf{Remark.} Throughout the paper, we make an abuse of notation and write $\log(x)$ instead of $\log(10x)$. This convention allows for neater expressions involving logs. Also, there is some dependence on $p$ in the correlation bound, but, since $p$ is fixed, we opt not to make it explicit.\\ 

\indent \textbf{Comparison with previous works.} The quantitative proofs of the inverse theorems for norms $\mathsf{U}^k$ for $k \geq 4$ can be split into two groups, depending on their starting point. Going back to the definition of the norms, for a function $f : G \to \mathbb{C}$ we have
\[\|f\|_{\mathsf{U}^k}^{2^k} = \exx_{x, a_1, \dots, a_k} \partial_{a_1} \dots \partial_{a_k} f(x) = \exx_{a_1, \dots, a_{k-2}} \|\partial_{a_1} \dots \partial_{a_{k-2}}f \|_{\mathsf{U}^2}^4.\]

Recalling the inverse theorem for the $\mathsf{U}^2$ norm, this means that there is a dense collection of parameters $(a_1, \dots, a_{k-2}) \in G^{k-2}$ for which there exists a large Fourier coefficient $\phi(a_1, \dots, a_{k-2}) \in \hat{G}$, i.e., 
\[\Big|\exx_{x \in G} \partial_{a_1} \dots \partial_{a_{k-2}}f(x)\,\phi(a_1, \dots, a_{k-2})(x) \Big| \geq c^{O(1)}.\]
An argument due to Gowers~\cite{GowerskAP} shows that such a map necessarily respects many additive quadruples in the principal directions. With more work, one can show that $\phi$ is in a certain sense an \textit{approximate multilinear map} or an \textit{approximate polynomial}. The approaches of Manners~\cite{MannersUk} and Gowers and the author~\cite{U4paper, FMulti} are to prove structural results for such objects and use them to deduce the inverse theorems for uniformity norms. It should be noted, however, that these works have different details due to different ambient groups, for example, paper~\cite{FMulti} makes use of the partition vs. analytic rank problem~\cite{BhowLov, GreenTaoPolys, Janzer2, LukaRank, MoshZhu}, while~\cite{MannersUk} relies on the fact that cyclic groups have low rank and no small subgroups.\\

On the other hand, stemming from the work of Green, Tao and Ziegler~\cite{GTZ}, one can relate the $\mathsf{U}^k$ norm to the values of $\mathsf{U}^{k-1}$ of related functions

\[\|f\|_{\mathsf{U}^k}^{2^k} = \exx_{x, a_1, \dots, a_k} \partial_{a_1} \dots \partial_{a_k} f(x) = \exx_{a} \|\partial_{a}f\|_{\mathsf{U}^{k-1}}^{2^{k-1}}.\]

Using the inverse theorem for lower order norm as inductive hypothesis yields a dense set $A \subseteq G$ and collection of obstruction functions $(\phi_a)_{a \in A}$, from $G$ to $\mathbb{C}$, such that for a positive proportion of additive quadruples $(a_1, a_2, a_3, a_4) \in A^4$ (meaning that $a_1 - a_2 + a_3 - a_4 = 0$) the map
\begin{equation}
\phi_{a_1}\overline{\phi_{a_2}}\phi_{a_3}\overline{\phi_{a_4}}\text{ is not equidistributed.}\label{noneqcondGTZ}
\end{equation}
We may think of this structure as an \textit{approximately linear system of obstructions}. A natural additive-combinatorial question is to try to relate such a system to another system of obstructions $\phi'_a$, which depends \textit{linearly} on $a$. Concretely, in the case of finite vector spaces, one may pose the following question.

\begin{question}\label{linlinsysqn}
    Suppose that $A \subseteq G= \mathbb{F}_p^n$ is a set and suppose that $\phi_a : G \to G$ is a linear map for each $a \in A$. Assume that there are at least $c|G|^3$ additive quadruples $(a_1, a_2, a_3, a_4) \in A$ such that
    \[\on{rank}\Big(\phi_{a_1} - \phi_{a_2} + \phi_{a_3} - \phi_{a_4}\Big) \leq r.\]
    Is there a map $\Phi : G \times G \to G$, affine in the first coordinate and linear in the second, such that, for $\Omega_c(|G|)$ elements $a \in A$ we have $\on{rank}(\phi_a - \Phi(a, \bcdot)) \leq O_{c,r}(1)$, where $\Phi(a, \bcdot)$ stands for the linear map sending $y \in G$ to $\Phi(a,y)$?
\end{question}

In fact, a slight variant of this question is known to have positive answer due to Kazhdan and Ziegler, but their work actually uses $\mathsf{U}^4$ inverse theorem as a crucial ingredient. However, they note that an independent answer to Question~\ref{linlinsysqn} could lead to a new proof of the inverse theorem.\\

\indent The strategy of Green, Tao and Ziegler~\cite{GTZ}, and thus of Leng, Sah and Sawhney~\cite{LengSahSawhney}, for obtaining linear structure from an approximately linear system of obstructions relies on the specific structure of obstruction functions on the cyclic groups. Namely, in that case, owing to the bounded rank of the group, an obstruction function can essentially be reduced to information on a bounded number of coefficients, and the condition~\eqref{noneqcondGTZ} can be used to find an approximate homomorphism, for which we have structural theorems, which themselves are among the classical results of additive combinatorics. Of course, to get to the point where one obtains such an approximate homomorphism, significant work and new ideas, including the steps Green, Tao and Ziegler refer to as the sunflower decomposition and thorough understanding of nilsequences, were required. Leng, Sah and Sawhney introduce further important ideas and improvements to these steps, and Leng~\cite{LengNil1, LengNil2} previously provided a key ingredient, which is an improved equidistribution theory for nilsequences.\\
\indent In particular, the arguments of Green, Tao, and Ziegler (and thus also of Leng, Sah, and Sawhney) rely crucially on the fact that there are a very small number of obstructions to the uniformity norms. In our case, however, there are many obstruction functions. To illustrate this, we will need the language of entropy. In a blog post, Tao~\cite{TaoBlogEntropy} defines $(\eta, \varepsilon)$\textit{-entropy}, for parameters $\eta > 0$ and $\varepsilon > 0$, as the least size of a family $\mathcal{F}$ of obstructions for the given norm $\mathsf{U}^k(G)$, such that the inverse theorem for functions $\|f\|_{\mathsf{U}^k(G)} \geq \eta$ holds with correlation bound $\varepsilon$. Hence, the smaller the entropy of the inverse
theorem is, the more information it gives, and hence the stronger it is. (Of course, all the inverse theorems above are sufficiently strong to give proofs of Szemer\'edi's theorem, and this notion of strength does not affect applications in general.) With this measure of strength, it turns out that entropy of inverse theorems in cyclic group $\mathbb{Z}/N\mathbb{Z}$ is a polynomial in $N$, namely $N^{O_{\eta, \varepsilon, k}(1)}$, but for finite vector spaces it grows superpolynomially in the size of the group. Hence, for $k \geq 3$, inverse theorems for $U^k(\mathbb{Z}/N\mathbb{Z})$ give stronger conclusions than those for $U^k(\mathbb{F}_p^n)$. This is relevant to the inverse question itself, as the proofs of inverse theorems typically have inductive structure, so that means that once $k$ is at least $4$, the inductive hypothesis in the case of finite vector spaces becomes weaker than that in the case of cyclic groups. Of course, in the same sense, the conclusion is also weaker, but in this case, it seems that weaker inductive hypothesis has more effect.\\

Having discussed the previous results, we are now in the position to discuss our work in this paper. Here, we take a different viewpoint, and we start from an approximately linear system of obstructions, which is the Green-Tao-Ziegler strategy, but gradually move towards the study of approximate bilinear maps. In particular, we address Question~\ref{linlinsysqn} without using the $U^4$ inverse theorem. We phrase the key theorem in the spirit of~\cite{U4paper}, where such a result appears implicitly (and with worse bounds). Recall that a map $\phi : A \to H$ is a \textit{Freiman homomorphism} if it \textit{respects all additive quadruples} in $A$, namely, whenever $a + b = c + d$ holds for some $a, b, c, d \in A$, then $\phi(a) + \phi(b) = \phi(c) + \phi(d)$.

\begin{theorem}[Structure theorem for Freiman bihomomorphisms]\label{fbihommain}
    Let $G$ and $H$ be finite-dimensional vector spaces over $\mathbb{F}_p$. Let $A \subseteq G \times G$ be a set of density $c$, and let $\varphi : A \to H$ be a \textit{Freiman bihomomorphism}, meaning that for each $x \in G$, the map $y \mapsto \varphi(x, y)$ is a Freiman homomorphism on the column $\{y \in G: (x,y) \in A\}$, and for each $y \in G$, the map $x \mapsto \varphi(x, y)$ is a Freiman homomorphism on the row $\{x \in G: (x,y) \in A\}$.\\
    \indent Then there exists a global biaffine map $\Phi : G \times G \to H$ (i.e., a map which is affine in each variable separately) such that $\varphi(x, y) = \Phi(x, y)$ for at least $\exp(-\log^{O(1)} c^{-1})|G|^2$ points $(x,y) \in A$.
\end{theorem}

Once this theorem is established, Theorem~\ref{u4invmain} follows from a symmetry argument of Green and Tao~\cite{GreenTaoU3}, and some arguments of additive combinatorics, including the partition vs.\ analytic rank problem for trilinear forms. There are some additional obstacles in the case of low characteristic, but these have been resolved by Tidor~\cite{Tidor}. This deduction is well-understood and is formulated as Theorem~\ref{u4nonclconcl}.\\
\indent We now give a short sketch of the proof of Theorem~\ref{fbihommain}, focusing on the main ideas.\\

\noindent\textbf{Brief overview of the proof.} The proof can be divided in several major steps. Let $\varphi : A \to H$ be a Freiman bihomomorphism defined a on set $A \subseteq G \times G$ of density $c$. By an additive quadruple we mean a quadruple $(a_1, a_2, a_3, a_4)$ that satisfies equality $a_1 - a_2 + a_3 - a_4 = 0$, or possibly some other permutation of elements leading to slightly different equality such as $a_1 + a_2 = a_3 + a_4$. However, it will always be obvious from the context which of these relations we refer to.\\
\indent We now proceed to give a brief overview of the proof; a more detailed exposition is the subject of Section~\ref{detailsoversec}.

\begin{itemize}
    \item[\textbf{Step 1.}] \textit{Obtaining a system of linear maps.} We use standard tools of additive combinatorics to relate $\varphi$ to some linear maps $\phi_a :G \to H$, defined for elements $a$ of a dense set $X \subseteq G$. This system has the property that 
    \begin{equation}\label{rankresp1}\on{rank}\Big(\phi_{a_1} - \phi_{a_2} + \phi_{a_3} - \phi_{a_4}\Big) \leq r\end{equation}
    holds for many additive quadruples and for some reasonably small $r$. We say that the additive quadruple $(a_1, a_2, a_3, a_4)$ is \textit{$r$-respected} when~\eqref{rankresp1} holds.
    \item[\textbf{Step 2.}] \textit{Abstract Balog-Szemer\'edi-Gowers theorem I.} We pass to a subset $X' \subseteq X$ on which all additive quadruples are $O(r)$-respected in the sense of~\eqref{rankresp1}. In fact, we need a slightly stronger conclusion that all additive 16-tuples respected, but we ignore this for the moment for simplicity. This step is in the spirit of Balog-Szemer\'edi-Gowers theorem, but it is more general than the classical version of the theorem in the sense that it concerns additive quadruples in a way that cannot directly be reduced to the study of pairs of elements. In the case of the original theorem, one can pass to a set of popular differences and essentially forget about the set of additive quadruples and focus immediately on the corresponding graph. The reason for calling this step an abstract Balog-Szemer\'edi-Gowers theorem is that it can be viewed as a version of the theorem that applies to approximate homomorphisms whose codomains are more general algebraic structures than groups. Namely, although the linear maps can be added together, the notion of `equality' (having small rank difference in our case) is not sufficiently well-behaved. The distinction from the usual group setting will be even more pronounced in \textbf{Step 6} below, where we shall not even have addition.\\
    \null\hspace{17pt}Furthermore, even though we use direct arguments in this paper, we extract an abstract statement of common ideas of this step and \textbf{Step 6} as Theorem~\ref{absg}, which might be of separate interest. 
    \item[\textbf{Step 3.}] \textit{Obtaining a full system of linear maps.} Once all additive 16-tuples of elements in $X'$ are respected, we may use the robust Bogolyubov-Ruzsa theorem to find another system, with the additional property that maps are defined for \textbf{all} indices $a$ in some subspace of small codimension.
    \item[\textbf{Step 4.}] \textit{Changing the category.} Let us revisit Question~\ref{linlinsysqn} before proceeding. That question can be rephrased as follows. Let $\mathcal{L}_{\leq r}$ be the set of linear maps $G \to H$ of rank at most $r$. The assumption of the question becomes $\phi_{a_1} - \phi_{a_2} + \phi_{a_3} - \phi_{a_4} \in \mathcal{L}_{\leq r}$ and the desired conclusion is $\phi_a - \Phi(a, \bcdot) \in \mathcal{L}_{\leq R}$ for a reasonably small $R$. In this form, the question seems very natural and it resembles structure theorem for approximate homomorphisms. However, the key obstacle is that the set $\mathcal{L}_{\leq r}$ has \textbf{terrible doubling properties}.\\
    \null\hspace{17pt} In this step, we make the crucial change of viewpoint. Rather than focusing on the set $\mathcal{L}_{\leq r}$ in the additive group of linear homomorphisms from $G$ to $H$, we consider more generally \textbf{partially defined} linear maps $\phi : U \to H$ for bounded codimension subspaces $U \leq G$ and we interpret the low rank condition as being zero on a bounded codimension subspace. For example, in the conclusion, instead of obtaining $\phi_a - \Phi(a, \bcdot) \in \mathcal{L}_{\leq R}$, we instead prove that $\phi_a(y) = \Phi(a, y)$ for elements $y$ of a subspace $S_a \leq G$ of codimension at most $R$. This also leads us to a different kind of property of being respected. From this point on, we say that an additive quadruple $(a_1, a_2, a_3, a_4)$ is \textit{subspace-respected} if 
    \begin{equation}
        \label{subspresp} \Big(\forall y \in U_{a_1} \cap U_{a_2} \cap U_{a_3} \cap U_{a_4}\Big)\hspace{1cm}\phi_{a_1}(y) - \phi_{a_2}(y) + \phi_{a_3}(y) - \phi_{a_4}(y) = 0,
    \end{equation}
    where $\phi_a$ has domain $U_a$. In other words, the appropriate linear combination of maps $\phi_{a_i}$ is zero on the set where all 4 maps are defined.
    \item[\textbf{Step 5.}] \textit{Bilinear Bogolyubov argument.} Having changed the viewpoint, in this step we modify the given system of the partially defined linear maps in such a way that allows us to assume that their domains $U_a$ depend linearly on the index $a$. This step is inspired by an argument of Hosseini and Lovett~\cite{HosseiniLovett}. 
    \item[\textbf{Step 6.}] \textit{Abstract Balog-Szemer\'edi-Gowers theorem II.} Similarly to the \textbf{Step 2}, we want to pass to a subset of the index set, where \textbf{all} the additive quadruples (and slightly longer similar configurations) are subspace-respected. However, it turns out that being subspace-respected is strictly harder to work with than the previous notion of being $r$-respected, so we need to make use of the algebraic regularity method (revisited in the preliminaries section) this time. The remaining two steps also make heavy use of this method.
    \item[\textbf{Step 7.}] \textit{Extending Freiman bihomomorphism from a structured set.} Once all additive 12-tuples are subspace-respected, we make another change of viewpoint and put all linear maps together to give a map from a subset of $G \times G$ to $H$. Using robust Bogolyubov-Ruzsa theorem, we obtain a Freiman bihomomorphism defined on a \textit{bilinear variety}, which is a set of the shape $\{(x,y) \in G \times G: \beta(x,y) = 0\}$ for some bilinear map $\beta : G \times G \to \mathbb{F}_p^r$. We say that $r$ is the \textit{codimension} of the variety.
    \item[\textbf{Step 8.}] \textit{Extending Freiman bihomomorphism from a bilinear variety.} Finally, using earlier arguments of Gowers and the author~\cite{U4paper}, we extend the Freiman bihomomorphism from the previous step to a global bilinear map. 
\end{itemize}

Once these steps were carried out, it is a simple task to relate the global bilinear map obtained in the final step to the initial map $\varphi$.\\ 
\indent In contrast to this paper, the previous quantitative approach of~\cite{U4paper} relied on an approximation theorem for iterated directional convolutions of functions in two variables (which can be seen as a strengthening of bilinear Bogolyubov argument~\cite{BienvenuLe, BilinearBog}). Such a theorem, with some additional additive combinatorial arguments, allows one to quickly arrive to case of having a Freiman bihomomorphism defined on $1-\varepsilon$ proportion of a bilinear variety, for some fixed constant $\varepsilon > 0$ (which is related to the \textbf{Step 8} above). However, it is this directional convolutions approximation theorem that introduces exponential loss in the previous work and it seems that one cannot evade such bounds in approximation theorems.\\
\indent It should also be noted that in the case of higher norms, the proof~\cite{FMulti} due to Gowers and the author also relies on an approximation theorem for higher dimensional directional convolutions, and, unlike the $\mathsf{U}^4$ case, has another place where exponential loss is present, namely in the extension result for Freiman multihomomorphism defined on $1-\varepsilon$ proportion of a multilinear variety, for fixed $\varepsilon > 0$. Interestingly, in the present proof, we actually evade such a situation, and only need an extension result for Freiman bihomomorphisms from $1-\varepsilon$ proportion of a bilinear variety, with $\varepsilon \leq p^{-Cr}$, where $r$ is the codimension of the variety and $C$ is a sufficiently large constant. This is an important distinction, as there are efficient extension results in higher dimension as well, if the domain is a full multilinear variety~\cite{extensionspaper}, (and the missing density smaller than $p^{-C r}$ is closely related to the full variety case).\\

\noindent\textbf{General abelian groups.} Although this paper is exclusively about finite vector spaces, the concepts and technical tools used here have been previously generalized to finite abelian groups in~\cite{boggenab}. In particular, we have appropriate generalizations of bilinear varieties and the algebraic regularity lemma for bilinear maps, allowing us to use algebraic regularity method in that setting. It is therefore plausible that our approach could lead to an inverse theorem for $\mathsf{U}^4$ norm in general finite abelian groups; we plan to return to this problem.\\

\noindent\textbf{Paper organization.} The rest of the paper is divided into several section. The next section contains preliminaries, important notation and useful general results. It is followed by a section devoted to a more detailed proof overview. The organization of the remainder of the paper follows the proof outline above, and there is a separate section for each step.\\

\noindent\textbf{Acknowledgements.} This research was supported by the Ministry of Science, Technological Development and Innovation of the Republic of Serbia through the Mathematical Institute of the Serbian Academy of Sciences and Arts, and by the Science Fund of the Republic of Serbia, Grant No.\ 11143, \textit{Approximate Algebraic Structures of Higher Order: Theory, Quantitative Aspects and Applications} - A-PLUS.\\
\indent I am grateful to Zach Hunter, James Leng and Terence Tao for comments on a draft of the paper.

\section{Preliminaries}\label{prelimsec}

\textbf{General notation and terminology.} As remarked in the introduction, we fix a prime $p$, and $G$ and $H$ stand for some finite-dimensional vector spaces over $\mathbb{F}_p$. We put a dot product on vector spaces, denoted by $\cdot$. By an \textit{additive quadruple} in a set $A \subset G$ we think of a quadruple $(a_1, a_2, a_3, a_4)$ such that $a_{\pi(1)} + a_{\pi(2)} = a_{\pi(3)} + a_{\pi(4)}$ for some fixed permutation $\pi$ of the set $[4]$, which will always be obvious from the context. Most frequently, the relationship will be $a_1 - a_2 + a_3 - a_4 = 0$, but sometimes we might insist on $a_1 + a_2 = a_3 + a_4$. This additional freedom simplifies the notation somewhat. More generally, we consider additive $2k$-tuples, which are $(a_1, \dots, a_{2k})$ that satisfy $\sum_{i \in [2k]} (-1)^i a_i = 0$.\\
\indent We use shorthand $a_{[k]}$ for $k$-tuple $(a_1, \dots, a_k)$.\\
\indent A map $\beta : G \times G \to H$ is \textit{bilinear} if it is linear in each variable separately and it is a \textit{bilinear form} if $H=\mathbb{F}_p$. A \textit{bilinear variety} in $G \times G$ is a zero set of such $\beta$. We say that a bilinear variety has \textit{codimension at most $r$}, if it can be defined by $\beta : G \times G \to \mathbb{F}_p^r$.\\
\indent For a subset $A \subseteq G \times G$, we write $A_{x \bcdot} = \{y \in G: (x,y) \in A\}$ for the column of $A$ indexed by $x \in G$, and, similarly, $A_{\bcdot y} = \{x \in G: (x,y) \in A\}$ for the row of $A$ indexed by $y \in G$.\\

\textbf{Additional asymptotic notation.} Apart from the standard big-Oh notation, i.e., $O(\cdot)$ and $\Omega(\cdot)$, we use additional, non-standard, asymptotic notation. The notation $\blc$ stands for sufficiently large positive constant whose actual value is not important. This notation is used for assumptions, and big-Oh for conclusions. For example, the fact that exponential grows faster than $x^{100}$ can be expressed as the follows: whenever $x \geq \blc$, then $\exp(x) \geq x^{100}$.\\
\indent From a logical perspective, following~\cite{FMulti}, $\blc$ is a shorthand in the following sense. Let $x, y_1, \dots, y_m$ be variables, let $X, Y_1, \dots, Y_m \subset \mathbb{R}$ be sets, let $f \colon \mathbb{R}^m \times \mathbb{R}^n \to \mathbb{R}$ be a function and let $P(x, y_1, \dots, y_m)$ be a proposition whose truth value depends on $x, y_1, \dots, y_m$. The notation 
\[(\forall x \in X)(\forall y \in Y_1)\dots (\forall y \in Y_m) x \geq f(y_1, \dots, y_m; \blc, \dots, \blc) \implies P(x,y_1,\dots,y_m)\] 
is then interpreted as a shorthand for
\[(\exists C_1 > 0) \dots (\exists C_n > 0)(\forall x \in X)(\forall y \in Y_1)\dots (\forall y \in Y_m) \hspace{3pt} x \geq f(y_1, \dots, y_m; C_1, \dots, C_n) \implies P(x,y_1,\dots,y_m).\]

Typically, such notation will appear in rank conditions on bilinear maps defining bilinear varieties that we want to be sufficiently quasirandom. For example, if $B$ is a bilinear variety defined by a bilinear form $\beta :G \times G \to \mathbb{F}_p$, provided $\on{rank} \beta \geq \blc$, we know that all but at most $\frac{1}{100}|G|$ columns $B_{x \bcdot}$ have size precisely $p^{-1}|G|$. As an instance of the use of this notation in the paper, see Claim~\ref{extendquadsclaim}.\\

The following lemma, appearing implicitly~\cite{GowU4}, is the graph-theoretic heart of Gowers's proof of Balog-Szemer\'edi-Gowers theorem. The following version appears in~\cite{TaoVu} as Corollary 6.20. We shall apply this lemma a few times in our paper.

\begin{lemma}[Gowers~\cite{GowU4}, dependent random choice (also Sudakov, Sz\'emeredi, Vu~\cite{SudSzeVu})]\label{gowerspaths}
    Let $G$ be a bipartite graph with vertex classes $A$ and $B$ of size $n$ with at least $cn^2$ edges. Then there are subsets $A' \subseteq A$ and $B' \subseteq B$ of size at least $c n/8$ with the property that there are at least $2^{-12}c^4n^2$ paths of length 3 between any $a \in A'$ and $b \in B'$.
\end{lemma}

As we shall work with graphs without bipartitions, and we also need the sets $A'$ and $B'$ to be the same in the conclusion, we rely on the following easy corollary instead (the proof is slightly longer than expected due to degenerate cases).

\begin{lemma}\label{gowerspathssingle}Let $G$ be a graph on $n$ vertices with at least $cn^2$ edges. Then there exists a  subset of vertices $X$ of size at least $2^{-5} c n$ with the property that there are at least $2^{-35} c^9 n^5$ paths of length 6 between any two vertices in $X$.\end{lemma}

\begin{proof}Let $V$ be the set of vertices. We shall first deal with the case when $n$ is even, and later easily deduce the result for odd $n$. The reason for this anomaly is that Lemma~\ref{gowerspaths} is about bipartitions of the same size.\\
\indent Let $n = 2m$. Take a random partition $V = V_1 \cup V_2$ into two sets of sizes $|V_1| = |V_2| = m$ and consider the bipartite graph induced by it. The probability that an edge $xy$ ends up in the induced bipartite graph is $\binom{n}{m}^{-1} \binom{n - 1}{m  - 1} = \frac{m}{n} = \frac{1}{2}$, so by linearity of expectation, there is such a partition with at least $\frac{c}{2}n^2 = 2c m^2$ edges between $V_1$ and $V_2$. By Lemma~\ref{gowerspaths}, there are subsets $A' \subseteq V_1$ and $B' \subseteq V_2$ of sizes at least $cm /8 = 2^{-4} c n$, such that there are at least $2^{-12}c^4m^2$ paths of length 3 between any $a \in A'$ and $b \in B'$. Hence, for any two $a, a' \in A'$, we may find $2^{-12}c^4m^2$ paths of length 3 from $a$ to $b$ and further $2^{-12}c^4m^2$ paths of length 3 from $b$ to $a'$ for each $b \in B'$. Hence, there are at least $2^{-27} c^9 m^5 = 2^{-32} c^9 n^5$ paths of length 6 between them, as desired.\\
\indent If $n$ is odd, suppose first that $cn \leq 2$. In this case, take an arbitrary vertex $x$ with at least one edge $xy$, and set $X = \{x\}$. There exists at least one path of length 6 starting and ending at $x$, so we are done and we may assume $cn > 2$. Similarly, we may assume that $n > 1$.\\
\indent Remove an arbitrary vertex, to get a graph with at least $cn^2 - n \geq \frac{c}{2}n^2$ edges on $n-1$ vertices. By the even case, we get desired set $X$ of size at least $2^{-4} c (n-1)$ with at least $2^{-32} c^9 (n-1)^5$ paths of length 6 between any two members of $X$. The result follows from trivial inequalities $n-1 \geq 2^{-1} n$ and $(n-1)^5 \geq 2^{-3} n^5$ and for $n \geq 3$.
\end{proof}

We use the robust version of Sanders's Bogolyubov-Ruzsa theorem, due to Schoen and Sisask~\cite{SchSis}.

\begin{theorem} [Robust Bogolyubov-Ruzsa lemma~\cite{SchSis}]\label{rbrlemma}
    Let $A \subseteq G$ be a set of density $\alpha$. Then there exists a subspace $U \leq G$ of codimension $O(\log^{O(1)} \alpha^{-1})$ such that for each $u \in U$, there are at least $\alpha^{O(1)}|G|^3$ quadruples $(a_1, a_2, a_3, a_4) \in A^4$ such that $a_1 + a_2 - a_3 - a_4 = u$.
\end{theorem}

Finally, we need a couple of facts about approximate homomorphisms. The first is the solution to the so-called $99\%$ case, which is elementary. The following version appears in~\cite{FMulti} as Lemma 25.

\begin{lemma}\label{3approxHom}Let $p$ be a prime and let $\epsilon \in (0, \frac{1}{100})$. Suppose that $G$ and $H$ are finite-dimensional $\mathbb{F}_p$-vector spaces, and that $\rho_1, \rho_2, \rho_3 \colon G \to H$ are three maps such that $\rho_1(x_1) - \rho_2(x_2) = \rho_3(x_1 - x_2)$ holds for at least a $1 - \epsilon$ proportion of the pairs $(x_1, x_2) \in G \times G$. Then there is an affine map $\alpha \colon G \to H$ such that $\rho_3(x) = \alpha(x)$ for at least $(1-\sqrt{\epsilon})|G|$ of $x \in G$.\end{lemma}

Secondly, we need a structure theorem which relates approximate homomorphisms to exact affine maps, which essentially stems from work of Gowers~\cite{GowU4}. Owing to recent resolution of Marton's conjecture~\cite{Marton1, Marton2} due to Gowers, Green, Manners and Tao, we now have polynomial bounds. For the purposes of this paper, an earlier version with quasipolynomial bounds due to Sanders~\cite{Sanders} would also suffice. 

\begin{theorem}[Structure theorem for approximate homomorphisms]\label{invhomm}
    Let $G$ and $H$ be finite-dimensional vector spaces over $\mathbb{F}_p$ and let $A \subset G$. Suppose that $\phi : A \to H$ is a map such that 
    \[
        \phi(x + a) - \phi(x) = \phi(y + a) - \phi(y)
    \]
    holds for at least $c|G|^3$ choices of $x,y, a \in G$ such that $x, y, x+a, y+a \in A$. Then there exists a global affine map $\Phi : G \to H$ such that $\phi(x) = \Phi(x)$ holds for at least $c^{O(1)}|G|$ elements $x \in G$.
\end{theorem}

\subsection{Algebraic regularity method}

\null\hspace{17pt} One of the most important tools in combinatorics is Szemerédi regularity lemma~\cite{SzemAP, SzemReg}. This lemma partitions a given graph into a bounded number of pieces most of whose pairs induce quasirandom bipartite graphs. Together with another result known as the counting lemma, which allows us to get asymptotics of the number of copies of any fixed graph in a quasirandom graph, this lemma is a basis of the (combinatorial) regularity method. However, as it is well-known, this method leads to terrible bounds in applications and this is unavoidable~\cite{TimRegConstruction, MSRegConstruction}.\\
\indent In the algebraic setting, we encounter graphs with specific structure. For example, in this paper, given a bilinear map $\beta : G \times G \to \mathbb{F}_p^r$, we may consider the bipartite graph, whose vertex classes are copies of $G$ and where a pair $xy$ is an edge if $\beta(x,y) = 0$. Such graphs are quasirandom precisely when the rank of $\beta$ is large. Such considerations lead to algebraic regularity method, whose bilinear case was developed in~\cite{U4paper}, and higher-order case originated in~\cite{FMulti}, with further applications in~\cite{asfbes, LukaU56}, in which the role of combinatorial regularity lemmas are taken by structural results for non-equidistributed forms. It should be noted, however, that, even though there is a pleasant analogy between the graph case and the bilinear varieties case, the algebraic regularity method is more involved in the higher order case, as the hypergraph regularity method involves hierarchies of structured pieces~\cite{TaoReg}, which lead to Ackermannian growth. Trying to pursue the same analogy in the higher order case would lead to a similar growth of bounds in Grzegorczyk hierarchy in the algebraic setting.\\
\indent The following theorem is the algebraic regularity lemma for our purpose in the present paper and appears as Corollary 5.2 in~\cite{U4paper}. 

\begin{theorem}[Algebraic regularity lemma] \label{arl}
    Suppose that $\beta \colon G \times G \to \mathbb{F}_p^r$ is a bilinear map and let $s > 0$. Then there exists a bilinear map $\gamma \colon G \times G \to \mathbb{F}_p^{r'}$, where $r' \leq r$, and a subspace $V$ of codimension at most $2rs$ such that 
    \begin{itemize}
        \item[\textbf{(i)}] for each $\lambda \not= 0$, the rank of $\lambda \cdot \gamma$ on $V \times V$ is at least $s$ (in this case we say that $\gamma$ has rank at least $s$),
        \item[\textbf{(ii)}] for all cosets $a + V$ and $b + V$ such that $\{\beta = 0\} \cap (a + V) \times (b + V) \not=\emptyset$ we have
        \begin{equation}\{\beta = 0\} \cap (a + V) \times (b + V) = \{\gamma = 0\} \cap (a + V) \times (b + V).\label{allcosetseq}\end{equation}
    \end{itemize}
 \end{theorem}

We give a proof for the sake of completeness.

\begin{proof}
    During the proof, we keep track of two subspaces, $V \leq G$ and $U \leq \mathbb{F}_p^r$. We set $V = G$ and $U = \mathbb{F}_p^r$ initially, and these subspaces will have the property that $\on{codim} V \leq 2i \cdot r$ and $\on{dim} U = r- i$ after step $i$. At each step, if $t = \dim U$ and $U$ has a basis $e_1, \dots, e_t$, we consider the map $\gamma : G \times G \to \mathbb{F}_p^t$ given by $\gamma_j = e_j \cdot \beta$. The condition~\eqref{allcosetseq} will hold at all times.\\
    \indent As long as $\gamma$ has insufficiently large rank, we may find a non-trivial linear combination $\lambda \in \mathbb{F}_p^t$ such that $\on{rank} \lambda \cdot \gamma|_{V \times V} \leq s$. If we write $\mu = \lambda_1 e_1 + \dots + \lambda_t e_t$, we see that $\on{rank} \mu \cdot \beta|_{V \times V} \leq s$ and $\mu \in U$. Hence, we may find linear forms $\alpha_i : V \to \mathbb{F}_p$ and $\alpha'_i : V \to \mathbb{F}_p$, $i \in [s]$ such that 
    \[\mu \cdot \beta(x,y) = \sum_{i \in [s]} \alpha_i(x) \alpha'_i(y)\]
    holds for all $x, y\in V$. Set $V' = \{x \in V: (\forall i \in [s]) \alpha_i(x) = \alpha'_i(x) = 0\}$. By the choice of this subspace, we have that $\mu \cdot \beta(x,y) = 0$ if $x \in V'$ or $y \in V'$. Take $U'$ to be any subspace such that $U = U' \oplus \langle \mu \rangle$.\\
    \indent We need to check~\eqref{allcosetseq} for the map $\gamma'$ arising from the subspace $U'$. Let $a + V'$ and $b + V'$ be any two cosets such that $\{\beta = 0\} \cap (a + V') \times (b + V') \not=\emptyset$. Clearly, $\{\beta = 0\} \cap (a + V') \times (b + V') \subseteq \{\gamma' = 0\} \cap (a + V') \times (b + V')$. On the other hand, let $x, y \in V'$ be such that $\gamma'(x + a, y + b) = 0$. We claim that $\mu \cdot \beta(x + a, y + b) = 0$ as well. This would then imply $\gamma(x + a, y + b) = 0$, so by~\eqref{allcosetseq} for $\beta$ and $\gamma$, we deduce $\beta(x +a, y + b) = 0$, as required.\\
    \indent Observe that, for any two $x', y' \in V'$, we have $\mu \cdot \beta(a + x', b + y') = \mu \cdot \beta(a, b)$, so $\mu \cdot \beta$ is constant on $(a + V') \times (b + V')$. Since $\{\beta = 0\} \cap (a + V') \times (b + V') \not=\emptyset$, it follows that this value is zero, and hence $\gamma(a + x, b+y) = 0$, as desired.\\
    \indent As we decrease the dimension of $U$ at each step, the procedure terminates after at most $r$ steps.
\end{proof}

The following lemma is our algebraic counting lemma, which generalizes Lemma 5.4 of~\cite{U4paper}.

\begin{lemma}\label{randomIntersection} 
    Suppose that $\beta : G \times G \to \mathbb{F}_p^r$ is a bilinear map such that $\on{rank} \beta|_{V \times V} \geq s$ for some subspace $V$. Fix cosets $a + V$ and $b + V$. Let $B$ be the bilinear variety given by $\{(x,y) \in (a + V) \times (b + V) : \beta(x,y) = 0\}$.\\
    \indent Suppose that $Q \subseteq (b + V)^k$ is a set of $k$-tuples of size $\delta |V|^k$. Pick $x_1, \dots, x_d$ independently and uniformly at random from $a + V$. Then the probability that $\Big| |Q \cap B_{x_1 \cdot}^k \cap \dots \cap B_{x_d \cdot}^k| - \delta p^{- k d r} |V|^k\Big| \geq \eta |V|^k$ is at most $3\eta^{-2}p^{ - s/4}$.
\end{lemma}

\begin{proof}
    Let $\lambda \in\mathbb{F}_p^r \setminus \{0\}$ be given. Observe that (on $V \times V$) $\|\omega^{\lambda \cdot \beta}\|_{\square} \leq p^{-s/4}$, namely
    \[\|\omega^{\lambda \cdot \beta}\|_{\square}^4 = \exx_{x, x', y, y' \in V} \omega^{\lambda \cdot \beta(x - x', y - y')} = p^{-\on{rank} \lambda \cdot \beta|_{V \times V}} \leq p^{-s}.\]

    We estimate the first and second moments of the random variable $\frac{1}{|V|^k} |Q \cap B_{x_1 \cdot}^k \cap \dots \cap B_{x_d \cdot}^k|$. Firstly,
    \begin{align*}&\exx_{x_1, \dots, x_d \in a + V} \frac{1}{|V|^k} |Q \cap B_{x_1 \cdot}^k \cap \dots \cap B_{x_d \cdot}^k| = \exx_{x_1, \dots, x_d \in a + V, y_1, \dots, y_k \in V} \prod_{i \in [d], j \in [k]}\id(\beta(x_i, b + y_j) = 0) \id_Q(b + y_1, \dots, b + y_k)\\
    &\hspace{2cm}= \exx_{\lambda_{1\,1}, \dots, \lambda_{d\,k} \mathbb{F}_p^r}\bigg(\exx_{\ssk{x'_1, \dots, x'_d \in V\\y_1, \dots, y_k \in V}} \prod_{i \in [d], j \in [k]} \omega^{\lambda_{i\,j} \cdot \beta(x'_i, y_j)} \prod_{i \in [d], j \in [k]} \omega^{\lambda_{i\,j} \cdot \beta(x'_i, b)} \prod_{i \in [d], j \in [k]} \omega^{\lambda_{i\,j} \cdot \beta(a, b + y_j)}\\
    &\hspace{10cm}\id_Q(b + y_1, \dots, b+ y_k)\bigg).
    \end{align*}

    For any choice of $\lambda_{1\,1}, \dots, \lambda_{d\,k}$, if some $\lambda_{i\,j} \not=0$, then we may use Gowers-Cauchy-Schwarz inequality to bound the contribution in the expression above by $p^{-s/4}$ in the absolute value. The remaining summand occurs when $\lambda_{1\,1} = \dots = \lambda_{d\,k} = 0$, when the contribution from the inner expectation is precisely $\delta$. Hence,
    \[\bigg|\exx_{x_1, \dots, x_d \in a + V} \frac{1}{|V|^k} |Q \cap B_{x_1 \cdot} \cap \dots \cap B_{x_d \cdot}| - p^{-r d k} \delta\bigg| \leq p^{ - s/4}.\]

    Similarly, 
    \begin{align*}&\exx_{x_1, \dots, x_d \in a + V} \frac{1}{|V|^{2k}} |Q \cap B_{x_1 \cdot}^k \cap \dots \cap B_{x_d \cdot}^k|^2\\
    &\hspace{2cm} = \exx_{\ssk{x_1, \dots, x_d \in a + V\\y_1, \dots, y_k \in V\\z_1, \dots, z_k \in V}} \prod_{i \in [d], j \in [k]}\id(\beta(x_i, b+ y_j) = 0) \prod_{i \in [d], j \in [k]}\id(\beta(x_i, b+ z_j) = 0)\\
    &\hspace{8cm}\id_Q(b+y_1, \dots, b + y_k)\id_Q(b+z_1, \dots, b + z_k)\\
    &\hspace{2cm}= \exx_{\ssk{\lambda_{1\,1}, \dots, \lambda_{d\,k} \in \mathbb{F}_p^r\\\mu_{1\,1}, \dots, \mu_{d\,k} \in \mathbb{F}_p^r\\x'_1, \dots, x'_d \in V, y_1, \dots, y_k, z_1, \dots, z_k \in V}} \prod_{i \in [d], j \in [k]} \omega^{\lambda_{i\,j} \cdot \beta(x'_i, y_j)} \prod_{i \in [d], j \in [k]} \omega^{\mu_{i\,j} \cdot \beta(x'_i, z_j)}\\
    &\hspace{4cm}\prod_{i \in [d], j \in [k]} \omega^{\lambda_{i\,j} \cdot \beta(x'_i, b)} \prod_{i \in [d], j \in [k]} \omega^{\mu_{i\,j} \cdot \beta(x'_i, b)} \prod_{i \in [d], j \in [k]} \omega^{\lambda_{i\,j} \cdot \beta(a, y_j + b)}\prod_{i \in [d], j \in [k]} \omega^{\mu_{i\,j} \cdot \beta(a, z_j + b)}\\
    &\hspace{4cm} \id_Q(b+y_1, \dots, b + y_r) \id_Q(b+z_1, \dots, b + z_r).
    \end{align*}

    As above, this expression equals $p^{-2rdk} \delta^2 + \varepsilon$, for some $|\varepsilon| \leq p^{ - s/4}$. Thus,
    \[\exx_{x_1, \dots, x_d} \Big| |Q \cap B_{x_1 \cdot} \cap \dots \cap B_{x_d \cdot}| - \delta p^{- d k r} |V|^k\Big|^2\leq 3 p^{-s/4} |V|^{2k}\]
    from which the claim follows.
\end{proof}

Finally, the following corollary about intersections of columns of a quasirandom variety with a given subspace will be used frequently in the later stage of the proof of the main result.

\begin{lemma} \label{subspsumsreg}
    Let $\beta$ and $B \subseteq (a + V) \times V$ be as in the previous lemma. Let $S \leq V$ be a subspace of codimension $d$. Let $k$ be a positive integer. Then
    \[(S \cap B_{x_1 \bcdot} \cap \dots \cap B_{x_k \bcdot}) + (S \cap B_{y_1 \bcdot} \cap \dots \cap B_{y_k \bcdot}) = S\]
    holds for all but at most $81p^{2d + 4kr -s/4} |V|^{2k}$ of choices $(x_{[k]},y_{[k]}) \in (a + V)^{2k}$.
\end{lemma}

\begin{proof}
    Using the fact that $S$ and $B_{x \bcdot}$ are subspaces, and the previous lemma with $\eta = \frac{1}{3}p^{-d - 2rk}$, we that have 
    \begin{align*}
       & |S \cap B_{x_1 \bcdot}  \cap \dots \cap B_{x_k \bcdot}| =  p^{-kr}|S| \\
       & |S \cap B_{y_1 \bcdot}  \cap \dots \cap B_{y_k \bcdot}| =  p^{-kr}|S| \\
       & |S \cap B_{x_1 \bcdot}  \cap \dots \cap B_{x_k \bcdot} \cap B_{y_1 \bcdot}  \cap \dots \cap B_{y_k \bcdot}| =  p^{-2kr}|S| \\
    \end{align*}
    hold for all but at most $81 p^{2d + 4rk - s/4}|V|^{2k}$ choices of $x_{[k]}, y_{[k]}$ in $a + V$. If these equalities hold, then
    \[|(S \cap B_{x_1 \bcdot} \cap \dots \cap B_{x_k \bcdot}) + (S \cap B_{y_1 \bcdot}  \cap \dots \cap B_{y_k \bcdot})| = \frac{|S \cap B_{x_1 \bcdot} \cap \dots \cap B_{x_k \bcdot}| |S \cap B_{y_1 \bcdot}  \cap \dots \cap B_{y_k \bcdot}|}{|S \cap B_{x_1 \bcdot}  \cap \dots \cap B_{x_k \bcdot} \cap B_{y_1 \bcdot}  \cap \dots \cap B_{y_k \bcdot}|} = |S|.\qedhere\]
\end{proof}

\section{Detailed proof overview}\label{detailsoversec}

    In this section, we elaborate on the proof overview provided in the introduction in the same step-by-step manner. Let $\varphi : A \to H$ be a Freiman bihomomorphism on a given set $A \subseteq G \times G$ of density $c$.\\ 

    \noindent \textbf{Step 1.} \textit{Obtaining a system of linear maps.} Using standard additive combinatorial arguments, we in fact assume that $\varphi$ satisfies a stronger condition. Instead of being just a Freiman homomorphism on each row $A_{x \bcdot} = \{y \in G: (x,y) \in A\}$ and column $A_{\bcdot y} = \{x \in G: (x,y) \in A\}$ of its domain, we may assume that $\varphi$ is a restriction of an affine map on each row and column. We may take linear parts $\phi_a$ of these affine maps for dense columns $A_{a \bcdot}$, for a certain dense index set $X \subseteq G$. Recall that an additive quadruple $a_1 - a_2 + a_3 - a_4 = 0$ in $X$ is \textit{$r$-respected} if 
    \[\on{rank}\Big(\phi_{a_1} - \phi_{a_2} + \phi_{a_3} - \phi_{a_4}\Big) \leq r.\]
    The fact that $\varphi$ is a Freiman bihomomorphism then guarantees that many additive quadruples in $X$ are $r$-respected for some $r \leq \log^{O(1)} c^{-1}$.\\       
    
    \noindent \textbf{Step 2.} \textit{Abstract Balog-Szemer\'edi-Gowers theorem I.} Given a linear system $(\phi_x)_{x \in X}$ with many $r$-respected additive quadruples, we aim to find a dense subset $X'$ of $X$ where all additive quadruples are $O(r)$-respected. This resembles the classical Balog-Szemer\'edi-Gowers theorem stating that any set $B$ with at least $\delta |B|^3$ additive quadruples contains a set with stronger additive structure, namely a $\Omega_\delta(1)$-dense subset $B' \subseteq B$ of doubling $O_\delta(1)$. However, as remarked in the introduction, the classical theorem is essentially only about pairs of elements with \textit{popular differences}, which are differences $d$ such that there are at least $\Omega_\delta(|B|)$ pairs $(b, b') \in B^2$ such that $d = b - b'$. The classical theorem then follows quickly from Lemma~\ref{gowerspaths} applied to the graph of pairs with popular differences.\\
    \indent In our case, as the set of linear maps of bounded rank has bad additive properties, the condition of being $r$-respected cannot be reduced in a similar way to information on pairs of elements of $X$, so instead we apply Lemma~\ref{gowerspaths} twice. First time, we apply it to the graph whose \textit{vertices are pairs of elements} of $X$. This application of the lemma gives certain sets of pairs as output, from which we build a new graph, where these pairs are now \textit{edges}. Applying Lemma~\ref{gowerspaths} the second time allows us to find a set $B \subseteq X$ with a special property of being \textit{everywhere arithmetically rich}, namely for all its dense subsets $B' \subseteq B$, there are still many $O(r)$-respected additive quadruples with all elements in $B'$.\\
    \indent Once we find such a set $B$, we remove its elements that belong to very few respected additive quadruples themselves, to obtain a further subset $C$ of $B$. Crucially, as $B$ is everywhere arithmetically rich, $C$ is still large and we show that $C$ is almost the set that we want. Namely, it turns out that for every additive quadruple $(a_1 ,a_2, a_3, a_4)$ we have many additive 12-tuples $b_{[12]}$ such that
    \[\on{rank}\Big(\phi_{a_1} - \phi_{a_2} + \phi_{a_3} - \phi_{a_4} - \sum_{i \in [12]} (-1)^i \phi_{b_i}\Big) \leq O(r).\]
    To simplify things, we may essentially assume, though this is not entirely correct, that there is a reasonably short list of maps $\pi_1, \dots, \pi_m :G \to G$ such that every additive quadruple $(a_1 ,a_2, a_3, a_4)$ in $C$ satisfies 
    \[\on{rank} \Big(\phi_{a_1} - \phi_{a_2} + \phi_{a_3} - \phi_{a_4} - \pi_i\Big) \leq O(r).\]
    To finally find the desired set, we perform algebraic dependent random choice argument that removes those maps $\pi_i$ that have large rank.\\
    
    \noindent \textbf{Step 3.} \textit{Obtaining a full system of linear maps.} As remarked in the introduction, this step is an application of the robust Bogolyubov-Ruzsa theorem and relies on the fact that the property of being $r$-respected behaves well with respect to addition.\\
    
    \noindent \textbf{Step 4.} \textit{Changing the category.} As remarked in the introduction, to get around the difficulty that the sets of low-rank linear maps have bad doubling properties, we make a change of category and consider partially defined maps instead of maps $G \to H$. The new domain subspaces are obtained using algebraic dependent random choice in a natural way. Namely, if $(\phi_x)_{x \in G}$ is the current system of linear maps, defined for all indices in $G$, we take random elements $a_1, \dots, a_m$, and take the subspace $U_x = \{y \in G: (\forall i \in [m])\,\, a_i \cdot \phi_x(y) = 0\}$. Observe that the index set remains intact and that the subspaces are still defined for all $x \in G$, which is crucial for the next step to succeed. The fact that additive quadruples are $r$-respected implies that the probability of being subspace-respected with respect to the subspaces chosen in this manner is high. It should be noted that, unlike usual applications of dependent random choice, we need to rely on the second-moment method to estimate all the required probabilities, and the linearity of expectation does not seem to be sufficient to complete this argument.\\
    
    \noindent \textbf{Step 5.} \textit{Bilinear Bogolyubov argument.} In this step, we assume that $(\phi_a, U_a)_{a \in G}$ are linear maps with domain $U_a$ for all $a \in G$ with the property that almost all additive 8-tuples are subspace-respected. Our goal is to pass to a related system where the subspaces \textit{vary linearly}, namely we shall find a new system $(\psi_a, B_a)_{a \in A}$, defined for a subset $A$ of indices, where subspaces $B_a$ are given by $\langle \theta_1(a), \dots, \theta_m(a)\rangle^\perp$ for some small $m$ and affine maps $\theta_i : G \to G$. In a nutshell, $\psi_a$ will arise in the following way. The key observation is that if the additive quadruple $(x + a, x, y + a, y)$ is subspace-respected then we have
    \[\phi_{x + a} - \phi_x = \phi_{y + a} - \phi_y\]
    on the subspace $(U_{x + a} \cap U_x) \cap (U_{y+a} \cap U_y)$. Hence, we have two linear maps $(\phi_{x + a} - \phi_x)|_{U_{x + a} \cap U_x}$ and $(\phi_{y + a} - \phi_y)|_{U_{y + a} \cap U_y}$ that agree on the intersection of their domains. Thus, there exists a unique linear extension of both maps on the sum of their domains $(U_{x + a} \cap U_x) + (U_{y+a} \cap U_y)$. We shall define $\psi_a$ as this extension.\\
    \indent On the other hand, from the bilinear Bogolyubov argument~\cite{BienvenuLe, BilinearBog, HosseiniLovett} (with some key ideas already present in Gowers in~\cite{GowerskAP}), we know that subspaces $(U_{x + a} \cap U_x) + (U_{y+a} \cap U_y)$ contain further linearly varying subspaces of low codimension. In our paper, we use an important observation of Hosseini and Lovett~\cite{HosseiniLovett} how to make the bounds efficient in this step. We need to generalize their approach to some extent as we also need to show the new system $(\psi_a, B_a)_{a \in A}$ itself has many subspace-respected additive quadruples.\\
    
    \noindent\textbf{Step 6.} \textit{Abstract Balog-Szemer\'edi-Gowers theorem II.} Similarly to the \textbf{Step 2}, we want to pass to a subset of indices where all additive quadruples are subspace-respected. However, unlike the low rank condition, the property of being subspace-respected is not approximately transitive. Namely, while it is trivial to check that additive quadruples $(x + a, x, y + a, y)$ and $(y + a, y, z + a, z)$ being $r$-respected implies that $(x + a, x, z + a, z)$ is $2r$-respected, assuming that the former quadruples are subspace-respected only implies that
    \[\phi_{x + a} - \phi_x - \phi_{z + a} + \phi_z = 0\]
    holds on the subspace $B_{x + a} \cap B_{x} \cap B_{y + a} \cap B_{y} \cap B_{z + a} \cap B_{z}$ which we expect to be strictly contained in $B_{x + a} \cap B_{x} \cap B_{z + a} \cap B_{z}$, which is the one that we want.\\
    \indent Fortunately, recalling that the subspaces $B_a$ arise as the columns of a bilinear variety, we may use the algebraic regularity method to recover transitivity. If the bilinear variety is sufficiently quasirandom, then having \textbf{many} elements $y$ such that $(x + a, x, y + a, y)$ and $(y + a, y, z + a, z)$ are subspace-respected also implies that $(x+a,x,z+a,z)$ is subspace-respected.\\
    \indent With this observation, the arguments here are similar to that \textbf{Step 2} to some extent, but there is still an important difference in how the step is completed. Namely, we use another algebraic dependent random choice argument, which, in order to be efficient, requires that we relate partial maps $\phi_{x + a} - \phi_x - \phi_{y + a} + \phi_y$ to \textbf{globally-defined maps} (i.e., defined on the whole $G$). This introduces further subtleties and explains why in the actual proof in Section~\ref{secbsg2}, we move from additive 12-tuples (which we need for \textbf{Step 7}) first to additive 36-tuples and then to additive 108-tuples.\\   
    
    \noindent\textbf{Step 7.} \textit{Extending Freiman bihomomorphism from a structured set.} In this step, by putting together the partially defined linear maps as the values on columns of a single function, we assume that we are given a Freiman 12-bihomomorphism (meaning that it respects all additive 12-tuples in both principal directions) on a set of the form $X \times G \cap B$, where $B$ is a bilinear variety. Using algebraic regularity method and robust Bogolyubov-Ruzsa theorem, we essentially extend the map to the whole variety $B$.\\
    
    \noindent\textbf{Step 8.} \textit{Extending Freiman bihomomorphism from a bilinear variety.} As remarked in the introduction, extending Freiman bihomomorphisms from bilinear varieties is a problem that has been studied before~\cite{U4paper, extensionspaper}.\\

Deducing the quasipolynomial structure theorem Freiman bihomomorphisms is now a simple matter of putting steps above together.\\

Finally, we record the abstract Balog-Szemer\'edi-Gowers theorem, containing the ideas appearing in \textbf{Step 2} and \textbf{Step 6} above. However, as the statement is somewhat hard to digest and \textbf{Step 6} itself requires revisiting the arguments in the proof, in the paper we carry out arguments from scratch in both steps, which results in minor repetition, but in our view leads to a cleaner presentation.

\begin{theorem}[Abstract Balog-Szemer\'edi-Gowers theorem]\label{absg}
    In this theorem, a quadruple $(a_1, a_2, a_3, a_4)$ is \textit{additive} if $a_1 - a_2 = a_3 - a_4$.\\
    \indent Let $A \subseteq G$ be a set. Suppose that, for each $i \in [36]$, we have a collection $\mathcal{Q}_i$ of additive quadruples in $A$, satisfying the following properties:
    \begin{itemize}
    \item[\textbf{(i)}] (largeness) $|\mathcal{Q}_1| \geq c |G|^3$,
    \item[\textbf{(ii)}] (symmetry) for each $i\in [36]$, if $(a_1, a_2, a_3, a_4) \in \mathcal{Q}_i$, then $(a_1, a_3, a_2, a_4), (a_3, a_4, a_1, a_2) \in \mathcal{Q}_i$,
    \item[\textbf{(iii)}] (weak transitivity) for any additive quadruple $(a_1, a_2, a_3, a_4) \in A^4$, if there are at least $c' |G|$ pairs $(b, b') \in A^2$ such that $(a_1, a_2, b, b') \in \mathcal{Q}_i$ and $(b, b', a_3, a_4) \in \mathcal{Q}_j$, then $(a_1, a_2, a_3, a_4) \in \mathcal{Q}_{i+j}$.
    \end{itemize}

    Then, provided $c' \leq (c/2)^\blc$, there exists a subset $A' \subseteq A$ such that for every quadruple $(a_1, a_2, a_3, a_4) \in {A'}^4$ we have $(c/2)^{O(1)}|G|^{11}$ 18-tuples indexed as $(x_{[4] \times [3]}, y_2, y_3, y_4, z_2, z_3, z_4)$, (where $x_{[4] \times [3]}$ is the shorthand for the 12-tuple of elements $x_{i\,j}$ with $i \in [4], j \in [3]$), such that:
    \begin{itemize}
        \item[\textbf{(i)}] for each $i \in [4]$, $(a_i, x_{i\, 1}, x_{i\, 2}, x_{i\, 3}) \in \mathcal{Q}_{36}$,
        \item[\textbf{(ii)}] $(x_{1\,1}, x_{2\,1}, y_2, z_2), (y_2, x_{3\,1}, y_3, z_3), (y_3, x_{4\,3}, y_4, z_4) \in \mathcal{Q}_{36}$.
    \end{itemize}
\end{theorem}

\indent A few remarks are in order. Firstly, notice that the conclusion holds for all quadruples in $A'$, not just the additive ones. As it should be expected, it is possible to quickly deduce the classical Balog-Szemer\'edi-Gowers theorem from the result above. However, in this paper we only focus on additive quadruples in the set $A'$.\\
\indent Focusing on additive quadruples, the theorem essentially says that for sets of additive quadruples that satisfy a natural symmetry condition and weak transitivity property, we have a structured set, in which all additive quadruples, while not necessarily belonging to the initial hierarchy of additive quadruples, can be embedded in a non-trivial configuration covered by the given additive quadruples.\\
\indent In our case, in \textbf{Step 2}, one could define $\mathcal{Q}_i$ to consist of all those additive quadruples that are $ir$-respected, so the hierarchy of quadruple sets is non-trivial, but transitivity condition is trivial. On the other hand, when it comes to \textbf{Step 6}, one can take $\mathcal{Q}_i$ to be the same for each $i$, consisting of those additive quadruples that are subspace-respected, but this time the weak transitivity property becomes a subtle matter.\\
\indent Finally, let us also mention that Tao proved a metric entropy version of Balog-Szemer\'edi-Gowers theorem~\cite{TaoBlogMetricBSG}. However, that version offers a 'small doubling' type of conclusion, which is a weaker than what Theorem~\ref{absg} implies and insufficient for our purposes. Additionally, one does not have a metric group structure in \textbf{Step 6}.

\section{Freiman bihomomorphisms and small rank condition}

In this section, we pass from Freiman bihomomorphisms to systems of linear maps that have many respected additive quadruples in a suitable sense. The second point of the conclusion allows us to deduce structural information about the given Freiman bihomomorphism, once we obtain sufficient understanding of the system of linear maps.

\begin{proposition}\label{passingtosystemstep}
    Let $\varphi : A \to H$ be a Freiman bihomomorphism on a $c$-dense subset $A$ of $G \times G$. Then there exist a set $X$, a collection of linear maps $\phi_x : G \to H$ and subspaces $U_x \leq G$ of codimension $O(\log^{O(1)}c^{-1})$ for each $x \in X$, and positive quantity $c' \geq \exp(-\log^{O(1)}c^{-1})$ and integer $r \leq \log^{O(1)}c^{-1}$, such that 
    \begin{itemize}
        \item $|X| \geq c'|G|$,
        \item for each $x \in X$ and $y \in U_x$, there are at least $c'|G|^3$ of triples $z_1, z_2, z_3 \in G$ such that the points $(x, z_1), (x, z_2), (x, z_3), (x, z_1 + z_2 - z_3 - y) \in A$ and 
        \[\phi_x(y) = \varphi(x, z_1) + \varphi(x, z_2) - \varphi(x, z_3) - \varphi(x, z_1 + z_2 - z_3 - y),\]
        \item for at least $c'|G|^3$ of additive quadruples in $X$, we have
        \[\on{rank}\Big(\phi_{x_1} + \phi_{x_2} - \phi_{x_3} - \phi_{x_4}\Big) \leq r.\]
    \end{itemize}
\end{proposition}

\begin{proof}
    Take $X$ to be the set of all $c/2$-dense columns. By averaging, $|X| \geq \frac{c}{2}|G|$. Since $\varphi$ is a Freiman bihomomorphism, we may apply the structure theorem for approximate homomorphisms (Theorem~\ref{invhomm}) to find a subset $B_x \subseteq A_{x \bcdot}$, of size $|B_x| \geq (c/2)^{O(1)}|G|$, on which $\varphi(x, \bcdot)$ coincides with an affine map $\theta_x$, whose linearization is $\phi_x$. By robust Bogolyubov-Ruzsa lemma (Theorem~\ref{rbrlemma}), there exists a subspace $U_x$ of codimension $O(\log^{O(1)} c^{-1})$ such that every element $y \in U_x$ can be written in at least $(c/2)^{O(1)}|G|^3$ ways as $z_1 + z_2 - z_3 - z_4$ for $z_i \in B_x$.\\

    Using the Cauchy-Schwarz inequality a couple of times, we get
    \begin{align*}&\frac{1}{|G|}\exx_{x,y, a} \id_{X}(x + a)\id_{X}(x)\id_{X}(y + a)\id_{X}(y) |B_{x + a} \cap B_{x} \cap B_{y+a} \cap B_{y}|\\
    &\hspace{2cm}=  \exx_{x,y, a, u} \id_{X}(x + a)\id_{X}(x)\id_{X}(y + a)\id_{X}(y) \id_{B_{x+a}}(u)\id_{B_{x}}(u)\id_{B_{y + a}}(u)\id_{B_{y}}(u)\\
    &\hspace{2cm}= \exx_{a, u} \Big(\exx_x \id_{X}(x + a)\id_{X}(x)\id_{B_{x+a}}(u)\id_{B_{x}}(u)\Big)^2\\
    &\hspace{2cm}\geq  \Big(\exx_{a, u, x} \id_{X}(x + a)\id_{X}(x)\id_{B_{x+a}}(u)\id_{B_{x}}(u)\Big)^2\\
    &\hspace{2cm}= \Big(\exx_{u, x, y} \id_{X}(x)\id_{X}(y)\id_{B_{x}}(u)\id_{B_{y}}(u)\Big)^2\\
    &\hspace{2cm}= \Big(\exx_{u} \Big(\exx_x\id_{X}(x)\id_{B_{x}}(u)\Big)^2\Big)^2\\
    &\hspace{2cm}\geq \Big(\exx_{u, x} \id_{X}(x)\id_{B_{x}}(u)\Big)^4 \geq (c/2)^{O(1)}.\end{align*}
    
    Hence, there are at least $ (c/2)^{O(1)}|G|^3$ additive quadruples $(x + a, x, y+ a, y)$ such that $|B_{x + a} \cap B_{x} \cap B_{y+a} \cap B_{y}| \geq  (c/2)^{O(1)}|G|$. Take any such additive quadruple and observe that, for any $z_1, z_2 \in B_{x + a} \cap B_{x} \cap B_{y+a} \cap B_{y}$ we have 
    \begin{align*}&\phi_{x + a}(z_1-z_2) - \phi_{x}(z_1-z_2) + \phi_{y}(z_1-z_2) -\phi_{y + a}(z_1-z_2) \\
    &\hspace{2cm}=  \theta_{x+a}(z_1) - \theta_{x + a}(z_2) - \theta_{x}(z_1) + \theta_{x}(z_2) + \theta_{y}(z_1) - \theta_{y}(z_2) - \theta_{y + a}(z_1) + \theta_{y + a}(z_2) \\
    &\hspace{2cm} = \varphi(x+a,z_1) - \varphi(x + a,z_2) - \varphi(x,z_1) + \varphi(x, z_2) \\
    &\hspace{6cm}+ \varphi(y,z_1) - \varphi(y,z_1) - \varphi(y+a,z_1) + \varphi(y+a,z_1)\\
    &\hspace{2cm} = \Big(\varphi(x+a,z_1)  - \varphi(x,z_1)+ \varphi(y,z_1) - \varphi(y+a,z_1)\Big)\\
    &\hspace{6cm}- \Big(\varphi(x+a,z_2)  - \varphi(x,z_2)+ \varphi(y,z_2) - \varphi(y+a,z_2)\Big)= 0.\end{align*}
    Hence, we obtain $\on{rank} \Big(\phi_{x + a} - \phi_{x} + \phi_y - \phi_{y + a}\Big) \leq \log^{O(1)}c^{-1}$, which is the desired low rank condition.\end{proof}

\section{Getting all additive 16-tuples with small rank combinations}

We say that an additive quadruple $(x_1, x_2, x_3, x_4)$, with $x_1 + x_2 = x_3 + x_4$, is \textit{$R$-respected} if 
\[\on{rank} \Big(\phi_{x_1} + \phi_{x_2} - \phi_{x_3} - \phi_{x_4} \Big) \leq R.\]
The goal of this section is to pass from a set where a positive proportion of additive quadruples is respected, to a subset where \textit{all} additive quadruples are respected. In fact, for technical reasons, we shall prove a slightly stronger result that all additive 16-tuples are respected, as this will allow us to eventually replace a partial domain with full $G$. As explained in the overview of the proof, the heart of this section is the abstract Balog-Szemer\'edi-Gowers argument.

\begin{proposition} \label{absBSGFirst}
    Suppose that we are given linear maps $\phi_x \colon G \to H$, $x \in X$, such that $\on{rank} \phi_{x_1} - \phi_{x_2} + \phi_{x_3} - \phi_{x_4} \leq r$ holds for at least $c |G|^3$ additive quadruples in $X$. Then there exists a subset $\tilde{X}$ of size $\exp(-O(r^{O(1)} \log^{O(1)} c^{-1})) |G|$ such that 
    \[\on{rank}\Big(\sum_{i \in [16]} \phi_{x_i}\Big) \leq O(r)\]
    holds for all additive 16-tuples $x_{[16]}$ in $\tilde{X}$.
\end{proposition}

\begin{proof}
    Consider the set of pairs $X^2$. We say that a pair of pairs $\Big((x, x'), (y, y')\Big)$ is \textit{good}, if they have the same difference, i.e., $x - x' = y - y'$, and form a respected additive quadruple, namely
    \[\phi_{x} - \phi_{x'} + \phi_{y'} - \phi_{y} \leq r.\]
    Hence, there are at least $c|G|^3$ good pairs of pairs. Since the number of pairs of pairs of a fixed difference is at most $|G|^2$, we can find a set $A$ of differences $a$ of size $|A| \geq \frac{c}{2}|G|$ such that there are at least $\frac{c}{2}|G|^2$ good pairs of pairs of difference $a$.\\    

    Fix an arbitrary $a \in A$, and look at the graph $\Gamma_a$ whose vertices are pairs $(x  + a, x)$ of difference $a$ and whose edges are put between good pairs of such pairs (notice that this is symmetric). Hence, the graph $\Gamma_a$ has at most $|G|$ vertices that lie inside $X^2$ and has at least $\frac{c}{4}|G|^2$ edges, owing to the choice of $a \in A$. Note that the number of edges also implies that $\Gamma_a$ has at least $\frac{c}{4}|G|$ vertices.\\
    \indent Apply Lemma~\ref{gowerspathssingle} to find a further set of pairs $P_a$ of difference $a$ and size at least $2^{-9} c^2 |G|$ such that whenever $(x + a, x), (y + a, y) \in P_a$, then there exists a 6-path $(x + a, x),$ $(z_1 + a, z_1), \dots,$ $(z_5 + a, z_5),$ $(y + a, y)$ in the graph $\Gamma_a$. Thus, any two consecutive pairs in this sequence produce a good pair of pairs of difference $a$. In particular, it follows that if $(x + a, x), (y + a, y) \in P_a$ are any two pairs in $P_a$, then the resulting additive quadruple $(x + a, x, y+a, y)$ is $6r$-respected. To see this, consider the above-mentioned 6-path along pairs $(z_i + a, z_i)$, $i \in [5]$, and the observe that
    \begin{align*}(\phi_{x + a} - \phi_x) - (\phi_{y+a} - \phi_y) = \Big((\phi_{x + a} - \phi_x) - (\phi_{z_1 + a} - \phi_{z_1})\Big) +& \Big((\phi_{z_1 + a} - \phi_{z_1}) - (\phi_{z_2 + a} - \phi_{z_2})\Big) + \dots\\
    + &\Big((\phi_{z_5 + a} - \phi_{z_5})- (\phi_{y+a} - \phi_y)\Big).\end{align*}
    
    We now want to take the union of these sets $P_a$ for all $a \in A$. However, to ensure symmetry, we only consider a single $a \in A$ for each pair $\{a, -a\}$ and then declare $P_{-a}$ to be $\{(x, x + a): (x + a, x) \in P_a\}$. Clearly, $P_{-a}$ has the same property concerning pairs of its pairs. Let $P$ be the union of all these sets of pairs $\cup_{a} P_a$, hence $|P| \geq 2^{-10} c^3 |G|^2$ and $P$ is symmetric.\\

    Now define a graph $\Pi$ on the vertex set $G$ with edges given by pairs in $P$, which thus has density at least $2^{-10} c^3$. Apply Lemma~\ref{gowerspathssingle} again, but this time to the graph $\Pi$ to find a set $B$ of size at least $c_1 |G|$, with $c_1 = 2^{-15} c^3|G|$, between whose elements we always have at least $c_2 |G|^5$ of 6-paths, where $c_2 = 2^{-200} c^{27}$.\\

    The fundamental property of $B$ is that it is arithmetically-rich everywhere, i.e., that all its further dense subsets give many respected additive quadruples.

    \begin{claim}\label{bsgcondlemma}
        Suppose that $B_1, B_2 \subseteq B$ are subsets of sizes $c'|G|$ and $c''|G|$. Then there exist at least $ c_2^2 {c'}^2 {c''}^2 |G|^3$ additive quadruples $b_1 - b_2 = b_1' - b'_2$ in $B_1 \times B_2 \times B_1 \times B_2$ that are $36r$-respected.
    \end{claim}

    \begin{proof}
        Look at any $b_1 \in B_1, b_2 \in B_2$. Both elements are in $B$, so there are $c_2|G|^5$ 6-paths of edges in $P$ between them. Hence, in total there are at least ${c'} c'' c_2 |G|^7$ of 7-tuples $(b_1, z_1, \dots, z_5, b_2) \in G^7$ such that $b_1 \in B_1, b_2 \in B_2$ and $(b_1, z_1), (z_1, z_2), \dots, (z_5, b_2) \in P$.\\
        \indent Changing the variables, namely by taking consecutive differences $d_1, \dots, d_6$, we deduce that there are at least ${c'} c'' c_2 |G|^7$ of 7-tuples $(b_1, d_1, d_2, \dots, d_6) \in G^7$ such that $b_1 \in B_1, b_1 - d_1 - \dots - d_6  \in B_2$ and $(b_1 - d_1 - \dots - d_{i-1}, b_1 - d_1 - \dots - d_i) \in P$ for $i \in [6]$.\\
        \indent By Cauchy-Schwarz inequality, there are at least ${c'}^2 {c''}^2 c_2^2 |G|^8$ of 8-tuples $(b_1, b'_1, d_1, d_2, \dots, d_6) \in G^8$ such that $b_1, b'_1 \in B_1, b_1 - d_1 - \dots - d_6, b'_1 - d_1 - \dots - d_6  \in B_2$ and $(b_1 - d_1 - \dots - d_{i-1}, b_1 - d_1 - \dots - d_i)$, $(b'_1 - d_1 - \dots - d_{i-1}, b'_1 - d_1 - \dots - d_i) \in P$ for $i \in [6]$.\\
        Take any such $8$-tuple. Set $z_i = b_1 - d_1 - \dots - d_i$ and $z'_i = b'_1 - d_1 - \dots - d_i$ for $i \in [0,6]$. In particular, $b_1 = z_0$, $b'_1 = z'_0$ and $z_6 = b_1 - d_1 - \dots - d_6$, $z'_6 = b_1 - d_1 - \dots - d_6 \in B_2$. Observe that we have $b_1  - b'_1 = z_i - z'_i$ for all $i \in [0,6]$. Furthermore, we claim that every additive quadruple $(b_1, b'_1, z_i, z'_i)$ is $i \cdot 6r$ respected. We prove this claim by induction on $i$, with the base case $i = 0$ being trivial. Assume the claim holds for some $i \geq 0$ and observe that
            \begin{align*}
                \Big(\phi_{z_{i+1}} - \phi_{z'_{i+1}}\Big) -  \Big(\phi_{b_1} - \phi_{b'_1}\Big) &= \Big(\phi_{z_{i+1}} - \phi_{z'_{i+1}}\Big) -  \Big(\phi_{z_i} - \phi_{z'_i}\Big) + \Big(\phi_{z_{i}} - \phi_{z'_{i}}\Big) -  \Big(\phi_{b_1} - \phi_{b'_1}\Big)\\
                & = \Big(\phi_{z_{i+1}} - \phi_{z_{i}}\Big) -  \Big(\phi_{z'_{i + 1}} - \phi_{z'_i}\Big) + \Big(\phi_{z_{i}} - \phi_{z'_{i}}\Big) -  \Big(\phi_{b_1} - \phi_{b'_1}\Big).
            \end{align*}

        Crucially, $(z_{i}, z_{i+1})$ and $(z'_{i}, z'_{i+1})$ both belong to $P_{d_i}$ and hence form a $6r$-respected additive quadruple. We are done by induction hypothesis as $\on{rank} \Big(\phi_{z_{i}} - \phi_{z'_{i}}\Big) -  \Big(\phi_{b_1} - \phi_{b'_1}\Big) \leq i \cdot 6r$.\\
        \indent Write $b_2 = z_6, b'_2 = z'_6$. Thus
        \[\Big(\phi_{b_1} - \phi_{b_2}\Big) - \Big(\phi_{b'_1} - \phi_{b'_2}\Big) = \Big(\phi_{b_1} - \phi_{b'_1}\Big) - \Big(\phi_{b_2} - \phi_{b'_2}\Big)\]
        has rank at most $36r$. Therefore, $b_1 - b_2 = b_1' - b'_2$ in $B_1 \times B_2 \times B_1 \times B_2$ is $36r$-respected.\\
        \indent Finally, observe that each quadruple $(b_1, b_2, b'_1, b'_2)$ can arise from at most $|G|^5$ 8-tuples $(b_1, b'_1, d_1, d_2, \dots, d_6)$, so the claim follows.    
    \end{proof}
    
    Let $B'$ be the set of elements which appear in at least $2^{-4} c_1^4 c_2^2 |G|^2$ $36r$-respected additive quadruples in $B$. We claim that $|B \setminus B'| \leq \frac{c_1}{2}|G|$. If this fails, then $|B\setminus B'| \geq \frac{c_1}{2}|G|$. Applying Claim~\ref{bsgcondlemma} gives $2^{-4} c_1^4 c_2^2 |G|^3$ $36r$-respected additive quadruples in $B \setminus B'$, so there is an element in $B \setminus B'$ that belongs to at least $2^{-4} c_1^4 c_2^2 |G|^2$ $36r$-respected additive quadruples in $B$, which is a contradiction. Hence, $|B'| \geq \frac{c_1}{2}|G|$.\\

    We claim that there is a list of linear maps $\psi_1, \dots, \psi_m : G \to H$, for some $m \leq c^{-O(1)}$, such that every additive 16-tuple $(b_1, b_2, \dots, b_{16})$ in $B'$, with $\sum_{i \in [16]} (-1)^i b_i = 0$, satisfies 
    \[\on{rank}\Big(\sum_{i \in [16]} (-1)^i\phi_{b_i}  - \psi_j\Big) \leq O(r)\]
    for some $\psi_j$.\\

    To see this, take an additive 16-tuple $(b_1, b_2, \dots, b_{16})$ in $B'$. By induction on $\ell \in [16]$, we show that that there are at least $\delta_\ell|G|^{3\ell - 1}$ of $3\ell$-tuples $(x_1, \dots, x_{3\ell}) \in B^{3\ell}$, where $\delta_\ell = \Big(\frac{c}{2}\Big)^{O(1)}$, such that 
    \[\sum_{i \in [\ell]} (-1)^{i} b_i = \sum_{i \in [3\ell]} (-1)^i x_i\]
    and
    \[\on{rank}\bigg(\Big(\sum_{i \in [\ell]} (-1)^{i} \phi_{b_i} \Big) - \Big(\sum_{i \in [3\ell]} (-1)^i \phi_{x_i}\Big)\bigg) \leq 72 \ell r.\]

    Let us first cover the base case $\ell = 1$. By the definition of $B'$, the element $b_1$ belongs to at least $2^{-4} c_1^4 c_2^2 |G|^2$ $36r$-respected additive quadruples in $B$. Each such additive quadruple can be written as $(b_1, x_1, x_2, x_3)$ with $b_1 = x_1 - x_2 + x_3$ and 
    \[\on{rank}\Big(\phi_{b_1} - (\phi_{x_1} - \phi_{x_2} + \phi_{x_3})\Big) \leq 36r,\]
    as desired.\\
    
    Now suppose that the claim holds for some $\ell \leq 15$. By induction hypothesis, there exists a set $X$ of size $\delta_\ell |G|^{3\ell -1}$ whose elements are $3\ell$-tuples with the described property. By the definition of $B'$, similarly to the base case, for every element $b_{\ell + 1} \in B'$ we have a set $Y$ of size at least $2^{-4} c_1^4 c_2^2 |G|^2 = \delta_1 |G|^2$ consisting of triples $(y_1, y_2, y_3)$ such that $b_{\ell + 1} = y_1 - y_2 + y_3$ and 
    \[\on{rank}\Big(\phi_{b_{\ell+1}} - (\phi_{y_1} - \phi_{y_2} + \phi_{y_3})\Big) \leq 36r.\]

    Observe that, for any element $x_{3\ell}$, there are at most $|G|^{3\ell-2}$ $3\ell$-tuples that belong in $X$ and end with $x_{3\ell}$. Thus, by averaging, there exists a set $B_1 \subseteq B$, of size at least $\frac{\delta_\ell}{2} |G|$, such that for each its member $x_{3\ell} \in B_1$, the number of $3\ell$-tuples ending with $x_{3\ell}$ is at least $\frac{\delta_\ell}{2} |G|^{3\ell -2}$. Similarly, by averaging, we may find a set $B_2 \subseteq B$ of size at least $\frac{\delta_1}{2}|G|$ such that for each $y_1 \in B_2$, there are at least $\frac{\delta_1}{2}|G|$ triples of the form $(y_1, y_2, y_3) \in Y$.\\

    Applying Lemma~\ref{bsgcondlemma} to the sets $B_1$ and $B_2$, we get $2^{-4} c_2^2 \delta_1^2 \delta^2_\ell |G|^3$ of $36r$-respected additive quadruples in $B_1 \times B_2 \times B_1 \times B_2$. Let $(z, w, x_{3\ell}, y_1) \in B_1 \times B_2 \times B_1 \times B_2$ be such a quadruple, thus $z - w = x_{3\ell} - y_1$. Recalling the definitions of $B_1$ and $B_2$, we conclude that there are at least $2^{-6} c_2^2 \delta_1^3 \delta^3_\ell |G|^{3\ell + 2}$ of $(3\ell + 5)$-tuples of the form
    \[(z, w, x_1, \dots, x_{3\ell}, y_1, y_2, y_3)\]
    of elements in $B$, with properties that $z - w = x_{3\ell} - y_1$, $\sum_{i \in [\ell]} (-1)^i b_i = \sum_{i \in [3\ell]} (-1)^i x_i$, $b_{\ell + 1} = y_1 - y_2 + y_3$,
    \[\on{rank}\bigg(\Big(\sum_{i \in [\ell]} (-1)^{i} \phi_{b_i} \Big) - \Big(\sum_{i \in [3\ell]} (-1)^i \phi_{x_i}\Big)\bigg) \leq 72 \ell r\] 
    and quadruples $(z, w, x_{3\ell}, y_1)$ and $(b_{\ell + 1}, y_1, y_2, y_3)$ are $36r$-respected. In particular, 
    \begin{align*}\sum_{i \in [\ell + 1]} (-1)^i b_i = &\Big(\sum_{i \in [\ell]} (-1)^i b_i\Big) + (-1)^{\ell + 1} b_{\ell + 1}\\
    = &\Big(\sum_{i \in [3\ell]} (-1)^i x_i\Big) + (-1)^{\ell + 1} (y_1 - y_2 + y_3) \\
    = &\Big(\sum_{i \in [3\ell- 1]} (-1)^i x_i\Big) + (-1)^{\ell + 1} ( - y_2 + y_3) + (-1)^{\ell} (x_{3\ell} - y_1)\\
    = &\Big(\sum_{i \in [3\ell- 1]} (-1)^i x_i\Big) + (-1)^{\ell + 1} ( - y_2 + y_3) + (-1)^{\ell} (z - w)\end{align*}
    and, performing similar algebraic manipulations, 
    \begin{align*}&\on{rank}\bigg(\Big(\sum_{i \in [\ell + 1]} (-1)^i \phi_{b_i}\Big) - \Big(\Big(\sum_{i \in [3\ell- 1]} (-1)^i \phi_{x_i}\Big) + (-1)^{\ell + 1} ( - \phi_{y_2} + \phi_{y_3}) + (-1)^{\ell} (\phi_z - \phi_w)\Big)\bigg)\\
    &\hspace{2cm}\leq \on{rank}\bigg(\Big(\sum_{i \in [\ell + 1]} (-1)^i \phi_{b_i}\Big) - \Big((-1)^{\ell + 1} \phi_{b_{\ell + 1}} + \sum_{i \in [3\ell]} (-1)^i \phi_{x_i}\Big)\bigg)\\
    &\hspace{3cm} + \on{rank}\bigg(\Big((-1)^{\ell + 1} \phi_{b_{\ell + 1}} + \sum_{i \in [3\ell]} (-1)^i \phi_{x_i}\Big) - \Big((-1)^{\ell + 1} (\phi_{y_1} - \phi_{y_2} + \phi_{y_3}) + \sum_{i \in [3\ell]} (-1)^i \phi_{x_i}\Big)\bigg)\\
    &\hspace{3cm} + \on{rank}\bigg(\Big((-1)^{\ell + 1} (\phi_{y_1} - \phi_{y_2} + \phi_{y_3}) + \sum_{i \in [3\ell]} (-1)^i \phi_{x_i}\Big)\\
    &\hspace{6cm}- \Big((-1)^{\ell + 1} ( - \phi_{y_2} + \phi_{y_3}) + \sum_{i \in [3\ell - 1]} (-1)^i \phi_{x_i} + (-1)^{\ell}(\phi_z - \phi_w)\Big)\bigg) \\
    &\hspace{2cm} = \on{rank}\bigg(\Big(\sum_{i \in [\ell]} (-1)^{i} \phi_{b_i} \Big) - \Big(\sum_{i \in [3\ell]} (-1)^i \phi_{x_i}\Big)\bigg) + \on{rank}\bigg(\phi_{b_{\ell + 1}} - \phi_{y_1} + \phi_{y_2} - \phi_{y_3}\bigg)\\
    &\hspace{3cm} + \on{rank}\bigg(\phi_{x_{3\ell}} - \phi_{y_1} - \phi_z + \phi_w\bigg)\\
    &\hspace{2cm} \leq 72(\ell + 1) r.\end{align*}
    Finally, consider the projection    
    \[(z, w, x_1, \dots, x_{3\ell}, y_1, y_2, y_3) \mapsto (z, w, x_1, \dots, x_{3\ell - 1}, y_2, y_3),\]
    given by omitting $x_{3\ell}$ and $y_1$. On the coset of a vector space given by equations $z - w = x_{3\ell} - y_1$, $\sum_{i \in [\ell]} (-1)^i b_i = \sum_{i \in [3\ell]} (-1)^i x_i$, $b_{\ell + 1} = y_1 - y_2 + y_3$, inside which we operate, the projection is injective. Thus, we get $2^{-6} c_2^2 \delta_1^3 \delta^3_\ell |G|^{3\ell + 2}$ of the desired $(3\ell + 3)$-tuples.\\

    Hence, when $\ell = 16$, for $b_{[16]}$ we have found a set of $\Omega(c^{O(1)}|G|^{47})$ additive 48-tuples $z_{[48]}$ such that 
    \[\on{rank}\Big(\sum_{i \in [16]} (-1)^i\phi_{b_i} - \sum_{i \in [48]} (-1)^i \phi_{z_i}\Big) \leq 2000 r.\]
    
    Take a maximal disjoint collection of these subsets of $48$-tuples, and set $\psi_1, \dots, \psi_m$ to be the linear combinations $\sum_{i \in [48]} (-1)^i \phi_{z_i}$, where we choose a single 48-tuple representative for each set. Hence $m \leq O(c^{-O(1)})$ and for each additive 16-tuple $(b_1, b_2, \dots, b_{16})$ in $B'$, we have
    \[\on{rank}\Big(\sum_{i \in [16]} (-1)^i\phi_{b_i}  - \psi_j\Big) \leq 6000r\]
    for some $\psi_j$, as claimed.\\

    To finish the proof, we apply algebraic dependent random choice to reduce to the case $\psi = 0$. 

\begin{claim}
    There exists a subset $B''\subseteq B'$ of size $p^{-O(r^{2} \log c^{-1})}|G|$ that has all additive 16-tuples $O(r)$-respected.
\end{claim}

\begin{proof}
    Write $R = 6000r$, which is the rank bound above. Let $d$ be a positive integer to be chosen later. We shall first find a sequence of subspaces $Y_1, \dots, Y_\ell \leq G$, for some reasonably small $\ell$, of dimension $R + 1$ and linear maps $\pi_1, \dots, \pi_\ell : H \to \mathbb{F}_p^d$ such that every $\psi_i$ of rank at least $d + 1$ has $j \in [\ell]$ with the property that $\pi_j \circ \psi_i|_{Y_j}$ is injective.\\ 

    We simply pick $(Y_i, \pi_i)$ uniformly at random at each step, ignoring the $\psi_i$ that already have a good pair of subspace and projection. At each step, we shall decrease the size of the list by a factor of two. Write $Y = Y_i, \pi = \pi_i$.\\

    Note that the probability that $\psi_i|_{Y}$ has rank $R+1$ is the same as the probability that $Y \cap \ker \psi_i = 0$. If $\on{rank} \psi_i \geq d$, this probability is at least $1 - 2 p^{-d + R + 1}$.\\

    Furthermore, the probability that $\pi \circ \psi_i|_Y$ is not injective is at most $p^{R+1} \mathbb{P}(\pi(a) = 0) \leq p^{-d + R + 1}$, where $a \not= 0$ is arbitrary. Hence, for each $i$, the probability that $(Y, \pi)$ has the desired properties for the map $\psi_i$ (of rank at least $d + 1$), is at least $1 - 3 p^{-d + R + 1} \geq 1/2$, provided we take $d = R + 10$. Hence, we may find the required sequence of projections and subspaces for $\ell \leq \log_2 m $.\\

    Finally, pick $B'' = \{b \in B' : (\forall i \in [\ell]) \pi_i \circ \phi_b|_Y = \rho_i\}$, where $\rho_i : Y_i \to \mathbb{F}_p^d$ is randomly and uniformly chosen for each $i \in [\ell]$. By linearity of expectation, there is such a choice of maps $\rho_1, \dots, \rho_\ell$ such that $|B''| \geq p^{-\ell d (R + 1)}|B'|$.\\ 
    
    Let us now show that $B''$ has all its additive 16-tuples $O(r)$-respected. Namely, given an additive 16-tuple $(b_1, \dots, b_{16})$ in $B''$, take the linear map $\psi_j$ such that 
    \begin{equation}\on{rank}\Big(\sum_{i \in [16]} (-1)^i \phi_{b_i} - \psi_j\Big) \leq R.\label{lowrankapprox}\end{equation}

    If $\on{rank} \psi_j \leq d$, then the 16-tuple is $(2R + 10)$-respected, as required. If it happens that $\on{rank} \psi_j \geq d + 1$, then there exists a pair $(Y_k, \pi_k)$ in the list of subspaces and projections that is suitable for $\psi_j$. Due to the rank bound~\eqref{lowrankapprox} and the fact that $\dim Y_k > R$, we may find a non-zero $y \in Y_k$ such that
    \[\Big(\sum_{i \in [16]} (-1)^i \phi_{b_i} - \psi_j\Big)(y) = 0.\]
    Applying $\pi_k$, we still have
    \[\pi_k \circ \Big(\sum_{i \in [16]} (-1)^i \phi_{b_i} - \psi_j\Big)(y) = 0.\]
    By the definition of $B''$, it follows that
    \[\pi_k \circ \Big(\sum_{i \in [16]} (-1)^i \phi_{b_i}\Big) (y) = \sum_{i \in [16]} (-1)^i \rho_k(y) = 0,\]
    so finally $\pi_k \circ \psi_j(y) = 0$. Recalling that $y \in Y_k \setminus \{0\}$ and the $\pi_k \circ \psi_j$ is injective on $Y_k$, we reach a contradiction.
\end{proof}

This completes the proof of the proposition. \end{proof}

\section{Full system of linear maps}

In this section, we use the robust Bogolyubov-Ruzsa theorem (Theorem~\ref{rbrlemma}), to replace the given system of linear maps with a new one in which we have a map $\phi'_a$ for all elements $a$ of a suitable subspace, and all additive quadruples are still $O(r)$-respected.

\begin{proposition}\label{fullspacesystem}
    Suppose that we are given linear maps $\phi_x \colon G \to H$, $x \in X$, such that $\on{rank} \phi_{x_1} - \phi_{x_2} + \phi_{x_3} - \phi_{x_4} \leq r$ holds for at least $c |G|^3$ additive quadruples in $X$. Then there exists a subspace $U \leq G$ of codimension $O(r^{O(1)} \log^{O(1)} c^{-1})$ and a collection of linear maps $\phi'_x : G \to H$, $x \in U$, such that 
    \[\on{rank}\Big(\sum_{i \in [4]} (-1)^i \phi'_{x_i}\Big) \leq O(r)\]
    holds for all additive quadruples $x_{[4]}$ in $U$ and for each $u \in U$, there are at least $\exp(-O(r^{O(1)} \log^{O(1)} c^{-1})) |G|^3$ choices of $(x_1, x_2, x_3) \in X^3$ such that $x_1 + x_2 - x_3 - u \in X$ and
    \[\on{rank} \Big(\phi'_u - \phi_{x_1} - \phi_{x_2} + \phi_{x_3} - \phi_{x_1 + x_2 - x_3 - u}\Big) \leq O(r).\]
\end{proposition}

As in the previous section, we shall eventually use the relationship between $\phi'_u$ and $\phi_{x}$ in the conclusion to transfer structural information on the system $(\phi'_u)_{u \in U}$ to the system $(\phi_{x})_{x \in X}$.

\begin{proof}
    Applying Proposition~\ref{absBSGFirst}, we may find a subset $\tilde{X}$ of size $\exp(-O(r^{O(1)} \log^{O(1)} c^{-1})) |G|$ such that 
    \[\on{rank}\Big(\sum_{i \in [16]} \phi_{x_i}\Big) \leq r'\]
    holds for all additive 16-tuples $x_{[16]}$ in $\tilde{X}$, for some $r' \leq O(r)$. Apply Theorem~\ref{rbrlemma} to the set $\tilde{X}$ to find a subspace $U \leq G$ of codimension at most $O(r^{O(1)} \log^{O(1)} c^{-1})$ such that every $u \in U$ can be written as $u = -x_1 + x_2 - x_3 + x_4$ for at least $\exp(-O(r^{O(1)} \log^{O(1)} c^{-1})) |G|^3$ quadruples $(x_1, x_2, x_3, x_4) \in \tilde{X}^4$.\\
    \indent For each $u \in U$, take an arbitrary quadruple in $\tilde{X}$ such that $u = -x_1 + x_2 - x_3 + x_4$ and define $\phi'_u$ as $-\phi_{x_1} + \phi_{x_2} - \phi_{x_3} + \phi_{x_4}$. It remains to check the properties of the system of linear maps $\phi'_u$ in the claim.\\

    Suppose that $-u_1 + u_2 - u_3 + u_4 = 0$ for some elements of $U$. By the definition of $\phi'_{u_i}$, there are quadruples $(x_{i\,1}, x_{i\,2}, x_{i\,3}, x_{i\,4}) \in \tilde{X}^4$ for each $i \in [4]$, such that $u_i = \sum_{j \in [4]} (-1)^j x_{i\,j}$ and $\phi'_{u_i} = \sum_{j \in [4]} (-1)^j \phi_{x_{i\,j}}$. Hence, 
    \[\sum_{i \in [4], j \in [4]} (-1)^{i + j} x_{i \, j} = \sum_{i \in [4]} (-1)^i u_i = 0,\]
    implying that $x_{[4]\times [4]}$ form an additive 16-tuple in $\tilde{X}$, and thus 
    \[r' \geq \on{rank}\Big(\sum_{i \in [4], j \in [4]} (-1)^{i + j} \phi_{x_{i \, j}} \Big)= \on{rank} \Big(\sum_{i \in [4]} (-1)^i \phi'_{u_i}\Big),\]
    showing that additive quadruples in $U$ are $O(r)$-respected.\\

    Finally, let $y_1, y_2, y_3, y_4 \in \tilde{X}$ be arbitrary elements such that $-y_1 + y_2 - y_3 + y_4 = u$, and let $x_1, x_2, x_3, x_4$ be those elements used to define $\phi'_u$. Then $\sum_{i \in [4]} (-1)^i x_i - \sum_{i \in [4]} (-1)^i y_i = u-u = 0$ and (as 8-tuples are also $r'$-respected in $\tilde{X}$)
    \[r' \geq \on{rank}\Big(\sum_{i \in [4]} (-1)^i \phi_{x_i} - \sum_{i \in [4]} (-1)^i \phi_{y_i}\Big) = \on{rank}\Big(\phi'_u - \sum_{i \in [4]} (-1)^i \phi_{y_i}\Big).\qedhere\]
\end{proof}

\section{Changing the category---system of partially defined maps}

In this section, we make a change of viewpoint and instead of considering global linear maps $\phi_x$ with additive quadruple being respected if the condition 
\[\on{rank}\Big(\sum_{i \in [4]} (-1)^i \phi_{x_i}\Big) \leq r\]
holds, we pass to partially defined linear maps $\phi:U_x \to H$, but with additive quadruples said to be \textit{subspace-respected} if they satisfy a stronger condition that
\[\Big(\phi_{x_1} + \phi_{x_2} - \phi_{x_3} - \phi_{x_4}\Big)|_{U_{x_1} \cap U_{x_2} \cap U_{x_3} \cap U_{x_4}} = 0.\]

We find suitable subspaces using another application of dependent random choice, and we rely on the second moment method to obtain the required probability estimates. For the argument to work, we need to have a sufficiently quasirandom system. Fortunately, if the system is not quasirandom, we may complete argument easily, which is reflected in the conclusion \textbf{(i)} of the proposition below.

\begin{proposition} \label{subspacePass}
    Suppose that $\phi_x : G \to H$ is a linear map for each $x \in G$ such that 
    \[\on{rank}\Big(\sum_{i \in [2k]} (-1)^i \phi_{x_i}\Big) \leq r\]
    holds for all additive $2k$-tuples $x_{[2k]}$. Let $\varepsilon > 0$. Then one of the following holds.
    \begin{itemize}
        \item[\textbf{(i)}] There exists a linear map $\psi : G \to H$ such that $\on{rank}(\phi_x - \psi) \leq O(k^2 + kr + k\log \varepsilon^{-1})$ holds for at least $\Big(\frac{\varepsilon}{k}\Big)^{O(k)}p^{-O(kr)}|G|$ elements $x \in G$.
        \item[\textbf{(ii)}] There exist subspaces $U_x$, $x \in G$, of codimension at most $O(k + r + \log \varepsilon^{-1})$, such that, for each for $\ell \leq k$, $1-\varepsilon$ proportion of all additive $(2\ell)$-tuples $x_{[2\ell]}$ in $G$ satisfy
        \[\sum_{i \in [2\ell]} (-1)^i\phi_{x_i} = 0\]
        on $\cap_{i \in [2\ell]} U_{x_i}$.
    \end{itemize}
\end{proposition}

\begin{proof}
    Let $m \in \mathbb{N}$ and $\eta > 0$ be two parameter to be chosen later. Take random elements $a_1, \dots, a_m$ uniformly and independently at random from $H$ and define the subspace $U_x = \{y \in G : (\forall i \in [m])\,a_i \cdot \phi_x(y) = 0\}$ for each $x \in G$, which has codimension at most $m$. Suppose that the conclusion \textbf{(i)} fails, namely that,
    \begin{align}\text{for any linear map }&\psi : G \to H,\text{ inequality }\on{rank}(\phi_x - \psi) \leq 8k + 2\log_p\eta^{-1} + 10\nonumber\\
    &\text{holds for at most }p^{-8k + 4}\eta^2|G|\text{ of }x \in G.\label{failurei}\end{align}

    We begin the proof by using the failure of \textbf{(i)} to deduce frequent independence of elements obtained by maps $\phi_x$.

    \begin{claim} Let $2 \leq \ell \leq k$.\label{independenceclaims}
        \begin{itemize}
            \item[\textbf{(i)}] For all but $\eta |G|^{2\ell}$ choices of $(x_1, \dots, x_{2\ell - 1}, z) \in G^{2\ell}$ the $(2\ell-1)$-tuple $(\phi_{x_i}(z))_{i \in [2\ell - 1]}$ is linearly independent.
            \item[\textbf{(ii)}] There are at most $\eta |G|^{2\ell}$ choices of $(x_1, \dots, x_{2\ell}, z) \in G^{2\ell + 1}$ such that $\sum_{i \in [2\ell]} (-1)^i x_i = 0$, $\sum_{i \in [2\ell]} (-1)^i \phi_{x_i}(z) \not= 0$ and the $(2\ell)$-tuple $(\phi_{x_i}(z))_{i \in [2\ell]}$ is linearly dependent.
            \item[\textbf{(iii)}] For all but $\eta |G|^{2\ell + 1}$ choices of $(x_1, \dots, x_{2\ell - 1}, z, z') \in G^{2\ell + 1}$ the $(4\ell-2)$-tuple $(\phi_{x_i}(z), \phi_{x_i}(z'))_{i \in [2\ell - 1]}$ is linearly independent. 
        \end{itemize}
    \end{claim}

    \begin{proof}
        \textbf{Proof of \textbf{(i)}.} Assume the contrary. By averaging over possible non-trivial linear combinations, we get $\lambda_1, \dots, \lambda_{2\ell - 1} \in \mathbb{F}_p$, not all zero, such that 
        \[\lambda_1 \phi_{x_1}(z) + \dots \lambda_{2\ell - 1} \phi_{x_{2\ell-1}}(z) = 0\]
        holds for at least $p^{-2\ell + 1} \eta |G|^{2\ell}$ choices of $x_1, \dots, x_{2\ell - 1}, z$. Without loss of generality $\lambda_1 = 1$. By averaging over $x_2, \dots, x_{2\ell - 1}$, we get a linear map $\psi = -(\lambda_2 \phi_{x_2} + \dots + \lambda_{2\ell -1 } \phi_{2\ell - 1})$ such that
        \[\phi_{x_1}(z) = \psi(z)\]
        holds for at least $p^{-2\ell + 1} \eta |G|^2$ pairs $(x_1, z)$. Hence 
        \[\on{rank}(\phi_{x} - \psi) \leq 2\ell + 2 + \log_p \eta^{-1}\]
        holds for at least $p^{-2\ell} \eta |G|$ elements $x \in G$, which is a contradiction with failure of~\eqref{failurei}.\\

        \textbf{Proof of (ii).} The proof is similar to the previous one. Assume the contrary. By averaging over possible non-trivial linear combinations, we get $\lambda_1, \dots, \lambda_{2\ell} \in \mathbb{F}_p$, not all zero, such that 
        \[\lambda_1 \phi_{x_1}(z) + \dots \lambda_{2\ell} \phi_{x_{2\ell}}(z) = 0\]
        holds for at least $p^{-2\ell} \eta |G|^{2\ell}$ choices of $x_1, \dots, x_{2\ell}, z$ with the additional properties described in the claim. As $\sum_{i \in [2\ell]} (-1)^i \phi_{x_i}(z) \not= 0$, we see that $\lambda \notin \langle (-1, 1, \dots, -1, 1)\rangle$. Hence, for at least $\frac{1}{2}p^{-2\ell} \eta |G|^{2\ell-1}$ of additive $2\ell$-tuples $x_1, \dots, x_{2\ell}$ we have 
        \[\on{rank}\Big(\sum_{i \in [2\ell]} \lambda_i \phi_{x_i}\Big) \leq \log_p \eta^{-1} + 2\ell + 1.\]
        But as all additive $2\ell$-tuples are $r$-respected, we conclude that
        \[\on{rank}\Big(\sum_{i \in [2\ell - 1]} (\lambda_i - (-1)^i \lambda_{2\ell}) \phi_{x_i}\Big) \leq r + \log_p \eta^{-1} + 2\ell + 1\]
        holds for at least $\frac{1}{2}p^{-2\ell} \eta |G|^{2\ell-1}$ of $(2\ell-1)$-tuples $(x_1, \dots, x_{2\ell-1})$. As in the part \textbf{(i)}, average over all $x_i$ but one to get a contradiction with~\eqref{failurei}.\\

        \textbf{Proof of (iii).} Again, suppose the contrary and average to find $\lambda_1, \dots, \lambda_{2\ell - 1},$ $\mu_1, \dots, \mu_{2\ell - 1}$, not all zero, such that 
        \[\sum_{i \in [2\ell - 1]} \lambda_i \phi_{x_i}(z) + \sum_{i \in [2\ell - 1]} \mu_i \phi_{x_i}(z') = 0\]
        holds for at least $p^{-4\ell + 2}\eta |G|^{2\ell + 1}$ choices of $(x_1, \dots, x_{2\ell - 1}, z, z') \in G^{2\ell + 1}$. Without loss of generality $\lambda_1 \not=0$. By the Cauchy-Schwarz inequality we can find $p^{-8\ell + 4}\eta^2 |G|^{2\ell + 1}$ choices of $(x_1, \dots, x_{2\ell - 1}, z, z') \in G^{2\ell + 1}$ such that 
        \[0 = \sum_{i \in [2\ell - 1]} \lambda_i \phi_{x_i}(z) - \sum_{i \in [2\ell - 1]} \lambda_i \phi_{x_i}(z') = \sum_{i \in [2\ell - 1]} \lambda_i \phi_{x_i}(z - z').\]

        Hence, for at least $p^{-8\ell + 4}\eta^2 |G|^{2\ell - 1}$ choices of $(x_1, \dots, x_{2\ell - 1}) \in G^{2\ell - 1}$ we have the bound
        \[\on{rank}\Big(\sum_{i \in [2\ell - 1]} \lambda_i \phi_{x_i}\Big) \leq 8\ell + 2\log_p \eta^{-1}.\]
        As before, after averaging, we get a contradiction with~\eqref{failurei}.
    \end{proof}

    We now study the probability that the kernel of the map $\sum_{i \in [2\ell]} (-1)^i \phi_{x_i}$ is contained in the subspace $\cap_{i \in [2\ell]} U_{x_i}$, where $x_1, \dots, x_{2\ell}$ is an additive $2\ell$-tuple, by which we mean $\sum_{i \in [2\ell]}(-1)^i x_i = 0$. This is done in two steps, represented by the following two claims. Probabilities of events in the claims below are denoted by $\mathbb{P}_a(\bcdot)$, suggesting that our probability space stems from the random choice of elements $a_1, \dots, a_m$. The fact that $2\ell$-tuples are $O(r)$-respected plays a central role.

    \begin{claim} \label{KinterUsizeClaim}
        For all but at most $\sqrt{\eta}|G|^{2\ell - 1}$ additive $2\ell$-tuples $(x_1, \dots, x_{2\ell})$ we have
        \[\mathbb{P}_{a}\Big(\Big|K \cap \bigcap_{i \in [2\ell]} U_{x_i}\Big| = p^{-(2\ell - 1)m} |K| \Big) \geq 1 - 36 p^{2r + (4\ell + 2)m} \sqrt{\eta},\]
        where $K = \on{ker} \Big(\sum_{i \in [2\ell]} (-1)^i \phi_{x_i}\Big)$, provided $\eta \leq \frac{1}{6^4}p^{-4(2\ell - 1)m}$.
    \end{claim}

    \begin{proof}
        By Claim~\ref{independenceclaims} \textbf{(iii)} for all but $\sqrt{\eta} |G|^{2\ell - 1}$ of additive $2\ell$-tuples $(x_1, \dots, x_{2\ell})$ we have $4\ell - 2$ elements
        \begin{equation}\label{secondorderindep}\Big(\phi_{x_i}(z), \phi_{x_i}(z')\Big)_{i \in [2\ell - 1]}\end{equation}
        linearly independent for all but $\sqrt{\eta} |G|^2$ choices of $(z,z') \in G^2$. Fix any such  additive $2\ell$-tuple $(x_1, \dots, x_{2\ell})$. We show that it satisfy the bound on the probability in the statement.\\

        Note that $K \cap \bigcap_{i \in [2\ell]} U_{x_i} = K \cap \bigcap_{i \in [2\ell-1]} U_{x_i}$. Indeed, if $z \in K \cap \bigcap_{i \in [2\ell-1]} U_{x_i}$, then
        \[a_j \cdot \phi_{x_{2\ell}}(z) = a_j \cdot \Big(\sum_{i \in [2\ell - 1]} (-1)^{i + 1} \phi_{x_i}(z)\Big) = 0\]
        for each $j \in [m]$. So we need to understand the distribution of $|K \cap \bigcap_{i \in [2\ell-1]} U_{x_i}|$.\\

        We use the second moment method. Firstly, we have
        \begin{align*}
            \exx_a |K \cap \bigcap_{i \in [2\ell-1]} U_{x_i}| = & \sum_{z \in K} \mathbb{P}_a\Big(z \in \bigcap_{i \in [2\ell-1]} U_{x_i}\Big)\\
            = & \sum_{z \in K} \mathbb{P}_a\Big((\forall i \in [2\ell - 1], j \in [m])\,\, a_j \cdot \phi_{x_i}(z) = 0\Big)\\
            \geq & p^{-(2\ell - 1)m} |K|.
        \end{align*}

        Secondly, relying on the fact that elements in~\eqref{secondorderindep} are independent for all but at most $\sqrt{\eta}|G|^2$ choices of $z, z'$, we have
        \begin{align*}
            \exx_a |K \cap \bigcap_{i \in [2\ell-1]} U_{x_i}|^2 = & \sum_{z, z' \in K} \mathbb{P}_a\Big(z, z' \in \bigcap_{i \in [2\ell-1]} U_{x_i}\Big)\\
            = & \sum_{z, z' \in K} \mathbb{P}_a\Big((\forall i \in [2\ell - 1], j \in [m])\,\, a_j \cdot \phi_{x_i}(z) = 0,\, a_j \cdot \phi_{x_i}(z') = 0\Big)\\
            \leq & p^{-2 (2\ell - 1)m} |K|^2 + \sqrt{\eta}|G|^2.
        \end{align*}

        Write $N$ for the random variable $|K \cap \bigcap_{i \in [2\ell-1]} U_{x_i}|$. It follows that 

        \[\on{var} N = \exx_a N^2 - \Big(\exx_a N\Big)^2 \leq \sqrt{\eta} |G|^2\]

        and
        \[\Big(\exx_a N\Big)^2 \leq p^{-2 (2\ell - 1)m} |K|^2 + \sqrt{\eta}|G|^2,\]

        so, $\ex_a N \leq p^{-(2\ell - 1)m} |K| + \sqrt[4]{\eta}|G|$. Provided $\eta \leq \frac{1}{6^4}p^{-4(2\ell - 1)m}$, we have $\Big|\ex_a N - p^{-(2\ell - 1)m} |K|\Big| \leq \frac{1}{6}p^{-(2\ell - 1)m} |K|$.\\

        By Chebyshev's inequality, and the fact that $K$ has codimension at most $r$ in $G$,

        \begin{align*}&\mathbb{P}_a\Big(\Big|N - p^{-(2\ell - 1)m} |K|\Big| \geq \frac{1}{3}p^{-(2\ell - 1)m} |K|\Big) \leq \mathbb{P}_a\Big(\Big|N - \ex N\Big| \geq \frac{1}{6}p^{-(2\ell - 1)m} |K|\Big)\\
        &\hspace{2cm}\leq \frac{36 \sqrt{\eta} |G|^2}{p^{-(4\ell - 2)m} |K|^2} \leq 36 p^{2r + (4\ell + 2)m} \sqrt{\eta}.\end{align*}

        Finally, since $K \cap \bigcap_{i \in [2\ell-1]} U_{x_i}$ and $K$ are vector spaces, and hence their sizes are powers of $p$,  the condition $\Big|N - p^{-(2\ell - 1)m} |K|\Big| < \frac{1}{3}p^{-(2\ell - 1)m} |K|$ implies that $N = p^{-(2\ell - 1)m} |K|$.        
    \end{proof}

    \begin{claim} \label{subspaceschoiceclaim}
        Suppose that $m \geq r + 4$ and $\eta \leq \frac{1}{4}p^{-8m}$. For all but at most $3\sqrt{\eta}|G|^{2\ell - 1}$ additive $2\ell$-tuples $(x_1, \dots, x_{2\ell})$ we have
        \[\mathbb{P}_{a}\Big(\bigcap_{i \in [2\ell]} U_{x_i} \subseteq K \Big) \geq 1 - \Big(24 p^{r - m} + 216 p^{2r + (6\ell + 3)m} \sqrt{\eta}\Big),\]
        where $K = \on{ker} \Big(\sum_{i \in [2\ell]} (-1)^i \phi_{x_i}\Big)$.
    \end{claim}

    \begin{proof}
        Note that the inequality $\Big|K \cap \bigcap_{i \in [2\ell]} U_{x_i}\Big| \geq \frac{2}{3}\Big|\bigcap_{i \in [2\ell]} U_{x_i}\Big|$ implies that the intersection $\bigcap_{i \in [2\ell]} U_{x_i}$ is contained inside $K$, so we estimate the probability of that inequality being true.\\ 

        Fix any additive $2\ell$-tuple $(x_1, \dots, x_{2\ell})$ such that
        \begin{itemize}
            \item for all but at most $\sqrt{\eta} |G|$ of $z\in G$, the elements $\Big(\phi_{x_i}(z)\Big)_{i \in [2\ell - 1]}$ are linearly independent,
            \item for all but at most $\sqrt{\eta} |G|$ of $z\in G \setminus K$, the elements $\Big(\phi_{x_i}(z)\Big)_{i \in [2\ell]}$ are linearly independent,
            \item $\mathbb{P}_{a}\Big(\Big|K \cap \bigcap_{i \in [2\ell]} U_{x_i}\Big| = p^{-(2\ell - 1)m} |K| \Big) \geq 1 - 9 p^{2r + (4\ell + 2)m} \sqrt{\eta}$,
        \end{itemize}
        where $K = \on{ker} \Big(\sum_{i \in [2\ell]} (-1)^i \phi_{x_i}\Big)$. By Claims~\ref{independenceclaims} and~\ref{KinterUsizeClaim}, these properties hold for all but at most $3\sqrt{\eta}|G|^{2\ell-1}$ additive $2\ell$-tuples.\\

        Observe that if $z \in K$ and $\Big(\phi_{x_i}(z)\Big)_{i \in [2\ell - 1]}$ are independent, then 
        \[\mathbb{P}_a\Big(z \in \bigcap_{i \in [2\ell]} U_{x_i}\Big) = \mathbb{P}_a\Big(z \in \bigcap_{i \in [2\ell-1]} U_{x_i}\Big) = p^{-(2\ell - 1)m}.\]
        On the other hand, if $z \notin K$ and $\Big(\phi_{x_i}(z)\Big)_{i \in [2\ell]}$ are independent, then
        \[\mathbb{P}_a\Big(z \in \bigcap_{i \in [2\ell]} U_{x_i}\Big) = p^{-2\ell m}.\]

        Hence, we have the inequality
        \begin{align}&\exx_a\Big(\Big|K \cap \bigcap_{i \in [2\ell]} U_{x_i}\Big| - (1 - 2p^{r - m})\Big|\bigcap_{i \in [2\ell]} U_{x_i}\Big|\Big)\label{expectpos}\\
        &\hspace{2cm}\geq \sum_{z \in K} 2p^{r-m}\mathbb{P}_a\Big(z \in \bigcap_{i \in [2\ell]} U_{x_i}\Big) - \sum_{z \notin K} \mathbb{P}_a\Big(z \in \bigcap_{i \in [2\ell]} U_{x_i}\Big) \nonumber\\
        &\hspace{2cm}\geq 2p^{r-4m}|K| - p^{-4m}|G| - 2\sqrt{\eta}|G|\nonumber\\
        &\hspace{2cm}\geq p^{-4m}|G| - 2\sqrt{\eta}|G|\nonumber\\
        &\hspace{2cm}\geq 0,\nonumber\end{align}

        as long as $\eta \leq \frac{1}{4}p^{-8m}$. We now use this bound to derive the desired probability.\\

        Let $A$ be the event that $\Big|K \cap \bigcap_{i \in [2\ell]} U_{x_i}\Big| \geq \frac{2}{3}\Big|\bigcap_{i \in [2\ell]} U_{x_i}\Big|$ and let $B$ be the event that $\Big|K \cap \bigcap_{i \in [2\ell]} U_{x_i}\Big| = p^{-(2\ell - 1)m}|K|$. Recall that we are interested in $\mathbb{P}_a(A)$. Observe that if $B$ holds, then  
        \begin{align*}&\Big|K \cap \bigcap_{i \in [2\ell]} U_{x_i}\Big| - (1 - 2p^{r - m})\Big|\bigcap_{i \in [2\ell]} U_{x_i}\Big| \\
        &\hspace{2cm}\leq \Big|K \cap \bigcap_{i \in [2\ell]} U_{x_i}\Big| - (1 - 2p^{r - m})\Big|K \cap \bigcap_{i \in [2\ell]} U_{x_i}\Big|\\
        &\hspace{2cm}\leq 2p^{r - m} \Big|K \cap \bigcap_{i \in [2\ell]} U_{x_i}\Big|\\
        &\hspace{2cm}= 2p^{r - 2\ell m} |K|.\end{align*}
        If additionally $A$ fails, then 
        \begin{align*}&\Big|K \cap \bigcap_{i \in [2\ell]} U_{x_i}\Big| - (1 - 2p^{r - m})\Big|\bigcap_{i \in [2\ell]} U_{x_i}\Big| \\
         &\hspace{2cm}\leq \Big(\frac{2}{3} - (1 - 2p^{r-m})\Big)\Big|\bigcap_{i \in [2\ell]} U_{x_i}\Big|\\
         &\hspace{2cm}\leq -\frac{1}{12}\Big|K \cap \bigcap_{i \in [2\ell]} U_{x_i}\Big|\\
         &\hspace{2cm}= -\frac{1}{12} p^{-(2\ell - 1)m}|K|,
        \end{align*}
        provided $m \geq r + 4$. Finally, we have the trivial bound 
        \[\Big|K \cap \bigcap_{i \in [2\ell]} U_{x_i}\Big| - (1 - 2p^{r - m})\Big|\bigcap_{i \in [2\ell]} U_{x_i}\Big| \leq |K|\]
        which always holds, even when the event $B$ does not occur.\\
        
        Hence, the expectation in~\eqref{expectpos}, which is non-negative, can be bound from above by
        \begin{align*} & \mathbb{P}_a(A \cap B) 2p^{r - 2\ell m} |K|\,\, - \,\,\mathbb{P}_a(A^c \cap B)\frac{1}{12} p^{-(2\ell - 1)m}|K|\,\, + \,\,\mathbb{P}_a(B^c) |K|\\
         &\hspace{2cm}\leq  \Big(2p^{r - 2\ell m} + 18 p^{2r + (4\ell + 2)m} \sqrt{\eta}- \,\,\mathbb{P}_a(A^c)\frac{1}{12} p^{-(2\ell - 1)m}\Big) |K|.\end{align*}

         Hence, $\mathbb{P}_a(A^c) \leq 24 p^{r - m} + 216 p^{2r + (6\ell + 3)m} \sqrt{\eta}$, completing the proof.
    \end{proof}

    To complete the proof of Proposition~\ref{subspacePass}, set $m = r + \log_p(240k\varepsilon^{-1})$ and $\eta = 10^{-10}\varepsilon^2 k^{-2} p^{-2r - (12k + 6)m}$. With this choice of parameters, Claim~\ref{subspaceschoiceclaim} shows that, for each $\ell \in [2,k]$, for all but at most $\frac{\varepsilon}{2}|G|^{2\ell - 1}$ additive $2\ell$-tuples $(x_1, \dots, x_{2\ell})$ we have $\bigcap_{i \in [2\ell]} U_{x_i} \subseteq K$ with probability at least $1 - \varepsilon/2k$. Write $X_\ell$ for the set of such additive $2\ell$-tuples. Let $F_\ell$ be the random variable that counts those additive $2\ell$-tuples in $X_\ell$ for which $\bigcap_{i \in [2\ell]} U_{x_i} \not\subseteq K$. Then 
    \[\exx_a |G|^{2k - 4} F_2 + |G|^{2k - 6}F_3 + \dots + F_k \leq \frac{\varepsilon}{2}|G|^{2k - 1},\]
    so there is a choice of $a_1, \dots, a_m$ such that, for each $\ell \in [2, k]$, $1-\varepsilon$ proportion of all additive $(2\ell)$-tuples $x_{[2\ell]}$ in $G$ satisfy $\bigcap_{i \in [2\ell]} U_{x_i} \subseteq K$. Since $K = \on{ker} \Big(\sum_{i \in [2\ell]} (-1)^i \phi_{x_i}\Big)$, this means that
    \[\sum_{i \in [2\ell]} (-1)^i\phi_{x_i} = 0\]
    on $\cap_{i \in [2\ell]} U_{x_i}$, as desired.
\end{proof}

\section{Making domain a bilinear variety}

\hspace{17pt}As explained in Section~\ref{detailsoversec}, we use a variant of the bilinear Bogolyubov argument to find another system of partially defined linear maps whose domains are columns of a bilinear variety. The facts that we start with a system of maps whose index set is a full subspace and that almost all additive 8-tuples are subspace-respected are key.\\
\indent In order to get good bounds, we rely on an observation due to Hosseini and Lovett~\cite{HosseiniLovett}, which is that, even though we generate a long list of certain affine maps throughout the proof, only a small number of those plays a role for each subspace $U_x$ and a simple pigeonhole argument allows us to recover the desired bounds.\\
\indent Let us remark that we misuse the notation slightly in the statement of the proposition below and write $\on{rank}(\psi_a + \phi_x - \phi_{x + a})$ for some linear map $\psi_a : G \to H$ and some partially defined linear maps $\phi_x : U_x \to H, \phi_{x + a} : U_{x + a} \to H$. We only need an upper bound on this rank, so we take arbitrary extensions of $\phi_x$ and $\phi_{x + a}$ to $G$ to have a meaningful expression. Equivalently, we show that $\psi_a + \phi_x - \phi_{x + a} = 0$ on a dense subset of $U_x \cap U_{x  + a}$.

\begin{proposition}\label{bilvardomain}
    Let $\varepsilon = 2^{-24}$. Let $\phi_x : U_x \to H$ be linear map, on a subspace $U_x\leq G$ of codimension $d$ for each $x \in G$. Suppose that $1-\varepsilon$ proportion of additive $2\ell$-tuples, for $\ell \in \{2,3,4\}$, is subspace-respected,  in the sense that
    \[\Big(\sum_{i \in [2\ell]} (-1)^i\phi_{x_i}\Big)|_{\cap_{i \in [2\ell]} U_{x_i}} = 0.\]
    Then, there exist quantities $d' \leq O(d)$ and $c \geq p^{-O(d^2)}$, affine maps $\theta_1, \dots, \theta_{d'} :G \to G$, a set $X \subseteq G$, linear maps $\psi_x : G \to H$ for $x \in G$ such that, defining subspaces $B_x = \langle \theta_1(x), \dots, \theta_{d'}(x)\rangle^\perp$,
    \begin{itemize}
        \item for each $a \in X$, bound $\on{rank}(\psi_a + \phi_x - \phi_{x + a}) \leq d'$ holds for at least $\frac{1}{2}|G|$ choices of $x \in G$,
        \item we have
        \[\Big(\sum_{i \in [4]} (-1)^i\psi_{x_i}\Big)|_{\cap_{i \in [4]} B_{x_i}} = 0.\]
        for at least $c |G|^3$ of additive quadruples $x_{[4]}$ in $X$.
    \end{itemize}
\end{proposition}

We need the following lemma, which is a slight generalization of the approximate homomorphisms structure theorem. Instead of being about additive quadruples, it is about somewhat more involved arithmetic configurations of Cauchy-Schwarz complexity 1. The configurations present in the lemma will appear naturally later in the proof.

\begin{lemma}\label{restrictedfreiman} Suppose that $\psi_{[8]}, \psi'_{[8]}$ are maps from $G$ to $G$. Suppose that we have a collection $Q$ of parameters $(a_1, a_2, a_3, x_1, x_2, x_3, x_4, y_1, y_2, y_3, y_4)$ of size at least $c|G|^{11}$ each of which satisfies
\begin{align*}&\psi_1(x_1 + a_1) - \psi_2(x_1) + \psi_3(x_2 + a_2) - \psi_4(x_2) + \psi_5(x_3 + a_3) - \psi_6(x_3) + \psi_7(x_4 + a_1 + a_2 - a_3) - \psi_8(x_4)\\
&\hspace{1cm}=\psi'_1(y_1 + a_1) - \psi'_2(y_1) + \psi'_3(y_2 + a_2) - \psi'_4(y_2) + \psi'_5(y_3 + a_3) - \psi'_6(y_3) + \psi'_7(y_4 + a_1 + a_2 - a_3) - \psi'_8(y_4).\end{align*}

Then there exist affine maps $\theta_{[8]}$ from $G$ to $G$ and elements $u_{[8]}$ in $G$, such that the following 16 equalities hold
\begin{align*}&\psi_1(x_1 + a_1) = \theta_1(x_1 + a_1), \psi_2(x_1) = \theta_1(x_1) + u_1, \psi_3(x_2 + a_2) = \theta_2(x_2 + a_2), \psi_4(x_2) = \theta_2(x_2) + u_2,\\
&\psi_5(x_3 + a_3) = \theta_3(x_3 + a_3), \psi_6(x_3) = \theta_3(x_3) + u_3,\\
&\hspace{3cm}\psi_7(x_4 + a_1 + a_2 - a_3) = \theta_4(x_4 + a_1 + a_2 - a_3), \psi_8(x_4) = \theta_4(x_4) + u_4,\\
&\psi'_1(y_1 + a_1) = \theta_5(y_1 + a_1), \psi'_2(y_1) = \theta_5(y_1) + u_5, \psi'_3(y_2 + a_2) = \theta_6(y_2 + a_2), \psi'_4(y_2) = \theta_6(y_2) + u_6,\\
&\psi'_5(y_3 + a_3) = \theta_7(y_3 + a_3), \psi'_6(y_3) = \theta_7(y_3) + u_7,\\
&\hspace{3cm}\psi'_7(y_4 + a_1 + a_2 - a_3) = \theta_8(y_4 + a_1 + a_2 - a_3), \psi'_8(y_4) = \theta_8(y_4) + u_8,\end{align*}
for at least $(c/2)^{O(1)}|G|^{11}$ of 11-tuples in $Q$.
\end{lemma}

\begin{proof}
    We first show that each $\psi_i$ can be assumed to be affine, and then we conclude that some of the affine maps are closely related.\\
    Since $|Q| \geq c |G|^{11}$, the set $Z$ of values $z$ such that $z = x_1 + a_1$ for at least $\frac{c}{2}|G|^{10}$ 11-tuples $(a_1, \dots, y_4)$ in $Q$ has size at least $\frac{c}{2}|G|$. Pass to those 11-tuples $Q'$ whose $x_1 + a_1 \in Z$, so $|Q'| \geq \frac{c^2}{4}|G|^{11}$. By Cauchy-Schwarz inequality, we have at least $\frac{c^4}{16}|G|^{12}$ choices of $(a_1, a_2, a_3, x_1, x_1', x_2, \dots, y_4)$ such that 11-tuples $(a_1, \dots, y_4)$ and $(a_1, a_2, a_3, x_1', x_2, \dots, y_4)$ both belong to $Q'$, and in particular $x_1 + a_1, x'_1 + a_1 \in Z$. Subtracting equalities for these 11-tuples from one another we get
    \[\psi_1(x_1 + a_1) - \psi_2(x_1) = \psi_1(x'_1 + a_1) - \psi_2(x'_1).\]
    Hence, this equality holds for at least $\frac{c^4}{16}|G|^{3}$ choices of $x_1, x'_1, a_1$ and additionally $x_1 + a_1, x'_1 + a_1 \in Z$. Another Cauchy-Schwarz step shows that
    \[\psi_1(x_1 + a_1) - \psi_1(x_1 + b_1) = \psi_1(x'_1 + a_1) - \psi_1(x'_1 + b_1)\]
    holds for at least $2^{-8}c^8 |G|^{4}$ of $(x_1, x'_1, a_1, b_1) \in G^4$. Apply Theorem~\ref{invhomm} to find an affine map $\rho_1$ which coincides with $\psi_1$ on a set $Z' \subseteq Z$ of size $|Z'| \geq (c/2)^{O(1)}|G|$. Pass to those 11-tuples whose $x_1 + a_1 \in Z'$, so we still have at least $(c/2)^{O(1)}|G|^{11}$ such 11-tuples.\\

    The same argument applies to any other choice of $\psi_i$ or $\psi'_i$, as the equation is symmetric. Hence, after 16 steps in total, we may assume that there exist affine maps $\rho_i$ and $\rho'_i$ and a collection of 11-tuples $Q' \subseteq Q$ such that $|Q'| \geq (c/2)^{O(1)}|G|^{11}$ and for each $(a_1, \dots, y_4) \in Q'$ we additionally have
    \[\psi_1(x_1 + a_1) = \rho_1(x_1 + a_1), \dots, \psi'_8(y_4) = \rho'_8(y_4).\]

    It remains to relate some pairs of $\rho_i, \rho'_i$. The same first step as above shows
    \[\rho_1(x_1 + a_1) - \rho_2(x_1) = \rho_1(x'_1 + a_1) - \rho_2(x'_1)\]
    for $(c/2)^{O(1)}|G|^3$ choices of $x_1, x'_1, a_1$. Hence $\on{rank} (\rho^{\text{lin}}_1 - \rho^{\text{lin}}_2) \leq O(\log c^{-1})$, where $\rho^{\text{lin}}_i$ is the linear part, i.e., $\rho_i - \rho_i(0)$. Taking $K = \on{ker} (\rho^{\text{lin}}_1 - \rho^{\text{lin}}_2)$ and averaging, we may find a further subset $Q'' \subseteq Q'$ of size $|Q''| \geq (c/2)^{O(1)}|G|^{11}$, where additionally all $x_1$ belong to the same coset $s + K$. Set $u_1 = \rho_2(s) - \rho_1(s)$. Hence, whenever an 11-tuple $(a_1, \dots, y_4)$ belongs to $Q''$, we have
    \[\psi_2(x_1) = \rho_2(x_1) = \rho_2^{\text{lin}}(x_1 - s) + \rho_2(s) = \rho_1^{\text{lin}}(x_1 - s) + \rho_1(s) + u_1 = \rho_1(x_1) + u_1,\]
    using the fact that $x_1 - s \in K$, as claimed.\\
    \indent Finally, apply analogous argument to the remaining 7 pairs of affine maps among $\rho_3, \dots, \rho'_8$.
\end{proof}

Special case for additive quadruples follows by setting all maps to be zero, except $\psi_1, \psi_2, \psi'_1$ and $\psi_2'$.

\begin{proof}[Proof of Proposition~\ref{bilvardomain}] Throughout the proof, we shall iteratively define affine maps $\theta_1, \dots, \theta_m : G \to G$, adding one new map in each step, and, for each $x,y \in G$, we shall maintain a list of indices $I_{x,y} \subseteq [m]$ of size $|I_{x,y}| \leq 2d$ such that 
\[U_x \cap U_y \subseteq \langle \theta_i(x - y) : i \in I_{x,y} \rangle^\perp.\]

The following two claims govern the procedure. Variants of the first claim are present in various forms in previous works involving bilinear Bogolyubov argument~\cite{BienvenuLe, BilinearBog, HosseiniLovett}.

\begin{claim}
    Assume that $|G| \geq 12 p^{4d} \varepsilon^{-1}$. Suppose that currently 
    \begin{equation}
    \langle \theta_i(a) : i \in I_{x + a, x} \cup I_{y + a, y} \rangle^\perp \subseteq (U_{x + a} \cap U_x) + (U_{y + a} \cap U_y)\label{firstbogclaimeqn}\end{equation}
    \textbf{fails} for at least $\varepsilon |G|^3$ of $(x, y, a) \in G^3$. Then there exists an affine map $\rho : G \to G$ such that $\rho(a) \in (U_{x + a} \cap U_x)^\perp \setminus \langle \theta_i(a) : i \in I_{x + a, x}\rangle$ for at least $p^{-O(d)} (\varepsilon/2)^{O(1)}  |G|^2$ pairs $(x,a) \in G^2$.
\end{claim}

\begin{proof} Note that the condition in assumptions, which is
    \[\langle \theta_i(a) : i \in I_{x + a, x} \cup I_{y + a, y} \rangle^\perp \not\subseteq (U_{x + a} \cap U_x) + (U_{y + a} \cap U_y)\]
becomes, after taking orthogonal subspaces,
\begin{equation}(U_{x + a}^\perp + U_x^\perp) \cap (U_{y + a}^\perp + U_y^\perp) \not\subseteq \langle \theta_i(a) : i \in I_{x + a, x} \cup I_{y + a, y} \rangle.\label{dualcond1}\end{equation}

Let us define maps $\psi_1, \psi_2, \psi_3, \psi_4 : G \to G$, by taking a random element $\psi_1(x), \psi_2(x), \psi_3(x), \psi_4(x) \in U_x^\perp$ for all $x \in G$, uniformly and independently. Note that for each $(x, y, a)$ satisfying~\eqref{dualcond1} and such that $x+a, x, y +a, y$ are distinct elements (we call this \textit{non-degenerate case}), the probability that 
\begin{equation}\psi_1(x+a) - \psi_2(x) = \psi_3(y + a) - \psi_4(y) \notin \langle \theta_i(a) : i \in I_{x + a, x} \cup I_{y + a, y} \rangle \label{QdefnEqn1bog}\end{equation}
holds is at least $p^{-4d}$.\\
\indent The number of degenerate triples is at most $6|G|^2 \leq \frac{1}{2}\varepsilon p^{-4d} |G|^3$ by assumptions on parameters. By linearity of expectation, there exists a choice of these 4 maps such that at least $\frac{1}{2}\varepsilon p^{-4d} |G|^3$ triples $(a,x,y)$ satisfy~\eqref{QdefnEqn1bog}. Let $Q$ be such triples $(a,x,y) \in G^3$. By (the special case of) Lemma~\ref{restrictedfreiman}, we have affine maps $\rho_1, \rho_2$ and elements $u_1, u_2$ such that
\[\psi_1(x+a) = \rho_1(x + a), \psi_2(x) = \rho_1(x) + u_1, \psi_3(y+a) = \rho_2(y + a), \psi_4(y) = \rho_2(y) + u_2, \]
holds for a set $Q' \subseteq Q$ of triples  $(a,x,y)$ of size at least $p^{-O(d)} (\varepsilon/2)^{O(1)} |G|^3$. Let $\rho = \rho_1^{\on{lin}} - u_1$, so $\rho(a) = \rho_1(x + a) - \rho_1(x) - u_1$ holds for all $a,x \in G$. We claim that $\rho$ has the desired property.\\ 

Let $(a,x,y) \in Q'$. Then, due to~\eqref{QdefnEqn1bog}, we have
\[\rho(a) = \rho_1(x + a) - \rho_1(x) - u_1 = \psi_1(x + a) - \psi_2(x) \notin \langle \theta_i(a) : i \in I_{x + a, x} \cup I_{y + a, y} \rangle.\]
\indent Observe also that $\rho(a) = \psi_1(x + a) - \psi_2(x) \in U_{x + a}^\perp + U_x^\perp = (U_{x + a} \cap U_x)^\perp$. The claim follows after averaging over $y$. \end{proof}

The second claim is more general, but it is proved similarly.

\begin{claim}
    Assume that $|G| \geq 240 p^{16d} \varepsilon^{-1}$. Suppose that
    \begin{equation}\cap_{j \in [4]} \langle \theta_i(a_j) : i \in I_{x_j + a_j, x_j} \cup I_{y_j + a_j, y_j}  \rangle^\perp \subseteq \Big(\cap_{j \in [4]} (U_{x_j + a_j} \cap U_{x_j})\Big) + \Big(\cap_{j \in [4]} (U_{y_j + a_j} \cap U_{y_j})\Big),\label{secondbogclaimeqn}\end{equation}
    \textbf{fails} for at least $\varepsilon |G|^{11}$ of 12-tuples $(x_{[4]}, y_{[4]}, a_{[4]})$ with $a_1 + a_2 = a_3 + a_4$. Then there exists an affine map $\rho : G \to G$ such that $\rho(a) \in (U_{x + a} \cap U_x)^\perp \setminus \langle \theta_i(a) : i \in I_{x + a, x}\rangle$ for at least $p^{-O(d)} (\varepsilon/2)^{O(1)}|G|^2$ pairs $(x,a) \in G^2$. 
\end{claim}

\begin{proof}  By taking orthogonal subspaces (and using $\sum$ as a shorthand for the sum of several subspaces), the condition in assumptions becomes
\[ \Big(\sum_{j \in [4]} (U_{x_j + a_j}^\perp + U_{x_j}^\perp)\Big) \cap \Big(\sum_{j \in [4]} (U^\perp_{y_j + a_j} + U^\perp_{y_j})\Big) \not\subseteq \sum_{j \in [4]} \langle \theta_i(a_j) : i \in I_{x_j + a_j, x_j} \cup I_{y_j + a_j, y_j}  \rangle.\]

Define 16 maps $\psi_1, \dots, \psi_8, \psi'_1, \dots, \psi'_8 : G \to G$ by taking random elements $\psi_1(x),$ $\psi_2(x), \dots,$ $\psi_8(x),$ $\psi'_1(x), \dots,$ $\psi'_8(x) \in U_x^\perp$ uniformly and independently for each $x \in G$. By linearity of expectation (and slight care regarding degenerate cases when some two points coincide), there exists a choice of these 16 maps such that at least $\frac{1}{2}\varepsilon p^{-16d} |G|^{11}$ choices of 12-tuples $(x_{[4]}, y_{[4]}, a_{[4]})$ with $a_1 + a_2 = a_3 + a_4$ we have

\begin{align*}&\psi_1(x_1 + a_1) - \psi_2(x_1) + \psi_3(x_2 + a_2) - \psi_4(x_2) + \psi_5(x_3 + a_3) - \psi_6(x_3) + \psi_7(x_4 + a_4) - \psi_8(x_4)\\
&\hspace{1cm}=\psi'_1(y_1 + a_1) - \psi'_2(y_1) + \psi'_3(y_2 + a_2) - \psi'_4(y_2) + \psi'_5(y_3 + a_3) - \psi'_6(y_3) + \psi'_7(y_4 + a_4) - \psi'_8(y_4)\\
&\hspace{5cm} \notin \sum_{j \in [4]} \langle \theta_i(a_j) : i \in I_{x_j + a_j, x_j} \cup I_{y_j + a_j, y_j}  \rangle.\end{align*}

Let $Q$ be such 11-tuples (omitting $a_4$, which equals $a_1 + a_2 - a_3$). By Lemma~\ref{restrictedfreiman}, we have affine maps $\rho_1, \rho_2, \dots, \rho_8 : G \to G$ and elements $u_1, u_2, \dots, u_8$ such that
\[\psi_1(x_1+a_1) = \rho_1(x_1 + a_1), \psi_2(x_1) = \rho_1(x_1) + u_1, \psi_3(x_2+a_2) = \rho_2(x_2 + a_2), \psi_4(x_2) = \rho_2(x_2) + u_2, \dots\]
holds for $p^{-O(d)}(\varepsilon/2)^{O(1)} |G|^{11}$ 11-tuples in $Q$.\\

Let $\alpha_i = \rho_i^{\on{lin}} - u_i$, for $i \in [4]$, so $\alpha_i(a) = \rho_i(x + a) - \rho_i(x) - u_i$. Hence,

\begin{align*}&\alpha_1(a_1) + \alpha_2(a_2) + \alpha_3(a_3) + \alpha_4(a_4) = \psi_1(x_1 + a_1) - \psi_2(x_1) + \psi_3(x_2 + a_2) - \psi_4(x_2)\\
&\hspace{2cm} + \psi_5(x_3 + a_3) - \psi_6(x_3) + \psi_7(x_4 + a_4) - \psi_8(x_4) \notin  \sum_{j \in [4]} \langle \theta_i(a_j) : i \in I_{x_j + a_j, x_j} \cup I_{y_j + a_j, y_j} \rangle\end{align*}

holds for $p^{-O(d)}(\varepsilon/2)^{O(1)} |G|^{11}$ of 12-tuples $(x_{[4]}, y_{[4]}, a_{[4]})$ with $a_1 + a_2 = a_3 + a_4$. In particular, for at least one $j$ we have 
\[\alpha_j(a_j) \notin \langle \theta_i(a_j) : i \in I_{x_j + a_j, x_j} \cup I_{y_j + a_j, y_j} \rangle.\]

Without loss of generality, we conclude that this conclusion holds for $j = 1$.\\

\indent Observe also that $\alpha_1(a_1) = \psi_1(x_1 + a_1) - \psi_2(x_1) \in U_{x_1 + a_1}^\perp + U_{x_1}^\perp = (U_{x_1 + a_1} \cap U_{x_1})^\perp$. The claim follows after averaging over remaining 9 parameters. \end{proof}

Note that we may assume the technical condition $|G| \geq 240 p^{16d} \varepsilon^{-1}$ appearing in the claims above, as otherwise the proposition is trivial.\\

Applying claims above until both conditions~\eqref{firstbogclaimeqn} and~\eqref{secondbogclaimeqn} hold for all but at most $\varepsilon$ proportion of relevant tuples, and adding new $\theta_{m+1} = \rho$ to the list of maps and index $m + 1$ to the index set $I_{x + a, x}$ for $x, a$ such that $\rho(a) \in (U_{x + a} \cap U_x)^\perp \setminus \langle \theta_i(a) : i \in I_{x + a, x}\rangle$ at each step, we get the desired list of maps and sets after $m \leq p^{O(d)}(\varepsilon/2)^{O(1)}$ steps. This bound holds due to the fact that, in each step, we increase the size of $I_{x,y}$ by 1, for at least $p^{-O(d)}(\varepsilon/2)^{-O(1)} |G|^2$ choices of $(x,y) \in G^2$, and $|I_{x,y}| \leq 2d$ holds at all times.\\ 

Recall that $1-\varepsilon$ proportion of additive quadruples in $G$ are subspace-respected. Hence, we can find a set $A$ of size at least $(1-\sqrt{\varepsilon})|G|$ such that for each $a \in A$, the number of additive quadruples $(x + a, x, y + a, y)$ which are subspace-respected is at least $(1-\sqrt{\varepsilon})|G|^2$. Additionally, let $X_a$ be the set of all $x \in G$ such that $(x + a, x, y + a, y)$ is subspace-respected for at least $(1-\sqrt[4]{\varepsilon})|G|$ of $y \in G$. Then $|X_a| \geq (1-\sqrt[4]{\varepsilon})|G|$.\\ 
\indent Furthermore, let $A' \subseteq A$ be the set of $a \in A$ such that~\eqref{firstbogclaimeqn} holds for at least $(1 - \sqrt{\varepsilon})|G|^2$ of pairs $(x,y) \in G^2$. Hence, $|A'| \geq (1 - 2\sqrt{\varepsilon})|G|$.\\

Take random pair $(x_a, y_a)$ uniformly from $G^2$, independently for each $a \in A'$. We say that $(x_a, y_a)$ is \textit{good} if 
\begin{itemize}
    \item $(x_a + a, x_a, y_a + a, y_a)$ is subspace-respected,
    \item $x_a, y_a \in X_a$, and
    \item $\langle \theta_i(a) : i \in I_{x_a + a, x_a} \cup I_{y_a + a, y_a} \rangle^\perp \subseteq (U_{x_a + a} \cap U_{x_a}) + (U_{y_a + a} \cap U_{y_a})$.
\end{itemize} 
The definition of $A'$ implies that the probability $(x_a, y_a)$ being good is at least $1 - 4\sqrt[4]{\varepsilon}$ for any given $a \in A'$.\\
\indent If these conditions hold, we define a linear map $\psi_a : G \to H$ as a simultaneous extension of maps $\phi_{x_a + a} - \phi_{x_a}|_{U_{x_a + a} \cap U_{x_a}}$ and $\phi_{y_a + a} - \phi_{y_a}|_{U_{y_a + a} \cap U_{y_a}}$. There exists such a map as the latter two maps coincide on the intersection of their domains, since $(x_a + a, x_a, y_a + a, y_a)$ is subspace-respected.\\
\indent We associate subspace 
\[S_a = \langle \theta_i(a) : i \in I_{x_a + a, x_a} \cup I_{y_a + a, y_a} \rangle^\perp\]
with $\psi_a$. Note that the conditions above imply that $\psi_a$ is uniquely determined on $S_a$.\\

Firstly, observe that if it happens that $(x_a, y_a)$ is good, and thus $\psi_a$ is defined, then we have 
\[\on{rank}\Big(\psi_a + \phi_z - \phi_{z + a}\Big)\leq 6d\]
for $(1-\sqrt[4]{\varepsilon})|G|$ of $z \in G$. Namely, first of all, $\psi_a$ equals $\phi_{x_a + a} - \phi_{x_a}$ on ${U_{x_a + a} \cap U_{x_a}}$, so we have 
\[\on{rank}\Big(\psi_a + \phi_{x_a} - \phi_{x_a + a}\Big)\leq 2d.\]

On the other hand, $x_a \in X_a$, so for $(1-\sqrt[4]{\varepsilon})|G|$ of $z \in G$, we have
\[\on{rank}\Big(\phi_{z + a} - \phi_z  + \phi_{x_a} - \phi_{x_a + a}\Big)\leq 4d.\]

Let us now show that the expected number of subspace-respected additive quadruples in $(\psi_a, S_a)$ is large. Let $Q$ be the set of all quadruples $a_{[4]}$ in $A'$ such that $a_1 + a_2 = a_3 + a_4$, such that additive 8-tuples $(x_1 + a_1, x_1, \dots, x_4 + a_4, x_4)$ are subspace-respected (with respect to system $\phi_\bcdot$) for at least $(1-\sqrt{\varepsilon})|G|^4$ quadruples $x_{[4]}$ in $G^4$, and such that~\eqref{secondbogclaimeqn} holds for at least $(1-\sqrt{\varepsilon})|G|^8$ of $x_{[4]}, y_{[4]}$ in $G^8$. Thus $|Q| \geq (1-20\sqrt[4]{\varepsilon})|G|^3$.\\

Fix any $a_{[4]}$ in $Q$. Suppose that $(x_{a_i}, y_{a_i})$ is good for all $i \in [4]$. Writing $x_i = x_{a_i}$ and $y_i = y_{a_i}$, if both additive 8-tuples $(x_1 + a_1, x_1, \dots, x_4 + a_4, x_4)$ and $(y_1 + a_1, y_1, \dots, y_4 + a_4, y_4)$ are subspace-respected, we see that 
\[\psi_{a_1} + \psi_{a_2} - \psi_{a_3} - \psi_{a_4} = 0\]
holds both on 
\[\cap_{i \in [4]}(U_{x_i + a_i} \cap U_{x_i})\]
and 
\[\cap_{i \in [4]}(U_{y_i + a_i} \cap U_{y_i}),\]
and hence on their sum. If it also happens that 
\[\cap_{j \in [4]} \langle \theta_i(a_j) : i \in I_{x_j + a_j, x_j} \cup I_{y_j + a_j, y_j}  \rangle^\perp \subseteq \Big(\cap_{j \in [4]} (U_{x_j + a_j} \cap U_{x_j})\Big) + \Big(\cap_{j \in [4]} (U_{y_j + a_j} \cap U_{y_j})\Big),\]
then 
\[\Big(\psi_{a_1} + \psi_{a_2} - \psi_{a_3} - \psi_{a_4}\Big)_{\cap_{i \in [4]} S_{a_i}} = 0,\]
i.e., we get a subspace-respected additive quadruple in the new system. Summing up, for fixed $a_{[4]}$ in $Q$, we have 
\[\mathbb{P}(a_{[4]}\,\,\psi_{\bcdot}\text{-subspace-respected}) \geq 1 - 20\sqrt[4]{\varepsilon}.\]

Hence, as $\varepsilon \leq 2^{-24}$ from assumptions, the expected number of subspace-respected additive quadruples in $(\psi_a, S_a)$ is at least $\frac{1}{2}|G|^3$. Thus, there is such a choice of the new system.\\

Finally, take a random set $J \subseteq [m]$ of size $16d$, uniformly among such sets and set $\tilde{A} = \{a \in A' : I_{x_a + a, x_a} \cup I_{y_a + a, y_a} \subseteq J\}$. The expected number of additive quadruples in $\tilde{A}$ that are subspace-respected with respect to the system $(\psi_a, S_a)$ is at least $\binom{m}{16d}^{-1}\frac{1}{2}|G|^3$. Hence, there is a choice of $J$ for which there are that many subspace-respected additive quadruples. Take the required affine maps to be $\theta_j$ for $j \in J$ and the final subspaces $B_a = \langle \theta_i(a) : i \in J\rangle^\perp$ (which are possibly proper subspaces of $S_a$) to finish the proof.
\end{proof}

\section{Bilinear variety with many respected additive quadruples}
\label{secbsg2}
\hspace{17pt}In this section, we assume that we are given a bilinear variety along with linear maps defined on a dense collection of columns, that induce many subspace-respected additive quadruples. In the rest of the argument, we simplify the terminology and write respected instead of subspace-respected. Our goal is to find a subset of columns on which \textbf{all} additive 12-tuples are respected. We need longer respected tuples for the final extension step.\\
\indent As explained in introduction and Section~\ref{detailsoversec}, we rely on the abstract Balog-Szemer\'edi-Gowers argument in this step. However, compared to the second step of the proof, which was similar in spirit, this step is more involved as subspace-respectedness is not a `transitive' property of additive quadruples on its own, requiring algebraic regularity method, and also because the algebraic dependent random choice argument carried out near the end of the proof would introduce exponential loss if applied naively to partially-defined linear maps, leading to additional subtleties present in Claim~\ref{commonglobalextension}. We comment on the necessity of global maps near inequality~\eqref{nexpquarteqn}.\\
\indent To reduce the burden of chasing the various constant parameters, in this section we use the large constant notation $\blc$, introduced near the beginning of Section~\ref{prelimsec}.

\begin{proposition}\label{Fbihomstep}
    Let $\beta \colon G \times G \to \mathbb{F}_p^r$ be a bilinear map and write $B = \{(x,y) \in G \times G\colon \beta(x,y) = 0\}$. Write $U_x = B_{x \bcdot}$. Let $X \subseteq G$ be a set and suppose that we are also given a linear map $\phi_x \colon U_x \to H$ for each $x \in X$. Assume that at least $c|G|^3$ additive quadruples $x_1 + x_2 = x_3 + x_4$ in $X$ are \textit{respected} in the sense that
     \[\Big(\phi_{x_1} + \phi_{x_2} - \phi_{x_3} - \phi_{x_4}\Big)\Big|_{U_{x_1} \cap U_{x_2} \cap U_{x_3} \cap U_{x_4}} = 0.\]
     Then there exists a subset $X' \subseteq X$ of size $|X'| \geq \exp\Big( - O\Big((r + \log c^{-1})^{O(1)}\Big)\Big)|G|$ and a subspace $V$ of codimension $O\Big((r + \log c^{-1})^{O(1)}\Big)$ such that all additive 12-tuples in $X'$ are respected by $(\phi_x, U_x \cap V)$ system.
\end{proposition}

\begin{proof}
    Let $R > 0$ be a parameter to be chosen later (we will end up choosing $R$ larger than $K (r + 1)(\log c^{-1} + 1)$ for some sufficiently large absolute constant $K$). We begin the proof by applying the algebraic regularity lemma (Theorem~\ref{arl}), which gives us a bilinear map $\gamma : G \times G \to \mathbb{F}_p^{r'}$, where $r' \leq r$, and a subspace $V$ of codimension at most $2rR$ making $\gamma$ have rank at least $R$ on $V \times V$ and property~\eqref{allcosetseq} for the bilinear variety $B = \{\beta = 0\}$. Note also that $B \cap (a + V) \times V \not= \emptyset$ for all $a$, as it contains $(a,0)$.\\
    \indent Let $G = V \oplus T$. By averaging, we may find $t_1, t_2, t_3 \in T$ such that the number of respected additive quadruples in $(t_1 + V) \times (t_2 + V) \times (t_3 + V) \times (t_1 + t_2 - t_3 + V)$ is at least $c|V|^3$. Write $X_i = X \cap t_i + V$, with $t_4 = t_1 +t_2 -t_3$. From now on, we restrict our attention to $V$ as the set of rows. Hence, we still have at least $c|V|^3$ respected additive quadruples in $\prod_{i \in [4]} (t_i + V)$ with respect to the system $(\phi_x, U_x \cap V)$. In particular, $|X_i| \geq c|V|$ for each $i \in [4]$.\\

    We misuse the notation and keep writing $\beta$ instead of $\gamma$. We also restrict $\beta$ to $G \times V$ (as only rows indexed by $V$ play a role). Write $r$ instead of $r'$ and $U_x$ instead of $U_x \cap V$. Hence, we assume that $\beta|_{V \times V}$ has rank at least $R$, and the cost is that we passed to subspace $V$ of codimension at most $2rR$.\\

    The following claim describes the fundamental property of respected additive quadruples. In contrast to the previous notion of being respected, which was about small ranks, this time the fact that $(x+a, x, z+a, z)$ and $(y+a, y, z+ a, z)$ are two respected additive quadruples need not imply that $(x +a, x, y + a, y)$ is also respected. However, if there are many such $z$, then the latter quadruple is respected. The claim also applies to longer `chains' of respected additive quadruple.\\

    Before proceeding with the proof, observe the following important property of the system of subspaces arising from bilinear varieties. Namely, if $x_1 + x_2 = x_3 + x_4$, then
    \[U_{x_1} \cap U_{x_2} \cap U_{x_3} \cap U_{x_4} = U_{x_1} \cap U_{x_2} \cap U_{x_3}.\]
    To see this, note that if $y \in U_{x_1} \cap U_{x_2} \cap U_{x_3}$, then $\beta(x_i, y) = 0$ for all $i \in [3]$ and hence $\beta(x_4, y) = \beta(x_1, y) + \beta(x_2, y) - \beta(x_3, y) = 0$, so $y \in U_{x_4}$ as well.\\    
    \indent Similar equalities hold for other intersections of subspaces whose indices satisfy linear relations.
    
\begin{claim}
    \label{extendquadsclaim}
    Suppose that $x, x+ a, y, y + a$ are points for which there are at least $c' |V|^k$ of $(z_1, \dots, z_k) \in (X \cap t + V)^k$ such that additive quadruples $(x+a, x, z_1+a, z_1)$, $(z_i + a, z_i, z_{i+1} + a, z_{i+1})$ for $i \in [k-1]$, and $(y+a, y, z_k+ a, z_k)$ are respected. If $R \geq \blc(kr + \log {c'}^{-1} + 1)$, then the quadruple $(x +a, x, y + a, y)$ is also respected.
\end{claim}

\begin{proof}
    Let $Z$ be the set of such $k$-tuples $z_{[k]}$ in $X \cap t + V$. For each $z_{[k]} \in Z$, we have
    \begin{align*}&\phi_{x + a} - \phi_x  - \phi_{y+a} + \phi_y = \Big(\phi_{x + a} - \phi_x - \phi_{z_1+a} + \phi_{z_1}\Big) + \Big(\phi_{z_1 + a} - \phi_{z_1}  - \phi_{z_2+a} + \phi_{z_2}\Big) + \dots\\
    &\hspace{2cm}\Big(\phi_{z_{k-1} + a} - \phi_{z_{k-1}}  - \phi_{z_{k}+a} + \phi_{z_k}\Big) +\Big(\phi_{z_k + a} - \phi_{z_k}  - \phi_{y+a} + \phi_y\Big) = 0\end{align*}
    on $U_{x+a} \cap U_x \cap U_{y+a} \cap U_{y} \cap \Big(\bigcap_{i \in [k]} (U_{z_i + a} \cap U_{z_i})\Big) = S \cap \Big(\bigcap_{i \in [k]} U_{z_i}\Big)$, where $S = U_{x+a} \cap U_x \cap U_{y+a} \cap U_{y}$.\\
    \indent Hence, for at least ${c'}^2|V|^{2k}$ of choices $z_{[k]}, w_{[k]}$ in $t + V$, we have 
    \[\phi_{x + a} - \phi_x - \phi_{y+a} + \phi_y = 0\]
    on $(S \cap (\cap_{i \in [k]} U_{z_i}))) + (S \cap (\cap_{i \in [k]} U_{w_i}))$. From Lemma~\ref{subspsumsreg}, since $S$ has codimension at most $4r$, this is $S$ for some choice of $z$ and $w$.
\end{proof}

Using the lemma, we may pass to a single coset of $V$. Namely, we know that there are at least $c|V|^3$ respected additive quadruples in $X_1 \times X_2 \times X_3 \times X_4$. By Cauchy-Schwarz inequality, there are at least $c^2|V|^4$ choices of $(x_1, x'_1, x_2, x_3, x'_3, x_4) \in X_1 \times X_1 \times X_2 \times X_3 \times X_3 \times X_4$ such that $x_1 - x_3 = x'_1 - x'_3 = x_4 - x_2$ and both $(x_1, x_2, x_3, x_4)$ and $(x'_1, x_2, x'_3, x_4)$ are respected. In particular, for at least $\frac{c^2}{2}|V|^3$ of additive quadruples $(x_1, x'_1, x_3, x'_3) \in X_1 \times X_1 \times X_3 \times X_3$, there are at least $\frac{c^2}{2}|V|$ choices of $x_2, x_4$ making both $(x_1, x_3, x_4, x_2)$ and $(x'_1, x'_3, x_4, x_2)$ respected. By Claim~\ref{extendquadsclaim}, $(x_1, x_3, x'_1, x'_3)$ is respected.\\
\indent Apply the same argument to $X_1$. We end up with at least $\frac{c^4}{8}|V|^3$ respected additive quadruples in $X_1 \subset t + V$, where $t = t_1$, as long as $R \geq \blc(r + \log {c'}^{-1} + 1)$. We misuse the notation and write $X = X_1$.\\

Consider a set $A$ of elements $a \in V$ such that there are at least $\frac{c^4}{16}|V|^2$ respected additive quadruples of difference $a$ in $X$ and keep at most one element in every pair $\{a, -a\}$. We may assume that $|A| \geq \frac{c^4}{32}|V|$. For each $a \in A$, consider the graph whose vertices are pairs $(x + a, x) \in X^2$ and whose edges $(x+a,x) (y+a, y)$ are pairs of pairs inducing respected quadruples. Lemma~\ref{gowerspathssingle} allows us to find a set $P_a$ of size $2^{-9}c^4|V|$ of pairs $(x + a, x)$ such that between any two pairs we have at least $2^{-71}c^{36}|V|^5$ 6-paths. In other words, we get that many 5-tuples $(z_1, \dots, z_5) \in X^5$ such that $z_i + a \in X$ and $(x+a, x, z_1 + a, z_1)$, $(z_i  + a, z_i, z_{i+1} + a, z_{i+1})$ and $(z_5 + a, z_5, y + a, y)$ are all respected. Applying Claim~\ref{extendquadsclaim}, we conclude that $(x + a, x, y+a, y)$ is respected whenever $(x+a,x), (y + a, y) \in P_a$, as long as $R \geq \blc(r + \log c^{-1} + 1)$.\\

Define $P$ to be the union of $P_a$ for $a \in A$, and let it as well contain pairs $(x, x+a)$ for $(x + a,x) \in P_a$ and $a \in A$, making it symmetric. Hence, whenever two pairs in $P$ have same difference they induce a respected additive quadruple and $|P|\geq 2^{-14}c^8 |V|^2$.\\

Now consider graph on $X$ whose edges are pairs in $P$. Apply Lemma~\ref{gowerspathssingle} to find a $2^{-19}c^8$-dense subset $Y$ such that between any two $y, y' \in Y$ we have $2^{-200}c^{80}|V|^5$ paths of length 6 with edges in $P$. The key property of the set $Y$ is that it is everywhere rich in structure, which is articulated in the following claim.

\begin{claim}\label{manyrespclaim2}
    Suppose that $Y_1, Y_2 \subseteq Y$ are sets of sizes $c' |V|$ and $c''|V|$. Then there are at least ${c'}^2{c''}^2 2^{-500}c^{160}|V|^3$ respected additive quadruples $(y_1, y_2, y_1 + a, y_2 + a) \in Y_1 \times Y_2 \times Y_1 \times Y_2$, as long as $R \geq \blc\Big(r  + \log c^{-1}+ \log {c'}^{-1} +\log {c''}^{-1} + 1\Big)$.
\end{claim}

\begin{proof}
    By assumption and property of $Y$ above, we have at least ${c'c''} 2^{-200}c^{80}|V|^7$ of $7$-tuples $(y_1, z_1, \dots, z_5, y_2) \in Y_1 \times (t + V)^5 \times Y_2$ such that $(y_1, z_1), (z_1, z_2), \dots, (z_5, y_2) \in P$.\\
    \indent By a change of variables and Cauchy-Schwarz inequality, 
    \begin{align*}&\Big(\sum_{y_1, y_2, z_1, \dots, z_5 \in t + V} \id_{Y_1}(y_1)\id_{Y_2}(y_2) \id_P(y_1, z_1) \id_P(z_1, z_2) \dots \id(z_5, y_2)\Big)^2\\
    &\hspace{1cm}=\Big(\sum_{y_1 \in t + V, a_1, \dots, a_6 \in V} \id_{Y_1}(y_1)\id_{Y_2}(y_1 + a_1 + \dots + a_6) \id_P(y_1, y_1 + a_1) \id_P(y_1 + a_1, y_1 + a_1 + a_2) \dots\\
    &\hspace{5cm}\id_P(y_1 + a_1 + \dots + a_5, y_1 + a_1 + \dots + a_6)\Big)^2\\
    &\hspace{1cm}\leq |V|^6 \sum_{a_1, \dots, a_6 \in t + V}  \Big(\sum_{y_1 \in t + V} \id_{Y_1}(y_1)\id_{Y_2}(y_1 + a_1 + \dots + a_6) \id_P(y_1, y_1 + a_1) \id_P(y_1 + a_1, y_1 + a_1 + a_2) \dots \\
    &\hspace{5cm}\id_P(y_1 + a_1 + \dots + a_5, y_1 + a_1 + \dots + a_6)\Big)^2.
    \end{align*}
    Hence, we have at least ${c'}^2{c''}^2 2^{-400}c^{160}|V|^8$ of pairs of such $7$-tuples $(y_1, z_1, \dots, z_5, y_2)$ and $(y'_1, z'_1, \dots, z'_5, y'_2)$ such that additionally $y_1 - z_1 = y'_1 - z'_1, \dots, z_5 - y_2 = z'_5 - y'_2$. Due to the property that any two pairs in $P$ of the same difference induce a respected additive quadruple, we get respected 
    \[(y_1, y'_1, z_1, z'_1), (z_1, z'_1, z_2, z'_2), \dots, (z_5, z'_5, y_2, y'_2).\]

    Applying Claim~\ref{extendquadsclaim} and averaging, we conclude that there are at least ${c'}^2{c''}^2 2^{-401}c^{160}|V|^3$ additive quadruples $(y_1, y'_1, y_2, y'_2)$ that are respected, as long as $R \geq \blc\Big(r + \log {c}^{-1} + \log {c'}^{-1} + \log {c''}^{-1} + 1\Big)$.
\end{proof}

Let $c_1$ be the density of $Y$ in $t + V$, which we know to be at least $2^{-19}c^{8}$. The crucial property for completing the proof is the following.

\begin{claim} \label{singletomanyaddtivetuples}
    Let $Y' \subseteq Y$ be a subset with the property that every element in $Y'$ belongs to $c'|V|^2$ respected additive quadruples in $Y$. Then, provided $R \geq \blc_k\Big(r + \log c^{-1} + \log {c'}^{-1} + 1\Big)$, for each additive $2k$-tuple $x_{[2k]}$ in $Y'$ we have at least  $(c c'/2)^{O_k(1)}|V|^{6k-1}$ of additive $6k$-tuples $y_{[6k]}$ in $Y$ such that 
    \[\sum_{i \in [2k]} (-1)^i \phi_{x_i} = \sum_{i \in [6k]} (-1)^i \phi_{y_i}\]
    on $(\cap_{i \in [2k]} U_{x_i}) \cap (\cap_{i \in [6k]} U_{y_i})$.
\end{claim}

\begin{proof}
    Take arbitrary additive $2k$-tuple $x_{[2k]}$ in $Y'$. Let $\delta = 2^{-500}c^{160}$. By induction on $\ell \in [2k]$, we show that there exist at least $(c c'/2)^{O_\ell(1)}|V|^{3\ell-1}$ choices of $3\ell$-tuples $y_1, \dots, y_{3\ell} \in Y$ such that
    \[\sum_{i \in [\ell]} (-1)^i x_i = \sum_{i \in [3\ell]} (-1)^i y_i\]
    and
    \[\sum_{i \in [\ell]} (-1)^i \phi_{x_i} = \sum_{i \in [3\ell]} (-1)^i \phi_{y_i}\]
    on the subspace $(\cap_{i \in [\ell]} U_{x_i}) \cap (\cap_{i \in [3\ell]} U_{y_i})$. The claim will follow from the case $\ell = 2k$.\\

    \indent For the inductive base, the claim holds trivially, as every element in $Y'$ belongs to $c'|V|^2$ respected additive quadruples in $Y$.\\
    \indent Suppose that the claim holds for some $\ell < 2k$. Hence, we have a set $Z$ of size $\alpha |V|^{3\ell - 1}$, where $\alpha \geq (c c'/2)^{O_\ell(1)}$, consisting of the relevant $3\ell$-tuples for $x_1, \dots ,x_\ell$. Let $Z'$ be the set of all triples $(z_1, z_2, z_3) \in Y^3$ such that $z_1 - z_2 + z_3 = x_{\ell + 1}$ and $(z_1, z_2, z_3, x)$ is respected. Thus $|Z'| \geq c'|V|^2$.\\
    \indent Let $W$ be the set of those $y_{3\ell} \in Y$ such that at least $\frac{\alpha}{2}|V|^{3\ell - 2}$ $3\ell$-tuples in $Z$ have $y_{3\ell}$ as their last element. Hence $|W| \geq \frac{\alpha}{2}|V|$. Similarly, let $W'$ be the set of those $z_3 \in Y$ such that at least $\frac{c'}{2}|V|$ triples in $Z'$ have $z_{3}$ as their last element. Thus $|W'| \geq \frac{c'}{2}|V|$.\\
    \indent Apply Claim~\ref{manyrespclaim2} to sets $W$ and $W'$, to find at least $2^{-4} \delta {c'}^2 \alpha^2 |V|^3$ triples $(y', z', t) \in Y \times Y \times V$ such that $(y', z', y' - t, z' - t)$ is respected and $y' - t \in W$, $z' - t \in W'$ (recall that $\delta$ was defined at the beginning of the proof as $2^{-500}c^{160}$). By the definition of $W$ and $W'$ we conclude that there is a set $\tilde{Z}$ of size at least $2^{-6} \delta {c'}^3 \alpha^3 |V|^{3\ell + 2}$, consisting of $(3\ell + 4)$-tuples $(y', z', t, y_1, \dots, y_{3\ell - 1}, z_1, z_2)$ such that $t \in V$, quadruple $(y', z', y' - t, z' -t)$ is respected, $(y_1, \dots, y_{3\ell - 1}, y' - t) \in Z$ and $(z_1, z_2, z' - t) \in Z'$. Thus, for each such $(3\ell + 4)$-tuple in $\tilde{Z}$, we have
    
    \begin{align*}&\sum_{i \in [3\ell - 1]} (-1)^i y_i + (-1)^{3\ell} y' + (-1)^{3\ell + 1} z_1 + (-1)^{3\ell + 2} z_2 + (-1)^{3\ell + 3} z'\\
    &\hspace{1cm} = \sum_{i \in [3\ell - 1]} (-1)^i y_i + (-1)^{3\ell} (y' - t) + (-1)^{\ell + 1} z_1 + (-1)^{\ell + 2} z_2 + (-1)^{\ell + 3} (z'-t)\\
    &\hspace{1cm} = \sum_{i \in [\ell]} (-1)^i x_i + (-1)^{\ell + 1} x_{\ell + 1}\end{align*}
    
    and

    \begin{align}&\sum_{i \in [3\ell - 1]} (-1)^i \phi_{y_i} + (-1)^{3\ell} \phi_{y'} + (-1)^{3\ell + 1} \phi_{z_1} + (-1)^{3\ell + 2} \phi_{z_2} + (-1)^{3\ell + 3} \phi_{z'}\label{twotosixkmapseq}\\ 
    &\hspace{2cm}=\sum_{i \in [3\ell - 1]} (-1)^i \phi_{y_i} + (-1)^{3\ell} \phi_{y'-t} + (-1)^{\ell + 1} \phi_{z_1} + (-1)^{\ell + 2} \phi_{z_2} + (-1)^{\ell + 3} \phi_{z'-t}\nonumber\\
    &\hspace{2cm} = \sum_{i \in [\ell]} (-1)^i \phi_{x_i} + (-1)^{\ell + 1} \phi_{x_{\ell + 1}}\nonumber\end{align}

    holds on the intersection
    
    \[U_{y'} \cap U_{z'} \cap U_{y' - t} \cap U_{z' - t} \cap \Big(\bigcap_{i \in [3\ell - 1]} U_{y_i} \Big) \cap U_{z_1} \cap U_{z_2} \cap \Big(\bigcap_{i \in [\ell + 1]} U_{x_i}\Big).\]
    
    But, since $y' - t = (-1)^{3\ell} \Big(\sum_{i \in [\ell]} (-1)^i x_i - \sum_{i \in [3\ell - 1]} (-1)^i y_i\Big)$ and $z' - t = (x_{\ell + 1} - z_1 + z_2)$, this subspace is in fact equal to 
    \[U_{y'} \cap U_{z'} \cap \Big(\bigcap_{i \in [3\ell - 1]} U_{y_i} \Big) \cap U_{z_1} \cap U_{z_2} \cap \Big(\bigcap_{i \in [\ell + 1]} U_{x_i}\Big).\]

    The map given by omission of $t$ in $(3\ell + 4)$-tuples in $\tilde{Z}$ is injective, as $t = z_1 - z_2 + z' - x_{\ell + 1}$, so the claim follows.
\end{proof}

Furthermore, the claim above allows us to deduce that linear combinations of linear maps assigned to additive tuples, while not necessarily zero, frequently come from the restriction of the same global linear map.

\begin{claim} \label{commonglobalextension}
    Let $Y' \subseteq Y$ be a subset with the property that every element in $Y'$ belongs to $c'|V|^2$ respected additive quadruples in $Y$ and let $Q$ be a collection of additive 36-tuples in $Y'$ of size $|Q| \geq c'' |V|^{35}$. Then, provided $R \geq \blc\Big(r  + \log c^{-1} + \log {c'}^{-1} + \log {c''}^{-1} + 1\Big)$, there exists a linear map $\psi : V \to H$ such that 
    \[(\sum_{i \in [36]} (-1)^i \phi_{x_i})|_{\cap_{i \in [36]} U_{x_i}} = \psi|_{\cap_{i \in [36]} U_{x_i}}\]
    holds for at least $(cc'c'')^{O(1)} |V|^{35}$ of $x_{[36]} \in Q$.
\end{claim}

\begin{proof}
    Consider the bipartite graph whose vertex classes are $Q$ and the set of all additive 108-tuples in $Y$. We put an edge between $x_{[36]}$ and $y_{[108]}$ if
    \[\sum_{i \in [36]} (-1)^i \phi_{x_i} = \sum_{i \in [108]} (-1)^i \phi_{y_i}\]
    on $(\cap_{i \in [35]} U_{x_i}) \cap (\cap_{i \in [107]} U_{y_i}) = (\cap_{i \in [36]} U_{x_i}) \cap (\cap_{i \in [108]} U_{y_i})$.\\
    \indent By the previous claim, this graph has density at least $(cc'/2)^{O(1)}$. Due to Lemma~\ref{subspsumsreg}, all but at most $O(p^{O(r) - R/4}|V|^{284})$ choices of $(x_{[36]}, x'_{[36]}, y_{[108]}, y'_{[108]})$, where $x_{[36]}$ and $x'_{[36]}$ are additive 36-tuples and $y_{[108]}$ and $y'_{[108]}$ are additive 108-tuples, all in $t+V$, satisfy equalities
    \begin{align}
    V = & (\cap_{i \in [107]} U_{y_i}) + (\cap_{i \in [107]} U_{y'_i}),\label{manysubspadd1}\\
    (\cap_{i \in [107]} U_{y_i}) \cap (\cap_{i \in [107]} U_{y'_i}) = &\Big((\cap_{i \in [107]} U_{y_i}) \cap (\cap_{i \in [107]} U_{y'_i}) \cap (\cap_{i \in [35]} U_{x_i})\Big) \nonumber\\
    &\hspace{4cm}+ \Big((\cap_{i \in [107]} U_{y_i}) \cap (\cap_{i \in [107]} U_{y'_i}) \cap (\cap_{i \in [35]} U_{x'_i})\Big),\label{manysubspadd2}\\
   (\cap_{i \in [35]} U_{x_i}) = &\Big((\cap_{i \in [35]} U_{x_i}) \cap (\cap_{i \in [107]} U_{y_i})\Big) + \Big((\cap_{i \in [35]} U_{x_i}) \cap (\cap_{i \in [107]} U_{y'_i})\Big),\label{manysubspadd3}\\
    (\cap_{i \in [35]} U_{x'_i}) = &\Big((\cap_{i \in [35]} U_{x'_i}) \cap (\cap_{i \in [107]} U_{y_i})\Big) + \Big((\cap_{i \in [35]} U_{x'_i}) \cap (\cap_{i \in [107]} U_{y'_i})\Big).\label{manysubspadd4}
    \end{align}
    
    Hence, provided $R \geq \blc\Big(r  + \log c^{-1} + \log {c'}^{-1} + \log {c''}^{-1} + 1\Big)$, we may find two additive 108-tuples $y_{[108]}$ and $y'_{[108]}$ such that~\eqref{manysubspadd1} holds, and there are at least $(cc' c''/2)^{O(1)} |V|^{70}$ pairs of 36-tuples $(x_{[36]}, x'_{[36]}) \in Q^2$ such that $x_{[36]}y_{[108]}, x'_{[36]}y_{[108]}, x_{[36]}y'_{[108]}, x'_{[36]}y'_{[108]}$ are edges in the bipartite graph above and~\eqref{manysubspadd2},~\eqref{manysubspadd3} and~\eqref{manysubspadd4} hold. Let $Z \subseteq Q^2$ be the set of such $(x_{[36]}, x'_{[36]})$.\\

    \indent Consider linear maps $\rho = \sum_{i \in [108]} (-1)^i \phi_{y_i}$ and $\rho' = \sum_{i \in [108]} (-1)^i \phi_{y'_i}$ and subspaces $K = \cap_{i \in [107]} U_{y_i}$ and $K' = \cap_{i \in [107]} U_{y'_i}$. We first show that $\rho|_{K \cap K'} = \rho'|_{K \cap K'}$. From this and~\eqref{manysubspadd1}, it follows that there exists a unique linear map $\psi : V \to H$ such that $\psi|_K = \rho|_K$ and $\psi|_{K'} = \rho'|_{K'}$.\\
    \indent To see that $\rho|_{K \cap K'} = \rho'|_{K \cap K'}$, take arbitrary $(x_{[36]}, x'_{[36]}) \in Z$ and define linear maps $\theta = \sum_{i \in [36]} (-1)^i \phi_{x_i}$ and $\theta' = \sum_{i \in [36]} (-1)^i \phi_{x'_i}$, and take subspaces $L = \cap_{i \in [35]} U_{x_i}$ and $L' = \cap_{i \in [35]} U_{x'_i}$. Hence, we have 
    \[\rho|_{K \cap K' \cap L} = \theta|_{K \cap K' \cap L} = \rho'|_{K \cap K' \cap L}\]
    and
    \[\rho|_{K \cap K' \cap L'} = \theta'|_{K \cap K' \cap L'} = \rho'|_{K \cap K' \cap L'}.\]

    Due to~\eqref{manysubspadd2}, we have $\rho|_{K \cap K'} = \rho'|_{K \cap K'}$, as desired. Note also that, even though we used $x_{[36]}$ and $x'_{[36]}$ in the argument, the map $\psi : V \to H$ is independent of this choice.\\
    
    Finally, we show that if $(x_{[36]}, x'_{[36]}) \in Z$, then 
    \[(\sum_{i \in [36]} (-1)^i \phi_{x_i})|_{\cap_{i \in [36]} U_{x_i}} = \psi|_{\cap_{i \in [36]} U_{x_i}}.\]
    Using the same notation as above, we have
    \[\psi|_{K \cap L} = \rho|_{K \cap L} = \theta|_{K \cap L}\]
    and
    \[\psi|_{K' \cap L} = \rho'|_{K' \cap L} = \theta|_{K' \cap L}\]
    so $\psi = \theta$ on $(K \cap L) + (K' \cap L) = L$, owing to~\eqref{manysubspadd3}.\end{proof}

Recall that $|Y| = \alpha_1|V|$ for some $\alpha_1 \geq (c/2)^{O(1)}$. Let $\alpha_2 = \alpha_1^42^{-4000} c^{1000}$ and let $Y'$ be the collection of all elements of $Y$ which belong to at least $\alpha_2 |V|^2$ respected additive quadruples in $Y$. By Claim~\ref{manyrespclaim2}, $|Y'| \geq \frac{1}{2}|Y|$.\\

By Claim~\ref{singletomanyaddtivetuples}, for each additive 12-tuple $q$ in $Y'$, we have $\alpha_3|V|^{35}$ additive 36-tuples in $Y$, for $\alpha_3 \geq (c/2)^{O(1)}$, whose linear combination of linear maps coincides with that of the 12-tuple in the sense of that claim.\\
\indent Let $Y''$ be the set of all elements in $Y$ that belong to at least $(c/2)^{10000}\alpha_3^4|V|^2$ respected additive quadruples in $Y$. Note that we defined $Y'$ in the same way as $Y''$, but with a different requirement on the number of respected additive quadruples. The reason for defining $Y''$ comes from Claim~\ref{commonglobalextension}, as we shall see.\\
\indent By Claim~\ref{manyrespclaim2}, $|Y \setminus Y''| \leq \frac{\alpha_3}{1000}|V|$. Hence, there are at most $\frac{\alpha_3}{2}|V|^{35}$ additive 36-tuples in $Y$ containing an element outside $Y''$. Thus, for each additive 12-tuple $q$ in $Y'$, we have $\frac{\alpha_3}{2}|V|^{35}$ additive 36-tuples in $Y''$. By Claim~\ref{commonglobalextension}, as long as $R \geq \blc\Big(r  + \log c^{-1} + 1\Big)$, we obtain a further subcollection $\mathcal{Z}_q$ of additive 36-tuples of size $\alpha_4 |V|^{35}$, $\alpha_4 \geq (c/2)^{O(1)}$, and a global linear map $\psi_q : V \to H$ such that all $x_{[36]} \in \mathcal{Z}_q$ satisfy
\[(\sum_{i \in [36]} (-1)^i \phi_{x_i})|_{\cap_{i \in [36]} U_{x_i}} = \psi_q|_{\cap_{i \in [36]} U_{x_i}}\]
and
\[\sum_{i \in [36]} (-1)^i \phi_{x_i} = \sum_{i \in [12]} (-1)^i\phi_{y_i}\]
on the subspace $(\cap_{i \in [36]} U_{x_i}) \cap (\cap_{i \in [12]} U_{y_i})$, where $q = (y_1, \dots, y_{12})$.\\

Let $q_1, \dots, q_m$ be a maximal collection of additive 12-tuples in $Y'$ such that $|\mathcal{Z}_{q_i} \cap \mathcal{Z}_{q_j}| \leq \frac{\alpha_4^2}{2}|V|^{35}$ for any $i \not= j$. Elementary double-counting argument gives
\[m |V|^{35} + \frac{\alpha_4^2}{2}|V|^{35} m^2 \geq \sum_{i, j \leq m} |\mathcal{Z}_{q_i} \cap \mathcal{Z}_{q_j}| \geq \frac{1}{|V|^{35}} \Big(\sum_{i \in [m]} |\mathcal{Z}_{q_i}|\Big)^2 \geq m^2 \alpha_4^2 |V|^{35},\]
showing that $m \leq 2 \alpha_4^{-2}$.\\

For each $q_i$, define $\psi_i$ as the global map assigned to $\mathcal{Z}_q$, that is $\psi_{q_i}$. We claim that these maps have the property that, for each additive 12-tuple $y_{{12}}$ (i.e., $\sum_{i \in [12]} (-1)^i y_i = 0$) in $Y'$ there exists index $j \in [m]$ such that 
\[\sum_{i \in [12]} (-1)^i \phi_{y_i} = \psi_j\]
holds on $(\cap_{i \in [12]} U_{y_i})$.\\

To that end, let $q$, given by $\sum_{i \in [12]} (-1)^i y_i = 0$, be any additive 12-tuple in $Y'$. By the choice of $q_1, \dots, q_m$, there exists $q_j = (x_{[12]})$ such that $|\mathcal{Z}_{q} \cap \mathcal{Z}_{q_j}| \geq \frac{\alpha_4^2}{2}|V|^{35}$. If $z_{[36]} \in \mathcal{Z}_{q}$, then we have 
\[\sum_{i \in [12]} (-1)^i \phi_{y_i} = \sum_{i \in [36]} (-1)^i \phi_{z_i}\]
holds on $(\cap_{i \in [12]} U_{y_i}) \cap (\cap_{i \in [35]} U_{z_i})$. If additionally $z_{[36]} \in \mathcal{Z}_{q_j}$, then 
\[\sum_{i \in [36]} (-1)^i \phi_{z_i} = \psi_j\]
on $\cap_{i \in [35]} U_{z_i}$. Hence, we find $\frac{\alpha_4^2}{2}|V|^{35}$ 35-tuples $z_{[35]}$ in $t + V$ such that
\[\sum_{i \in [12]} (-1)^i \phi_{y_i} = \psi_j\]
holds on $(\cap_{i \in [12]} U_{y_i}) \cap (\cap_{i \in [35]} U_{z_i})$. By taking pairs of such 35-tuples, there are at least $\frac{\alpha_4^4}{4}|V|^{70}$ of 70-tuples $(z_{[35]}, w_{[35]}) \in (t + V)^{70}$ such that the equality above holds on
\[\Big((\cap_{i \in [12]} U_{y_i}) \cap (\cap_{i \in [35]} U_{z_i})\Big) + \Big((\cap_{i \in [12]} U_{y_i}) \cap (\cap_{i \in [35]} U_{w_i})\Big).\]
By Lemma~\ref{subspsumsreg}, provided $R \geq \blc (r + \log c^{-1} + 1)$, we have 
\[\sum_{i \in [12]} (-1)^i \phi_{y_i}  = \psi_j,\]
on $(\cap_{i \in [12]} U_{y_i})$, as claimed.\\

Finally, we apply dependent random choice to find a subset where all additive 12-tuples are respected. Firstly, we shall find elements  $a_1, \dots, a_\ell \in H, t_1, \dots, t_\ell \in V$, for some $\ell \leq 7 \log_2 m$, such that every non-zero $\psi_j$ has $a_i \cdot \psi_j(t_i) \not=0$ for some $i \in [\ell]$. This will be done iteratively; we shall add $a_i, t_i$ to the lists above in $i$\tss{th} step. Additionally, at each step, we pass to $Y'_i = Y' \cap B_{\bcdot t_1} \cap \dots B_{\bcdot t_i}$, and ensure the bound $|Y'_i| \geq  2^{-i-1}p^{-ir} \alpha_1|V|$ holds along the way. Initially, we set $Y'_0 = Y'$.\\
\indent Suppose that we have carried out the first $i$ steps. Let $J \subseteq [m]$ be the collection of indices $j$ such that the map $\psi_j$ is non-zero and $\psi_j$ still does not have a desired pair of elements among $(a_1, t_1), \dots, (a_i, t_i)$. Let $a \in H$ and $t\in V$ be chosen uniformly and independently at random, and let $N$ be the random variable that counts those $j \in J$ with $a \cdot \psi_j(t) \not= 0$. Observe that 
\begin{equation}\label{nexpquarteqn}\exx N \geq \frac{1}{4} |J|.\end{equation}
In particular, $|J| \mathbb{P}(N \geq |J| / 8) + |J| / 8 \geq \ex N$, so we have $\mathbb{P}(N \geq |J| / 8) \geq \frac{1}{8}$. This also shows that, provided this event occurs at each step, the procedure will terminate after at most $\ell \leq 7\log_2 m$ steps.\\
\indent Note that it is crucial that the domain of $\psi_j$ is the whole space, as otherwise we would not have got a bound of $N \geq |J| / 8$, which led us to logarithmic bound in the number of steps. Instead, we would have had a much slower rate of removal of indices $j$ and the number of steps would have suffered an exponential loss.\\
 \indent On the other hand, due to Lemma~\ref{randomIntersection} and the fact that $|Y'_i| \geq \alpha_1 2^{-\ell - 1}p^{-r\ell}$ at each step and that $\ell \leq O(\log m)$, we have
 \[\mathbb{P}\Big(|Y'_i \cap B_{\bcdot t_{i+1}}| \geq \frac{1}{2}p^{-r} |Y'_i|\Big) \geq \frac{15}{16},\]
 provided $R \geq \blc (r + 1)(\log c^{-1} + 1)$. Hence, there exists a choice of $a_{i+1} = a$ and $t_{i+1} = t$ which makes $|J|$ decrease by a factor of $7/8$ and ensures $|Y'_{i+1}| \geq  2^{-(i+1)-1}p^{-(i+1)r} \alpha_1|V|$.\\
\indent Lastly, take $\lambda_1, \dots, \lambda_\ell \in \mathbb{F}_p$ uniformly and independently at random and set $Z = \{z \in Y'_\ell : (\forall i \in [\ell]) a_i \cdot \phi_z(t_i) = \lambda_i\}$. It is important to stress that, whenever $z \in Y'_\ell$, due to our choice of $Y'_\ell$, all elements $t_1, \dots, t_\ell$ belong to the domain $U_z = B_{z \bcdot}$ of $\psi_z$, so the definition of $Z$ makes sense.\\
\indent By linearity of expectation, there is a choice of $\lambda_1, \dots, \lambda_\ell$ such that $|Z| \geq p^{-\ell}|Y'_\ell| \geq (c/2)^{O(1)}p^{-r \ell}|V|$.\\
\indent It remains to show that all additive 12-tuples in $Z$ are respected. Take any $\sum_{i \in [12]} (-1)^i z_i = 0$ in $Z$. There exists $j \in [m]$ such that
\[\sum_{i \in [12]} (-1)^i \phi_{z_i}  = \psi_j\]
on $\cap_{i \in [12]} U_{z_i}$. If the map $\psi_j$ is non-zero, we have $i$ such that $a_i \cdot \psi_j(t_i) \not=0$. By the way we defined $Y'_\ell$, we have that $(z_{i'}, t_i) \in B$ holds for all $i' \in [12]$, so $t_i \in U_{z_1} \cap U_{z_2} \cap \dots \cap U_{z_{12}}$ and, by the definition of $Z$, $a_i \cdot \phi_{z_{i'}}(t_i) = \lambda_i$ for $i' \in [12]$. But then 
\[0  = a_i \cdot \phi_{z_1}(t_i) - a_i \cdot \phi_{z_2}(t_i) + \dots  - a_i \cdot \phi_{z_{12}}(t_i) = a_i \cdot \psi_j(t_i) \not= 0\]
which is a contradiction. To finish the proof, it remains to choose $R$ larger than $K (r + 1)(\log c^{-1} + 1)$ for some sufficiently large absolute constant $K$, so that all required bounds involving $R$ hold.\end{proof}

\section{Extending bihomomorphisms from a dense set of columns}

\hspace{17pt}In this step of the proof, we take a different perspective and put the linear maps in the given system together to obtain a function of two variables. Since all 12-tuples are (subspace-)respected, we obtain a Freiman 12-bihomomorphism. Moreover the domain is an intersection of a bilinear variety with a set of the form $X \times V$. A combination of robust Bogolyubov-Ruzsa theorem and algebraic regularity method allows us enlarge the domain to a full bilinear variety. It should also be noted that the previous application of algebraic regularity method a similar context in~\cite{U4paper} was in the so-called $99\%$ regime, while our situation can be seen as $1\%$ regime.

\begin{proposition}\label{dcbihomext}
    Suppose that $X \subseteq G$ has density $c$ and let $V \leq G$ be a subspace. Let $B$ be the bilinear variety $\{(x,y) \in G \times V: \beta(x,y) = 0\}$ defined by a map $\beta : G \times G \to \mathbb{F}_p^r$ and let $\phi : (X \times V) \cap B \to H$ be a map which respects all horizontal additive 12-tuples and is linear in the vertical direction.\\
    \indent Then there exist a subspace $U \leq V$ of codimension $O(\log^{O(1)} c^{-1} r^{O(1)})$, a map $\psi : (U \times U) \cap B \to H$, linear in both directions such that for each $(x,y) \in (U \times U) \cap B$ the identity
     \begin{align*}\psi(x,y) = &\phi(x_1, y - v) + \phi(x_2, y - v) - \phi(x_3, y - v) - \phi(x_1 +x_2 - x_3 - x, y - v)\\
    &\hspace{2cm} + \phi(x'_1, v) + \psi(x'_2, v) - \phi(x'_3, v) -\phi(x'_1 +x'_2 - x'_3 - x,  v)\end{align*}
    holds for at least $\exp(-O(\log^{O(1)} c^{-1} r^{O(1)}))|V|^7$ choices of $(v, x_1, x_2, x_3, x'_1, x'_2, x'_3) \in G^7$.
\end{proposition}

\begin{proof}
    Similarly to the proof of Proposition~\ref{Fbihomstep}, we begin the proof by applying the algebraic regularity lemma. This is, strictly logically, redundant, as we could have kept track of the bilinear map and its rank produced by the previous application of the regularity lemma. However, this would introduce more notational burden in the statements of Proposition~\ref{Fbihomstep} and this proposition, so we opted to use the regularity lemma another time.\\
    
    \indent Let $R > 0$ be a parameter to be chosen later (we will end up choosing $R$ larger than $K (r + \log c^{-1} + 1)$ for some sufficiently large absolute constant $K$). Apply the algebraic regularity lemma (Theorem~\ref{arl}) to the map $\beta|_{V \times V}$. We obtain a bilinear map $\gamma : V \times V \to \mathbb{F}_p^{r'}$, where $r' \leq r$, and a subspace $V' \leq V$ of codimension at most $2rR$ such that $\gamma$ has rank at least $R$ on $V' \times V'$ and property~\eqref{allcosetseq} holds for the bilinear variety $B = \{\beta = 0\}$. Note also that $B \cap (a + V') \times V' \not= \emptyset$ for all $a$, as it contains $(a,0)$.\\
    \indent We misuse the notation, and write $V$ instead of $V'$, $\beta$ instead $\gamma$ and $r$ instead of $r'$. Also, by pigeonhole principle, there exists a coset $a + V$ such that $|(a +V) \cap X | \geq c |V|$. We misuse the notation further, by writing $X$ instead of $X \cap (a + V)$. To sum up, we now have that $X \subseteq a +V$ is a set of size at least $c|V|$, $\phi|_{(X \times V) \cap B}$ respects all horizontal 12-tuples and it is linear in the vertical direction, and $\beta|_{V \times V}$ has rank at least $R$.\\

    \indent Using robust Bogolyubov-Ruzsa theorem (Theorem~\ref{rbrlemma}), we have a subspace $U\leq V$ of codimension $O(\log^{O(1)} c^{-1})$ such that every $u \in U$ can be written as $\Omega(c^{O(1)} |G|^3)$ solutions to $u = x_1 + x_2 - x_3 - x_4$ for $x_i \in X$. Let $T_u$ be the set of such $(x_1, x_2, x_3)$ for the given $u \in U$. The rest of the proof will proceed in two steps. We shall first define the map $\psi$ on the vast majority of the variety $(U \times V) \cap B$ and then extend it to the whole of that variety.\\
    
    \indent For every point $(u, y) \in U \times V$ that we shall define the map $\psi$ at in the first step, $\psi(u,y)$ will always be given by $\phi(x_1, y) + \phi(x_2, y) - \phi(x_3, y) - \phi(x_1 + x_2 - x_3 - u, y)$ for some $(x_1, x_2, x_3) \in T_u$. As the map $\phi$ respects all horizontal additive 8-tuples, the value $\psi(u, y)$ is well-defined, irrespective of the choice of triple $(x_1, x_2, x_3)$. Moreover, as the map $\phi$ respects all horizontal additive 12-tuples, we also have that $\psi$, defined in the way above on some domain that we shall specify, respects all horizontal additive triples.\\
    \indent Let $S_1$ be the collection of all $(u, y) \in (U \times V) \cap B$ such that $|T_u \cap (B_{\bcdot y})^3| \geq \frac{1}{2} p^{-3r} |T_u|$. By Lemma~\ref{randomIntersection}, we have $|(S_1)_{u \bcdot}| \geq (1 - c^{-O(1)}p^{O(r) - R/4}) |B_{u \bcdot} \cap V|$ for all $u \in U$. We may consider $S_1$ as the current domain of $\psi$.\\
    \indent Fix now an arbitrary $u \in U$. We claim that $\psi$ respects a vast majority of vertical additive triples in $B_{u \bcdot} \cap V$. To that end, take any $y_1, y_2, y_3 \in B_{u \bcdot} \cap V$ that satisfy $y_1 + y_2 = y_3$. If there exists a triple $(x_1, x_2, x_3) \in T_u \cap (B_{\bcdot y_1})^3 \cap (B_{\bcdot y_2})^3$, and we have $(u, y_1), (u, y_2), (u, y_3) \in S_1$, then $x_1, x_2, x_3 \in B_{\bcdot y_3}$ and 
    \begin{align*}\psi(u, y_3) = &\phi(x_1, y_3) + \phi(x_2, y_3) - \phi(x_3, y_3) - \phi(x_1 + x_2 - x_3 - u, y_3)\\
    =&\phi(x_1, y_1) + \phi(x_2, y_1) - \phi(x_3, y_1) - \phi(x_1 + x_2 - x_3 - u, y_1)\\
    &\hspace{2cm}+\phi(x_1, y_2) + \phi(x_2, y_2) - \phi(x_3, y_2) - \phi(x_1 + x_2 - x_3 - u, y_2)\\
    = & \psi(u, y_1)  + \psi(u, y_2) ,\end{align*}
    so the vertical additive triple is respected.\\
    \indent By Lemma~\ref{randomIntersection}, we have that $|T_u \cap (B_{\bcdot y_1})^3 \cap (B_{\bcdot y_2})^3| \geq \frac{1}{2}p^{-6r}|T_u|$ for all but at most $c^{-O(1)}p^{O(r) - R/4}|V|^2$ choices of $(y_1, y_2) \in V$. Hence, all but at most $c^{-O(1)}p^{O(r) - R/4}|V|^2$ additive triples are respected. By Lemma~\ref{3approxHom}, there exists a linear map $\theta_u : B_{u \bcdot} \cap V \to H$ which coincides with $\psi(u, \cdot)$ on a subset of $B_{u \bcdot } \cap V$ of size at least $(1 - c^{-O(1)}p^{O(r) - R/8})|B_{u \bcdot } \cap V|$. Let $S_2$ be the set of all $(u,y) \in (U \times V) \cap B$ where such maps coincide. Hence, $|(S_2)_{u \bcdot}| \geq (1 - c^{-O(1)}p^{O(r) - R/8}) |V \cap B_{u \bcdot}|$ for all $u \in U$ and $\psi$ respects all additive triples in both principal directions on $S_2$.\\

    \indent Secondly, we extend $\psi$ easily to the whole variety by the formula $\psi'(u, y) = \psi(u, y') + \psi(u, y- y')$ for any $y'$ such that $(u, y'), (u, y-y') \in S_2$. Trivially, $\psi'$ respects all vertical additive triples (as $\psi(u, \cdot)$ is a restriction of a linear map for each $u \in U$) and it remains to check the same for the horizontal direction. Let $u_1 + u_2 = u_3$ hold for some elements of $U$ and let $y \in V \cap B_{u_1 \bcdot} \cap B_{u_2 \bcdot}$. Then
    \[|(V \cap B_{u_1 \bcdot}) \setminus (S_2)_{u_1 \bcdot}| + |(V \cap B_{u_2 \bcdot}) \setminus (S_2)_{u_2 \bcdot}| + |(V \cap B_{u_3 \bcdot}) \setminus (S_2)_{u_3 \bcdot}| \leq c^{-O(1)}p^{O(r) - R/8} |V|,\]
    so, provided $R \geq \blc\Big(r + 1 + \log c^{-1}\Big)$, we can find $y' \in V$ such that $(u_i, y'), (u_i, y - y') \in S_2$ for all $i \in [3]$. Thus, the horizontal additive triple that we considered is respected.\\

    \indent Finally, we show that so defined $\psi'$ is related to the original map $\phi$. Let $(u, y) \in U \times V \cap B$ be any point. By the properties of $S_2$, we have at least $\frac{1}{2}p^{-r}|V|$ elements $v \in V$ such that $(u, v), (u, y - v) \in S_2$. Since $S_2 \subseteq S_1$, we may also find at least $c^{O(1)}p^{-O(r)}|V|^6$ elements $(x_1, x_2, x_3, x'_1, x'_2, x'_3)$  such that  
    \begin{align*}\psi'(u,y) = &\psi(u, y - v) + \psi(u, v) \\
    = &\phi(x_1, y - v) + \phi(x_2, y - v) - \phi(x_3, y - v) - \phi(x_1 +x_2 - x_3 - u, y - v)\\
    &\hspace{2cm} + \phi(x'_1, v) + \psi(x'_2, v) - \phi(x'_3, v) -\phi(x'_1 +x'_2 - x'_3 - u,  v).\qedhere\end{align*}
    \end{proof}

\section{Structure theorem for Freiman bihomomorphisms}

We begin by recalling the fact that bilinear maps on low-codimensional bilinear varieties coincide with global bilinear maps on sets of large density. In fact, the maps coincide on a further bilinear variety of low codimension, but we do not need this additional property.

\begin{theorem}[Gowers and Mili\'cevi\'c~\cite{U4paper}]\label{bilextnthm}
    Let $G_1, G_2, H$ be finite-dimensional vector spaces over $\mathbb{F}_p$. Let $B = \{(x,y) \in G_1 \times G_2: \beta(x,y) = 0\}$ be a bilinear variety defined by a bilinear map $\beta : G_1 \times G_2 \to \mathbb{F}_p^r$. Let $\phi:B \to H$ be a bilinear map. Then there exists a global bilinear map $\Phi : G_1 \times G_2 \to H$ such that $\Phi(x,y) = \phi(x,y)$ holds for at least $p^{-r^{O(1)}}|G_1||G_2|$ points $(x,y) \in B$.
\end{theorem}

The rest of the section is devoted to a proof of Theorem~\ref{fbihommain}, which is just the matter of putting together all the steps carried out in previous sections.

\begin{proof}[Proof of Theorem~\ref{fbihommain}]
    Let $\varphi : A \to H$ be the given Freiman bihomomorphism. Applying Proposition~\ref{passingtosystemstep}, we find a set $X \subseteq G$, a collection of linear maps $\phi_x : G \to H$ and subspaces $U_x \leq G$ for each $x \in X$, and positive quantity $c' \geq \exp(-\log^{O(1)}c^{-1})$ and integer $r \leq \log^{O(1)}c^{-1}$, that satisfy the three properties in the conclusion of the proposition. In particular, elements $t_x \in G$ such that $\varphi(x, y) = \phi_x(y)  + t_x$ holds for elements $y$ in a subset $A'_{x} \subseteq A_{x \bcdot}$ of size $|A'_x| \geq (c'/2)^{O(1)}|G|$.\\
    \indent The rest of the proof will be exclusively about various systems of (partially defined) linear maps. However, let us first check that understanding the structure of the system $(\phi_x)_{x \in X}$ suffices to complete the proof.

    \begin{claim}\label{linsystovarphi}
        Suppose that $\psi : G \to \on{Hom}(G, H)$ is an affine map from $G$ to the vector space of linear maps from $G$ to $H$. Suppose that $\on{rank}(\phi_x - \psi(x)) \leq s$ holds for at least $\delta|G|$ elements $x \in X$. Then there exists a global biaffine map $\Phi : G \times G \to H$ such that $\Phi(x,y) = \varphi(x,y)$ holds for at least $(c' \delta p^{-s} / 2)^{O(1)}|G|^2$ points $(x,y) \in A$.
    \end{claim}

    \begin{proof} For each $x \in X$ such that $\on{rank}(\phi_x - \psi(x)) \leq s$, we have a subspace $S_x \leq G$ of codimension $s$ such that $\phi_x(y) = \psi(x)(y)$ holds for $y \in S_x$. By averaging, there exists a coset $s_x + S_x$ such that $|(s_x + S_x) \cap A'_x| \geq (c'/2)^{O(1)} p^{-s} |G|$. For any two $y, y' \in (s_x + S_x) \cap A'_x$ we have $y - y' \in S_x$ and therefore 
    \[\psi(x)(y - y') = \phi_x(y - y') = \phi_x(y) - \phi_x(y') = \varphi(x, y) - \varphi(x, y').\]
    There are at least $(c'/2)^{O(1)} \delta p^{-2s} |G|^3$ triples $(x, y, y')$ that satisfy the equation above. Theorem~\ref{fbihommain} follows after averaging over $y'$ and using Theorem~\ref{invhomm} for the map $z \mapsto \varphi(z, y')$. We may take $\Phi(x, y) = \psi(x)(y - y') + \theta(x)$ for an affine map $\theta$ that stems from the application of Theorem~\ref{invhomm}.\end{proof}

    \indent Let us now return to the study of the system $(\phi_x)_{x \in X}$. The final property in the conclusion of Proposition~\ref{passingtosystemstep} guarantees that
    $\on{rank} \Big(\sum_{i \in [4]} (-1)^i \phi_{x_i}\Big) \leq r$ holds for  $c'|G|^3$ of additive quadruples $(x_1, x_2, x_3, x_4)$ in $X$, where $r \leq \log^{O(1)}c^{-1}$.\\

    Apply Proposition~\ref{fullspacesystem} to find a subspace $U^{(1)}$ of codimension $O(\log^{O(1)} c^{-1})$ and a collection of linear maps $\phi^{(1)}_x$, $x \in U^{(1)}$, such that 
    \[\on{rank}\Big(\sum_{i \in [4]} (-1)^i \phi^{(1)}_{x_i}\Big) \leq O(\log^{O(1)} c^{-1})\]
    holds for all additive quadruples $x_{[4]}$ in $U^{(1)}$, and for each $u \in U^{(1)}$, there are at least $\exp(-\log^{O(1)} c^{-1}) |G|^3$ choices of $(x_1, x_2, x_3) \in X^3$ such that $x_1 + x_2 - x_3 - u \in X$ and
    \begin{equation}\on{rank} \Big(\phi^{(1)}_u - \phi_{x_1} + \phi_{x_2} - \phi_{x_3} + \phi_{x_1 + x_2 - x_3 - u}\Big) \leq O(\log^{O(1)} c^{-1}).\label{phiphi1}\end{equation}

    Apply Proposition~\ref{subspacePass} with $k = 4$ and $\varepsilon = 2^{-24}$. Suppose first that the case \textbf{(i)} of the proposition occurs. Thus, we find a set $X_1 \subseteq U^{(1)}$ of size $|X_1| \geq \exp(-\log^{O(1)} c^{-1})|G|$ and a linear map $\psi : G \to H$ such that $\on{rank}(\psi - \phi^{(1)}_x) \leq O(\log^{O(1)} c^{-1})$ holds for all $x \in X_1$. From~\eqref{phiphi1}, we have $\exp(-\log^{O(1)} c^{-1}) |G|^4$ choices of $(x_1, x_2, x_3, u) \in G^4$ such that
    \[\on{rank} \Big(\psi - \phi_{x_1} + \phi_{x_2} - \phi_{x_3} + \phi_{x_1 + x_2 - x_3 - u}\Big) \leq O(\log^{O(1)} c^{-1}).\]
    By a change of variables and averaging, there is $\psi' = \psi + \phi_{x_2} - \phi_{x_3} + \phi_{x_4}$ such that $\on{rank}(\phi_x - \psi) \leq O(\log^{O(1)} c^{-1})$ holds for at least  $\exp(-\log^{O(1)} c^{-1}) |G|$ elements $x \in G$. Apply Claim~\ref{linsystovarphi} to finish the proof in this case.\\
    \indent Now suppose that the possibility \textbf{(ii)} occurs. This time we find subspaces $U^{(2)}_x$ for $x \in U^{(1)}$, of codimension $O(\log^{O(1)} c^{-1})$ such that, for each for $\ell \in \{2,3,4\}$, $1-\varepsilon$ proportion of all additive $(2\ell)$-tuples $x_{[2\ell]}$ in $U^{(1)}$ satisfy
        \[\sum_{i \in [2\ell]} (-1)^i\phi^{(1)}_{x_i} = 0\]
    on $\cap_{i \in [2\ell]} U^{(2)}_{x_i}$. We may also assume that $U^{(2)}_x \leq U^{(1)}$.\\

    We may now apply Proposition~\ref{bilvardomain} (with $U^{(1)}$ instead of $G$) which gives us $d' \leq O(\log^{O(1)} c^{-1})$, affine maps $\theta_1, \dots, \theta_{d'} : U^{(1)} \to U^{(1)}$, a set $X_2 \subseteq U^{(1)}$, linear maps $\phi^{(2)}_x : U^{(1)} \to H$ for $x \in U^{(1)}$ such that, defining subspaces $B_x = \langle \theta_1(x), \dots, \theta_{d'}(x)\rangle^\perp \leq U^{(1)}$,
    \begin{itemize}
        \item for each $a \in X_2$, bound 
        \begin{equation}\on{rank}(\phi^{(2)}_a + \phi^{(1)}_x - \phi^{(1)}_{x + a}) \leq d'\label{phi1phi2}\end{equation}
        holds for at least $\frac{1}{2}|U^{(1)}|$ choices of $x \in U^{(1)}$,
        \item we have
        \[\Big(\sum_{i \in [4]} (-1)^i\phi^{(2)}_{x_i}\Big)|_{\cap_{i \in [4]} B_{x_i}} = 0\]
        for at least $p^{-O(\log^{O(1)} c^{-1})} |U^{(1)}|^3$ of additive quadruples $x_{[4]}$ in $X_2$.
    \end{itemize}

    Next, Proposition~\ref{Fbihomstep} allows us to find a subset $X_3 \subseteq X_2$ of size $|X_3| \geq \exp\Big( - O\Big(\log^{O(1)} c^{-1}\Big)\Big)|U^{(1)}|$ and a subspace $U^{(3)}$ of codimension $O\Big(\log^{O(1)} c^{-1}\Big)$ such that all additive 12-tuples in $X_3$ are respected by $(\phi^{(2)}_x, B_x \cap U^{(3)})$ system.\\

    Write $\theta_i(x) = \theta_i^{\on{lin}}(x) + \theta_i(0)$ and let $W = \langle \theta_1(0), \dots, \theta_{d'}(0) \rangle^\perp$. Let us now change the viewpoint slightly and consider map $\phi^{(3)} : (X_3 \times (U^{(3)} \cap W)) \cap B \to H$, given by $\phi^{(3)}(x,y) = \phi^{(2)}_x(y)$, where $B = \{(x,y) \in U^{(1)} \times U^{(1)} : (\forall i \in [d'])\,\theta^{\on{lin}}_i(x) \cdot y = 0\}$. As each $\phi^{(2)}_x$ is linear, it follows that $\phi^{(3)}$ is linear in the vertical direction. Also, $\phi^{(3)}$ respects all horizontal additive 12-tuples. Namely, if $(x_i, y) \in (X_3 \times (U^{(3)} \cap W)) \cap B$, $i \in [12]$, satisfy  $\sum_{i \in [12]} (-1)^i x_i = 0$, then $y \in \cap_{i \in [12]} B_{x_i \bcdot}$ so we have
    \[\sum_{i \in [12]} (-1)^i \phi^{(3)}(x_i, y) = \sum_{i \in [12]} (-1)^i \phi^{(2)}_{x_i}( y) = 0,\]
    as desired.\\
    \indent Apply Proposition~\ref{dcbihomext} to find a subspace $U^{(4)} \leq U^{(3)}$ of codimension $O(\log^{O(1)} c^{-1})$, a map $\phi^{(4)}: (U^{(4)} \times U^{(4)}) \cap B \to H$, linear in both directions, such that for each $(x,y) \in (U^{(4)} \times U^{(4)}) \cap B$ the identity
     \begin{align}\phi^{(4)}(x,y) = &\phi^{(3)}(x_1, y - v) + \phi^{(3)}(x_2, y - v) - \phi^{(3)}(x_3, y - v) - \phi^{(3)}(x_1 +x_2 - x_3 - x, y - v)\nonumber\\
    &\hspace{2cm} + \phi^{(3)}(x'_1, v) + \psi^{(3)}(x'_2, v) - \phi^{(3)}(x'_3, v) -\phi^{(3)}(x'_1 +x'_2 - x'_3 - x,  v)\label{phi4eqn}\end{align}
    holds for at least $\exp(-O(\log^{O(1)} c^{-1}))|U^{(4)}|^7$ choices of $(v, x_1, x_2, x_3, x'_1, x'_2, x'_3) \in G^7$.\\

    Finally, apply Theorem~\ref{bilextnthm} to the map $\phi^{(4)}$ to find a global bilinear map $\Phi : G \times G \to H$, which coincides with $\phi^{(4)}$ at $\exp(-O(\log^{O(1)} c^{-1})|G|^2$ points inside $(U^{(4)} \times U^{(4)}) \cap B$. It remains to relate $\Phi$ with the initial system $(\phi_x)_{x \in X}$ and thus, by Claim~\ref{linsystovarphi}, with $\varphi$.\\
    \indent By~\eqref{phi4eqn} and a change of variables, we have 
    \begin{align*}\Phi(x_1 + x_2 - x_3 - x_4, y + y') = &\phi^{(2)}_{x_1}(y) + \phi^{(2)}_{x_2}(y) - \phi^{(2)}_{x_3}(y) - \phi^{(2)}_{x_4}(y)\nonumber\\
    &\hspace{2cm} + \phi^{(2)}_{x_1 + a_1}(y') + \psi^{(2)}_{x_2 + a_2}(y') - \phi^{(2)}_{x_3 + a_3}(y') -\phi^{(3)}_{x_4 + a_1 + a_2 - a_3}(y')\end{align*}
    for  $\exp(-O(\log^{O(1)} c^{-1}))|G|^9$ choices of $(x_1, x_2, x_3, x_4, a_1, a_2, a_3, y, y')$ (and all arguments belong to the relevant domains). By averaging over $x_2, x_3, x_4, a_1, a_2, a_3$ and $y'$, there exist a linear map $\alpha : G \to H$, a map $f : G \to H$ (with no assumed structure) and elements $a,b \in G$ such that 
    \[\Phi(x + a, y + b) = \phi^{(2)}_x(y) + \alpha(y) + f(x)\]
    holds for at least $\exp(-O(\log^{O(1)} c^{-1}))|G|^2$ pairs $(x,y) \in G^2$. By Cauchy-Schwarz inequality, we have $\exp(-O(\log^{O(1)} c^{-1}))|G|^3$ triples $(x,y, y')$ such that the equality above holds for $(x,y)$ and $(x,y')$, so, using the fact that $\phi^{(2)}_x$ is a linear map on $G$,
    \begin{align*}\Phi(x +a, y - y') = \Phi(x +a, y + b) - \Phi(x +a, y' + b) = &\Big(\phi^{(2)}_x(y) + \alpha(y) + f(x)\Big) - \Big(\phi^{(2)}_x(y') + \alpha(y') + f(x)\Big)\\
    = &\phi^{(2)}_x(y - y') + \alpha(y - y').\end{align*}
    Hence, by averaging, $\on{rank}(\phi^{(2)}_x - \Phi(x + a, \bcdot) - \alpha) \leq -O(\log^{O(1)} c^{-1})$ holds for at least $\exp(-O(\log^{O(1)} c^{-1}))|G|$ elements $x \in G$. Going back to~\eqref{phiphi1} and~\eqref{phi1phi2}, we conclude that $\on{rank}(\phi_x - \Phi(x + a', \bcdot) - \alpha') \leq O(\log^{O(1)} c^{-1})$ holds for at least $\exp(-O(\log^{O(1)} c^{-1}))|G|$ elements $x \in X$, for some element $a' \in G$ and some linear map $\alpha' : G \to H$. We are done by Claim~\ref{linsystovarphi}.    
\end{proof}

\section{Inverse theorem for $\mathsf{U}^4$ norm}

In this short final section, we deduce the quasipolynomial inverse theorem for $\mathsf{U}^4(\mathbb{F}_p^n)$ norm. The proof follows start from the assumption that $\|f\|_{\mathsf{U}^4} \geq \delta$, which immediately gives
\[\delta^{O(1)} \leq \|f\|_{\mathsf{U}^4}^{16} = \exx_{x, a, b, c, d} \partial_a \partial_b \partial_c \partial_d f(x) = \exx_{a, b} \|\partial_a \partial_b f\|_{\mathsf{U}^2}^4.\]

The structure theorem for Freiman bihomomorphisms (Theorem~\ref{fbihommain}) quickly gives a trilinear form $\alpha : G \times G \times G \to \mathbb{F}_p$ such that
\[\Big|\exx_{x, a, b, c} \partial_a \partial_b \partial_c f(x) \omega^{\alpha(a, b, c)}\Big| \geq \Omega_\delta(1).\]

The symmetry argument argument of Green and Tao allows us to deduce that $\alpha$ is approximately symmetric in the sense that $\ex_{a, b, c} \omega^{\alpha(a,b,c) - \alpha(b, a, c)} \geq \Omega_\delta(1)$ holds, from which we deduce that $(a,b,c) \mapsto \alpha(a,b,c) - \alpha(b, a, c)$ has small partition rank. Analogous inequalities hold for other transpositions of variables. When $p \geq 5$, it is easy to replace $\alpha$ by a symmetric trilinear form and finish the proof, which was the case in~\cite{U4paper}.\\
\indent However, when $p \leq 3$, Tidor~\cite{Tidor} found an alternative way of relating an approximately symmetric trilinear form to an exactly symmetric one, and extended the quantitative inverse theorem for $\mathsf{U}^4(\mathbb{F}_p^n)$ norm to the low characteristic as well. We sum up his work as the following theorem, and use it as a black box. Let us digress for a moment and mention that, unlike the trilinear case, there exist approximately symmetric forms in 4 variables that are far from exactly symmetric ones~\cite{asfbes}. Nevertheless, this kind of approach can give a quantitative inverse theorem for $\mathsf{U}^5(\mathbb{F}_2^n)$ norm~\cite{LukaU56}. We now recall Tidor's result.

\begin{theorem}[Tidor~\cite{Tidor}]\label{u4nonclconcl}
    Let $f : G \to \mathbb{D}$ be a function and let $\alpha : G \times G \times G \to \mathbb{F}_p$ be a triaffine form. Suppose that 
    \begin{equation}\label{derivativeremoval}
        \Big|\exx_{x, a, b, c} \partial_a \partial_b \partial_c f(x) \omega^{\alpha(a, b, c)}\Big| \geq \eta.
    \end{equation}
    Then there exists a non-classical cubic polynomial $q : G \to \mathbb{T}$ such that
    \[\Big|\exx_x f(x) \on{e}(q(x))\Big| \geq \exp(-\log^{O(1)} \eta^{-1}).\]
\end{theorem}

We are now ready to prove Theorem~\ref{u4invmain}.

\begin{proof}[Proof of Theorem~\ref{u4invmain}]
    In the light of Theorem~\ref{u4nonclconcl}, we only need to obtain inequality of the form~\eqref{derivativeremoval}. Let $f : G \to \mathbb{D}$ be a function with $\|f\|_{\mathsf{U}^4} \geq \delta$. Then
    \[\delta^{O(1)} \leq \|f\|_{\mathsf{U}^4}^{16} = \exx_{x, a, b, c, d} \partial_a \partial_b \partial_c \partial_d f(x) = \exx_{a, b} \|\partial_a \partial_b f\|_{\mathsf{U}^2}^4.\]
    By averaging, and the inverse theorem for $\mathsf{U}^2$ norm, we have a set $A \subseteq G^2$ of size $(\delta/2)^{O(1)}|G|^2$ such that for each $(a,b) \in A$ there exists $\varphi(a,b) \in G$ with 
    \[\Big|\exx_{x} \partial_a \partial_b f(x) \omega^{\varphi(a,b) \cdot x}\Big| \geq (\delta/2)^{O(1)}.\]
    We show that $\varphi$ is a Freiman bihomomorphism on a large set. We have
    \begin{align*}(\delta/2)^{O(1)} \leq &\exx_{a, b} \id_{A}(a,b) \Big|\exx_{x} \partial_a \partial_b f(x) \omega^{\varphi(a,b) \cdot x}\Big|^2 \\
    = &\exx_{x, a, b, c} \id_{A}(a,b) \partial_a \partial_b \partial_c f(x) \omega^{\varphi(a,b) \cdot c}.\end{align*}
    By Cauchy-Schwarz inequality, we have
    \begin{align*}(\delta/2)^{O(1)} \leq &\Big|\exx_{x, a, b, c} \id_{A}(a,b)\, \partial_b \partial_c f(x + a)\, \overline{\partial_b \partial_c f(x)}\, \omega^{\varphi(a,b) \cdot c}\Big|^2\\
    \leq & \exx_{x, b, c} \Big| \overline{\partial_b \partial_c f(x)}\, \exx_{ a} \id_{A}(a,b)\, \partial_b \partial_c f(x + a) \, \omega^{\varphi(a,b) \cdot c}\Big|^2\\
    \leq & \exx_{x, a, b, c, u} \id_{A}(a,b)\, \id_{A}(a + u, b) \,\partial_b \partial_c f(x + a + u) \,\overline{\partial_b \partial_c f(x + a)}\,  \omega^{(\varphi(a + u,b)-\varphi(a,b))  \cdot c}.\end{align*}
    Make a change of variables $x \mapsto x - a$ and apply Cauchy-Schwarz inequality one more time to get
    \begin{align*}(\delta/2)^{O(1)} \leq &\Big|\exx_{x, a, b, c, u} \id_{A}(a,b)\, \id_{A}(a + u, b) \,\partial_b \partial_c f(x + u) \,\overline{\partial_b \partial_c f(x)}\,  \omega^{(\varphi(a + u,b)-\varphi(a,b))  \cdot c}\Big|^2\\
    \leq & \exx_{x, b, c, u} \Big| \partial_b \partial_c f(x + u)\, \overline{\partial_b \partial_c f(x)} \,\exx_{ a} \id_{A}(a,b) \,\id_{A}(a + u,b)\, \omega^{(\varphi(a + u,b)-\varphi(a,b)) \cdot c}\Big|^2\\
    \leq & \exx_{x, a, b, c, u, v} \id_{A}(a,b)\, \id_{A}(a + u,b)\, \id_{A}(a+ v,b)\, \id_{A}(a + u +v,b)\, \omega^{(\varphi(a + u+v,b)-\varphi(a+u,b)-\varphi(a+v,b) + \varphi(a,b)) \cdot c}.\end{align*}
    By averaging, there are $(\delta/2)^{O(1)}|G|$ rows on which $\varphi$ respects $(\delta/2)^{O(1)}|G|^3$ horizontal additive quadruples. Apply Theorem~\ref{invhomm} to these rows to find a subset of $A$ of size $(\delta/2)^{O(1)}|G|^2$ on which $\varphi$ respects all horizontal additive quadruples. Applying the same argument in the vertical direction allows us to assume that $\varphi$ respects all horizontal and vertical additive quadruples, i.e., that it is a Freiman bihomomorphism.\\
    
    Apply Theorem~\ref{fbihommain} to find a global biaffine map $\Phi : G \to G$ which coincides with $\varphi$ on a subset $A' \subseteq A$ of size $\exp(-\log^{O(1)}  \delta^{-1})|G|^2$. Thus
    \begin{align*}\exp(-\log^{O(1)} \delta^{-1})\leq &\exx_{a,b} \id_{A'}(a,b)\Big|\exx_{x} \partial_a \partial_b f(x) \omega^{\varphi(a,b) \cdot x}\Big|^2\\
    = &\exx_{a,b} \id_{A'}(a,b)\Big|\exx_{x} \partial_a \partial_b f(x) \omega^{\Phi(a,b) \cdot x}\Big|^2\\
    \leq & \exx_{x, a, b, c} \partial_a \partial_b \partial_c f(x) \omega^{\Phi(a,b) \cdot c}.\end{align*}

    Considering the triaffine map $(a,b,c) \mapsto \Phi(a,b) \cdot c$, we get inequality of the form~\eqref{derivativeremoval}. Apply Theorem~\ref{u4nonclconcl} to finish the proof.
\end{proof}

\end{oldthebibliography}

\end{document}